\documentclass{article}
\usepackage{fullpage}
\usepackage{amsmath}
\usepackage{tabularx,array}
\usepackage{amsmath, amsthm, amssymb, amsfonts,bm}
\usepackage{graphicx}
\usepackage{color}
\usepackage{enumitem}
\usepackage{longtable}
\usepackage{environ}
\usepackage[dvipsnames]{xcolor}
\usepackage{booktabs}
\usepackage{ltxtable} 
\usepackage{ltablex}
\usepackage{mathrsfs}
\usepackage{amsfonts}
\usepackage{amsmath}
\usepackage{lineno}
\usepackage[all]{xy}
\usepackage{ tipa }
\usepackage{authblk}
\usepackage[latin1]{inputenc}

\newtheorem{thm}{Theorem} 
\newtheorem{lemma}[thm]{Lemma} 

\newtheorem{cor}[thm]{Corollary}

\newtheorem{conjecture}[thm]{Conjecture} 
\newtheorem{proposition}[thm]{Proposition}

\newtheorem{definition}{Definition}  

\newtheorem{claim}{Claim}

\newtheorem{note}{Note}

\def\halnospace{\hfill $\triangle$}

\def\pp{\mathcal P}
\def\qq{\mathcal Q}
\def\tt{\mathcal T}
\def\ss{\mathcal S}
\def\xx{\mathcal X}

\newcommand{\mc}{\mathcal}

\title{The Maximum Chromatic Number of the Disjointness Graph of Segments on $n$-point Sets in the Plane with $n\leq 16$}

\author[1]{J. J. Garc\'{\i}a-Davila}
\author[2]{J. Lea\~nos}
\author[3]{M. Lomel\'{\i}-Haro}
\author[4]{L. M. R\'{\i}os-Castro}

\affil[1,2]{ Unidad Acad\'emica de Matem\'aticas, Universidad Aut\'onoma de Zacatecas, M\'exico, 
\texttt{jgarcia.mate.uaz@gmail.com, jleanos@uaz.edu.mx}}
\affil[3]{ Instituto de F\'{\i}sica y Matem\'aticas, Universidad Tecnol\'ogica de la Mixteca, M\'exico, 
\texttt{lomeli@mixteco.utm.mx}}
\affil[4]{Academia de F\'{\i}sico--Matem\'aticas, Instituto Polit\'ecnico Nacional, CECYT 18, M\'exico, 
\texttt{lriosc@ipn.mx} }

\begin{document}
\maketitle
%\linenumbers
%\tableofcontents
	
\abstract{
Let $P$ be a finite set of points in general position in the plane. 
The {\em disjointness graph of segments} $D(P)$ of $P$ is the graph
whose vertices are all the closed straight line segments with endpoints in $P$, two of which are adjacent in $D(P)$ if and only if they are disjoint. As usual, we use 
$\chi(D(P))$ to denote the chromatic number of $D(P)$, and use $d(n)$ to denote the maximum $\chi(D(P))$ taken over all sets $P$ of $n$ points in general position in the plane. In this paper we show that $d(n)=n-2$ if and only if  $n\in \{3,4,\ldots ,16\}$.}


\section{Introduction}\label{sec:intro}

The {\em chromatic number} of a graph $G$ is the minimum number of colors needed to color 
its vertices so that adjacent 
vertices receive different colors; it is denoted by $\chi(G)$. 
For $k,n\in {\mathbb Z}^+$ with 
$k \leq n/2$, the {\em Kneser graph} $KG(n,k)$ has as its vertices the $k$-subsets of 
$\{1,2,\ldots,n\}$ and two of such $k$-subsets form an edge if and only if they are disjoint. 
Kneser conjectured \cite{kneser} in 1956 that $\chi( KG (n,k )) = n-2k+2$. 
This conjecture was proved by Lov\'asz \cite{lovasz} 
and independently by B\'ar\'any \cite{barany} in $1978$.

Let $P$ be a set of $n\geq 2$ points in general position in the plane,   
and let ${\cal P}$ be the set of all $\binom{n}{2}$ closed straight line segments with endpoints in $P$. 
The {\em disjointness graph of segments} $D(P)$ of $P$ is the graph whose vertex set
is ${\cal P}$, and two elements of ${\cal P}$ are adjacent if and only if their
corresponding segments are disjoint. See Figure~\ref{fig:T1T2} for an example. 
 
The disjointness graph of segments $D(P)$ was introduced in $2005$ by Araujo et al. \cite{gaby}, as a geometric 
version of the Kneser graph $KG(n,k)$ for $k=2$. It follows from the definitions of $KG(n,2)$ and $D(P)$ that if $|P|=n\geq 2$, then both have size
$\binom{n}{2}$ and $KG(n,2)$ contains a subgraph which is isomorphic to $D(P)$. Similarly, we note that the crossings between
segments of $\pp$ are responsible for the edges of $KG(n,2)$  that are not in $D(P)$. Indeed, suppose that $P=\{p_1, p_2,\ldots ,p_n\}$
and consider the  natural bijection $f(i)\to p_i$ between $\{1,2,\ldots ,n\}$ and $P$. Note that if $\{i,j\}$ and $\{k,l\}$ are adjacent in 
$KG(n,2)$, then the corresponding segments of $\pp$ defined by $\{p_i,p_j\}$ and $\{p_k,p_l\}$ are also adjacent in $D(P)$, unless they cross each other.  
In a nutshell, the difference between $KG(n,2)$ and $D(P)$ is completely determined by the crossings between segments of $\pp$. On the other hand, in~\cite{magistral}
it was proved that the number of crossings in $\pp$ grows at least $0.37997\binom{n}{4}+\Theta(n^3)$, and so the number of edges that are in $KG(n,2)$ but not in $D(P)$ 
is linear in the number of edges of $KG(n,2)$.   

It is natural to ask about Kneser's conjecture in the context of the disjointness graph of segments. Since $D(P)$ is a subgraph of $KG(n,2)$, it is clear that
 $\chi(D(P))\leq\chi(KG(n,2))=n-2$. Moreover, since the number of edges of $D(P)$ is significantly smaller than the number of edges of $KG(n,2)$,
 one might expect that $\chi(D(P))< n-2$ most of the time.
 
 For $n\in {\mathbb Z}^+$, let us define
 $$d(n):=\max\{\chi(D(P))~:~P \text{ is an } n-\text{point set in general position in the plane}\}.$$ 
 
\noindent The following questions arise naturally from the above discussion.\\
 
\noindent {\bf Question 1.} {\em For what values of $n$ is $d(n)=n-2$?} \\

\noindent {\bf Question 2.} More generally, {\em What is the value of $d(n)$ for each $n\in {\mathbb Z}^+$?}\\     

 The systematic study of several combinatorial properties of $D(P)$ began with the work of Araujo et al. in 2005~\cite{gaby}. In particular, the following general bounds for
 $d(n)$ with $n\geq 3$ were proved in~\cite{gaby},
\[
5\left\lfloor\frac{n}{7}\right\rfloor \leq d(n) \leq 
\min \left\{ n-2, n + \frac{1}{2} - \frac{\lfloor \log\log(n)\rfloor}{2}\right\}. 
\]
We note that these inequalities imply that $d(7)=5$.  

The problem of determining the chromatic number $\chi(D(P))$ of $D(P)$ remains open in general. Furthermore, as far as we know, this problem has been solved only for two families of points: the convex sets and the double chains. 

For $m\in {\mathbb Z}^+$ we shall use $C_m$ to denote a set of $m$ points in (general and) {\em convex position} in the plane. We recall that for $k,l\in {\mathbb Z}^+$ 
a {\em double chain} $C_{k,l}$ is a $(k+l)$-point set in general position in the plane such that $C_{k,l}$ is the dijoint union of $C_k$ and $C_l$ where the points are located in such a way that  any point of $C_k$ (resp. $C_l$) is below (resp. above) every straight line spanned by two points of $C_l$ (resp. $C_k$). In Figure~\ref{fig:T1T2}($a$) the sets 
$T_1=\{t_1^1, t_1^2, t_1^3\}$ and $T_2=\{t_2^1, t_2^2, t_2^3\}$ are two instances of $C_3$, and $T_1\cup T_2$ is an instance of $C_{3,3}$. 

In~\cite{Cn} 2018 the exact value of $\chi(D(C_n))$ has been settled by Fabila-Monroy,  Jonsson,  Valtr and Wood:
 \[
\chi(D(C_n))=n-\left\lfloor\sqrt{2n+\frac{1}{4}}-\frac{1}{2}\right\rfloor.
\]

Similarly, in~\cite{mario}  2020 Fabila-Monroy, Hidalgo-Toscano, Lea\~nos and Lomel\'{\i}­-Haro have shown that if $C_{k,l}$ is a double 
chain with $l\geq \max\{3,k\}$,  then
\[
\chi(D(C_{k,l}))=k+l-\left\lfloor\sqrt{2l+\frac{1}{4}}-\frac{1}{2}\right\rfloor.
\]
Note that if $l=k$, then $\chi(D(C_{k,l}))$ provides the currently best lower bound of $d(n)$, namely:
\[
d(n)\geq n-\left\lfloor\sqrt{n+\frac{1}{4}}-\frac{1}{2}\right\rfloor.
\]

As we shall see, the exact values of $\chi(D(C_n))$ and $\chi(D(C_{k,l}))$ will be useful in our proof. Another important fact that we use is a result proved by Szekeres and 
Peters in~\cite{6gons} 2006, in which they confirmed the following conjecture for $n=6$.

\begin{conjecture}[Erd\H{o}s and Zsekeres, 1935]\label{con:ES}
Any set of $2^{m-2}+1$ points in general position in the plane contains a convex $m$-gon.
\end{conjecture}

The study of the combinatorial properties of the disjointness graph of segments and related graphs have received considerable attention lattely. 
For more results about this kind of graphs we refer the reader to~\cite{lara,birgit,gaby,dujmovic,us1,pach-tardos-toth,pach-tomon}. 

Our main objective in this paper is to give a full answer to the Question 1. As we have mentioned above, by Lov\'asz's 
theorem~\cite{lovasz} and the fact that $D(P)$ is a subgraph of $KG(n,2)$, it follows that $\chi(D(P))\leq n-2$ whenever $n=|P|\geq 2$. 
Our main result  is the following theorem.
 
\begin{thm}\label{thm:main}
For $n\geq 2$ integer, let $d(n)$ be defined as above. Then $d(n)=n-2$ iff $n\in \{3,4,\ldots ,16\}$. 
\end{thm}

In view of the above discussion, it is clear that in order to show that $d(n)=n-2$ for $n\in \{3,4,\ldots ,16\}$, it suffices to exhibit an $n$-point set $P$ in general position in the plane
such that $\chi(D(P))=n-2$. The rest of the paper is devoted to this end, and in fact the heart of our proof is to show 
that any subset $X'$ of the $16$-point set $X$ given in Figure~\ref{fig:X}, satisfies the required inequality for each $|X'|\geq 3$, namely that $\chi(D(X'))\geq |X'|-2$.  

In the early stages of this work, we were trying to attack this problem by computer search, but soon we convinced ourselves that 
the number of colorings of $D(X)$ that must be considered grows exponentially with the number of vertices of $D(X)$. As we shall see later, a surprising fact 
that illustrates the computational complexity of this problem is the following: $D(X)$ has a subset $\ss$ with at least 100 vertices such that if $v\in \ss$, then $D(X)\setminus \{v\}$ can be colored with only 13 colors. %In particular, we believe that a hard computer-assisted proof of Theorem~\ref{thm:main} is out of reach.  
    
The rest of the paper is organized as follows. In Section~\ref{s:X} we define the set $X$ by means of the precise coordinates of its $16$ points, and
give a brief discussion on its geometric properties. In Section~\ref{s:notation} we introduce some notation and terminology that we shall use
in most of the proofs. In Section~\ref{s:basic-facts} we present a summary of basic (or known) facts that we shall use in Section~\ref{s:lemmas}. The
 more technical work of this paper is given in Section~\ref{s:lemmas}, there we prove the main claims behind the proof of Theorem~\ref{thm:main}.
Finally, in Section~\ref{s:proof} we establish Theorem~\ref{thm:main} by combining in several ways the results
stated in previous sections. 

 \begin{figure}[t]
\centering
\includegraphics[width=0.6\textwidth]{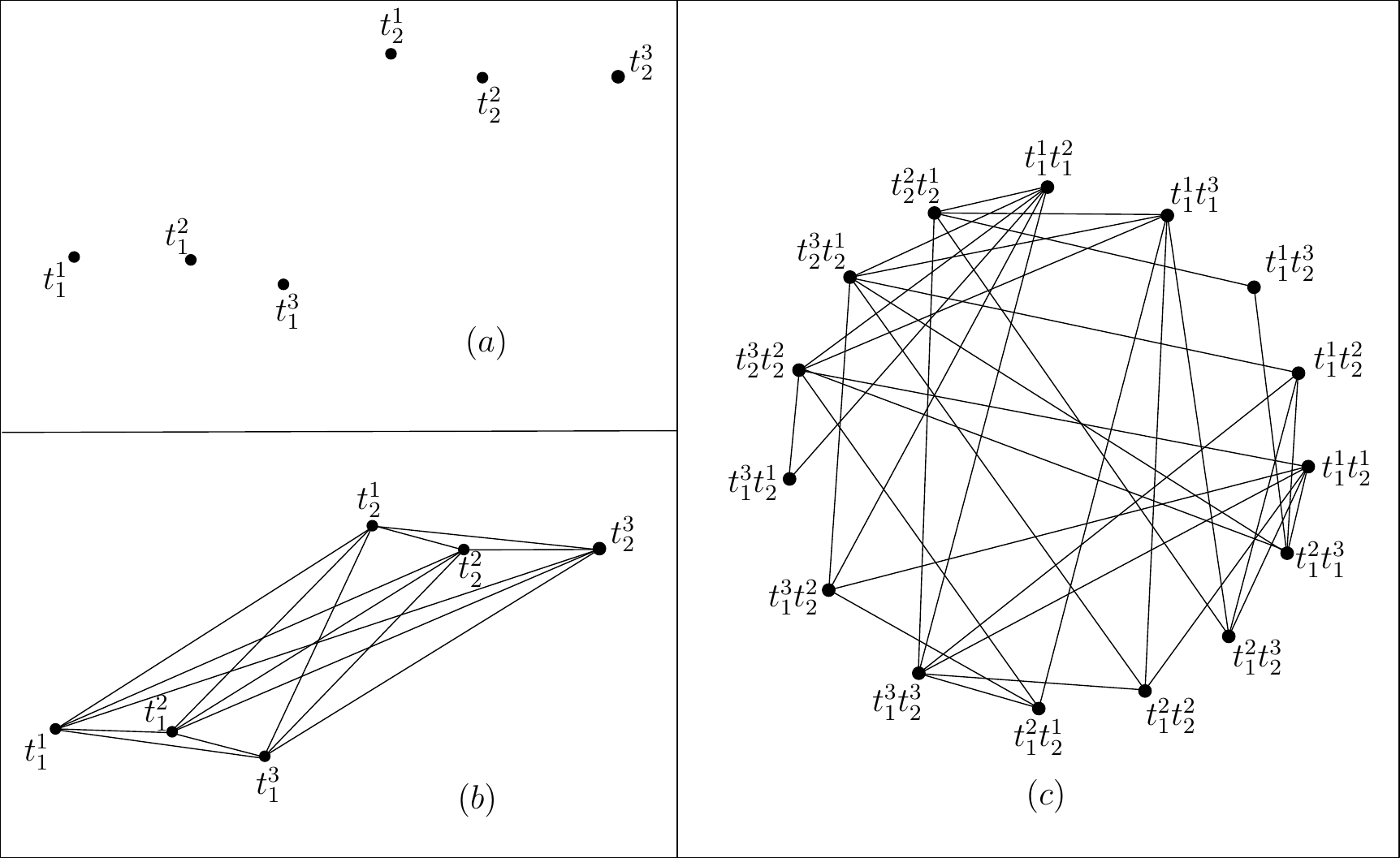}
\caption{\small ($a$) Let $T_1=\{t_1^1, t_1^2, t_1^3\}$ and $T_2=\{t_2^1, t_2^2, t_2^3,\}$. 
Then $T_1\cup T_2$ is a set of 6 points in general position in the plane. Each of $T_1$ and
$T_2$ is an instance of $C_3$, and $T_1\cup T_2$ is an instance of the double chain $C_{3,3}$. ($b$) This is the complete geometric graph $\tt_1\cup \tt_2$ induced by 
$T_1\cup T_2$. The graph  in ($c$) is the disjointness graph of segments  induced by 
$T_1\cup T_2$, and we shall denote it by $D(T_1\cup T_2)$.}
\label{fig:T1T2}
\end{figure}

\section{The $16$-point set $X$}\label{s:X}

We recall that if $P$ and $Q$ are two finite point sets in general position in the plane, then it is said they have the same {\em order type} iff there is a bijection
$f:P\to Q$ such that each ordered triple $abc$ in $P$ has the same orientation as its image $f(a)f(b)f(c)$.
We shall write $P\sim Q$ if $P$ and $Q$ have the same order type. In particular, it is well known (see for instance~\cite{magistral}) that if $P\sim Q$, then $D(P)$ and $D(Q)$ are isomorphic.

For the rest of the paper, $X$ denotes the $16$-point set given in Figure~\ref{fig:X} and we will refer to its points according to the labelling depicted there. In particular, we 
partition $X$ in the subsets $A, B, T_1, T_2$, where $A:=\{a_1, a_2, \ldots ,a_5\}$,  $B:=\{b_1, b_2, \ldots ,b_5\}$, 
$T_1:=\{t^1_1, t^2_1, t^3_1\}$ and $T_2:=\{t^1_2, t^2_2, t^3_2\}$. We recall that $\xx$ is the set of $\binom{16}{2}$ closed straight line segments
with endpoints in $X$.   
\begin{figure}[t]
\centering
\includegraphics[width=0.95\textwidth]{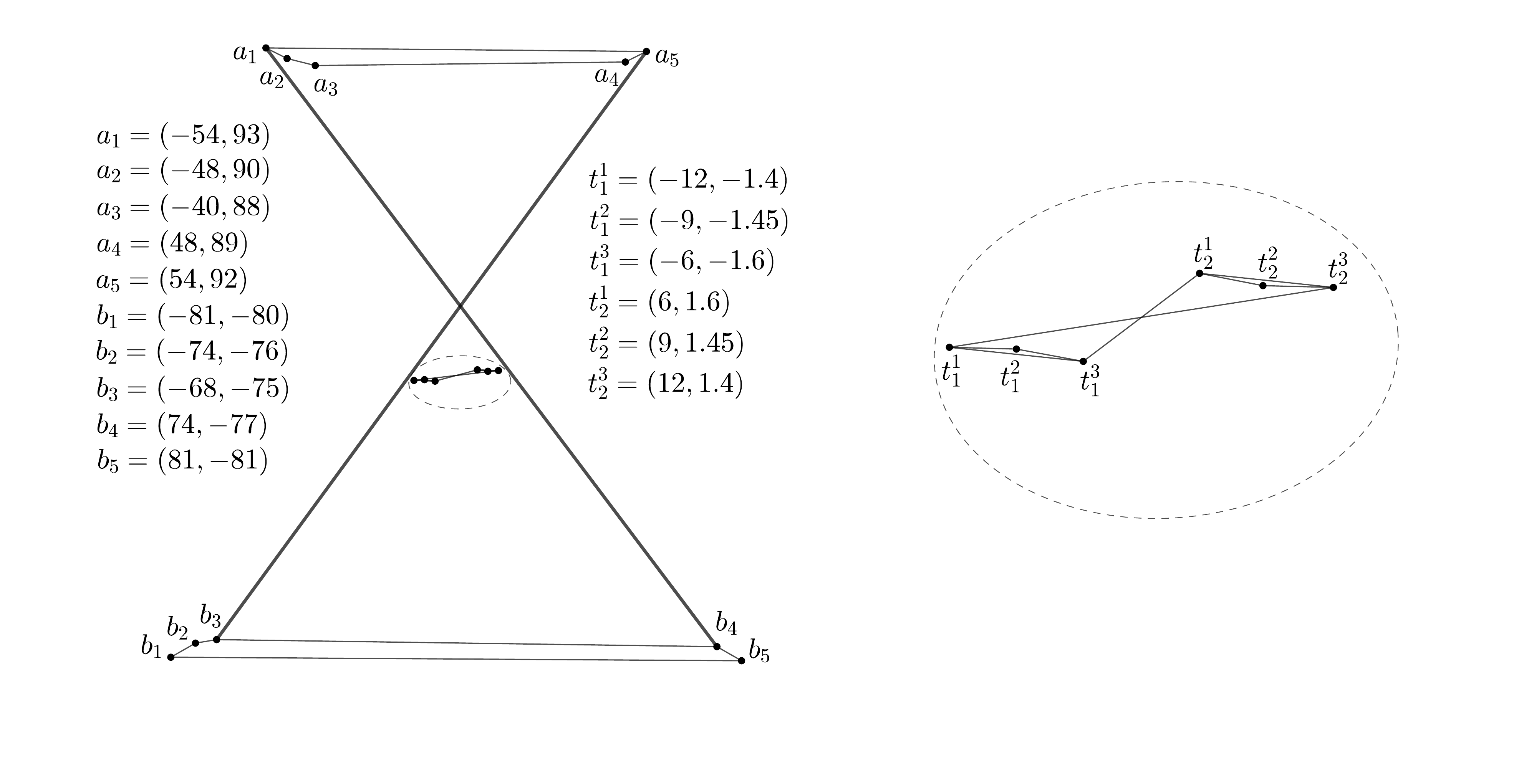}
\caption{\small  The $16$-point set defined by these coordinates is $X$. The ellipse on the right is a zoom of the ellipse on the left. We will derive the proof of Theorem~\ref{thm:main} by showing that $\chi(D(X'))=|X'|-2$ for each $X'\subseteq X$ with $|X'|\geq 3$. We remark that the $6$-point set 
$\{t_1^1, t_1^2, t_1^3, t_2^1, t_2^2, t_2^3\}$ is the same (order type) as those in Figure~\ref{fig:T1T2}($a$)-($b$).}
\label{fig:X}
\end{figure}

The following properties of $X$ are easy to check: 
\begin{itemize}
\item $X$ is $16$-point set in general position in the plane which has no $6-$points forming a convex hexagon, 
\item no segment of $\xx$ is neither horizontal nor vertical, 
\item each segment with endpoints in $T_i$ has negative slope for each $i\in \{1,2\}$,
\item $A\cup B \sim C_{5,5}$, $A\cup T_1 \sim C_{5,3}$,  $T_2\cup B \sim C_{3,5}$, $T_1\cup T_2 \sim C_{3,3}$ and $A\cup T_2\sim B\cup T_1$.  
\end{itemize}

From now on, we will use these properties of $X$ without explicit mention.  

For $e\in \xx$ we shall use $l(e)$ and $r(e)$ to denote the leftmost and the rightmost point of $e$, respectively. Since $\xx$ 
has no vertical segments, then $l(e)$ and $r(e)$ are well-defined for each $e\in \xx$. Then $l(e)$ and $r(e)$ are points of 
$X$ and are the endpoints of $e$. We remark that if $e$ has both endpoints in $T_i$ for some
$i\in \{1,2\}$, then $l(e)$ (resp. $r(e)$) is also the topmost (resp. lowest) point of $e$.

%%%%%%%%%%%%%%%%%%%%%%%%%%%%%%%%%%%%%%%%%%%%%%%%%%%%%%%%%%%%%%%%%%%%%%%% 
%%%%%%%%%%%%%%%%%%%%%%%%%%%%%%%%%%%%%%%%%%%%%%%%%%%%%%%%%%%%%%%%%%%%%%%%
%%%%%%%%%%%%%%%%%%%%%%%%%%%%%%%%%%%%%%%%%%%%%%%%%%%%%%%%%%%%%%%%%%%%%%%%
%%%%%%%%%%%%%%%%%%%%%%%%%%%%%%%%%%%%%%%%%%%%%%%%%%%%%%%%%%%%%%%%%%%%%%%%
%%%%%%%%%%%%%%%%%%%%%%%%%%%%%%%%%%%%%%%%%%%%%%%%%%%%%%%%%%%%%%%%%%%%%%%%
%%%%%%%%%%%%%%%%%%%%%%%%%%%%%%%%%%%%%%%%%%%%%%%%%%%%%%%%%%%%%%%%%%%%%%%%  
    
\section{Notational conventions and terminology}\label{s:notation}

Our aim in this section is to introduce some useful notational conventions, terminology and concepts that we shall use in the rest of the paper. 

Throughout this section, $P$ denotes a set of $n\geq 2$ points in general position in the plane. If $x$ and $y$ are distinct points of $P$,
then we shall use $xy$ to denote the closed straight line segment whose endpoints are $x$ and $y$. If $R$ and $S$  are disjoint nonempty subsets of $P$, then 
$R*S:=\{xy~:~x\in R \mbox{ and } y\in S\}$. We simply write $x*S$ (resp. $R*y$) whenever $R=\{x\}$ (resp. $S=\{y\}$). 

If $x,y,z$ are three distinct points of $P$, then we will refer to the triangle formed by them in any of the following ways:
$\Delta(x,y,z)$, $\Delta(xy,z)$, $\Delta(xy,xz)$, etc.

We shall denote by $\overline{P}$ the boundary of the convex hull of $P$.

If $Q\subseteq P$ with $|Q|\geq 2$, then  we shall use the font style $\qq$ to denote the set of all $\binom{|Q|}{2}$ closed straight 
line segments with endpoints in $Q$. We often make no distinction between the set of segments $\qq$ 
and the complete geometric graph that it induces.

\begin{note}\label{rem:edge-coloring} From the definition of $D(P)$ it follows that $\chi(D(P))$ is the minimum number of colors in an edge-coloring of the complete geometric graph $\pp$ in which any two edges belonging to the same chromatic class cross or are incident. We remark that all our proofs are given in terms of this kind of
 edge-colorings of $\pp$. Unless otherwise stated, from now on, any coloring of $D(P)$ must be assumed to be an edge-coloring of the complete geometric 
 graph $\pp$.    
 \end{note}
 
 We classify each chromatic class $c$ of a given coloring $\gamma$ of $\pp$ as either thrackle or star. We call $c$
  a {\em thrackle}  if at least two segments in $c$ cross each other, and otherwise $c$ is a {\em star}. If $c$ is a star whose
 segments are incident with exactly $m$ points of $P$, then $c$ is an {\em $m$--star}. Clearly, each $m$--star $c$ with $m\geq 3$ has a unique {\em apex}, i.e. a
 point of $P$ that is incident with all the segments of $c$. If $c$ is a $2$--star then $c$ consists of a singleton segment $e$, and in this case both ends of
 $e$ are considered apices of $c$. 

Let $\gamma$ be a coloring of $\pp$ and let $Q\subseteq P$. For exposition purposes, we abuse notation and use $\gamma(Q)$ to refer to the 
set of colors of $\gamma$ that are present in $\qq$. We note that the restriction $\gamma|_{\qq}$ of $\gamma$ to $\qq$ is a coloring of $\qq$. 
We let $\gamma^{\star}(Q)$ denote the number of points in $Q$ that are apices of some star of $\qq$ with respect to $\gamma|_{\qq}$.

For $Q\subseteq P$, let $\qq^+$ be the set  of all segments of $\pp$ that have no endpoints in $Q$ and cross a segment of $\qq$. 
We say that $Q$ is {\em separable with respect to} $\gamma$ whenever $\gamma(e)\notin \gamma(Q)$ for any $e\in \qq^+$. If $\qq^+=\emptyset$, then 
 $Q$ is a {\em separable} set of $P$. Clearly, if $Q$ is a separable set of $P$, then $Q$ is separable with respect to any $\gamma$. 
For example, each of  $A, B, T_1, T_2, A\cup B, T_1\cup T_2$ is a separable set of $X$.

%%%%%%%%%%%%%%%%%%%%%%%%%%%%%%%%%%%%%%%%%%%%%%%%%%%%%%%%%%%%%%%%%%%%%%%% 
%%%%%%%%%%%%%%%%%%%%%%%%%%%%%%%%%%%%%%%%%%%%%%%%%%%%%%%%%%%%%%%%%%%%%%%%
%%%%%%%%%%%%%%%%%%%%%%%%%%%%%%%%%%%%%%%%%%%%%%%%%%%%%%%%%%%%%%%%%%%%%%%%
%%%%%%%%%%%%%%%%%%%%%%%%%%%%%%%%%%%%%%%%%%%%%%%%%%%%%%%%%%%%%%%%%%%%%%%%
%%%%%%%%%%%%%%%%%%%%%%%%%%%%%%%%%%%%%%%%%%%%%%%%%%%%%%%%%%%%%%%%%%%%%%%%
%%%%%%%%%%%%%%%%%%%%%%%%%%%%%%%%%%%%%%%%%%%%%%%%%%%%%%%%%%%%%%%%%%%%%%%%  

\section{Basic facts}\label{s:basic-facts}

Our aim in this section is to prove some basic facts, which we will often use in the rest of the paper. 

Again, in this section $P$ denotes a set of $n\geq 3$ points in general position in the plane. We recall that if $\gamma$ is a coloring of $\pp$,
then $\gamma(P)$ is the set of colors used by $\gamma$ to color the segments of $\pp$. 
 
The next proposition is an immediate consequence of Lov\'asz's theorem, but here we give an algorithmic proof because we will use the resulting coloring several times. 

%%%%%%%%%%%%%%%%%%%%%%%%%%%%%%%%%%%%%%%%%%%%% %%%%%%%%%%%%%%%

\begin{proposition}\label{p:n-2} $\chi(D(P))\leq |P|-2.$ 
\end{proposition}
\begin{proof} Let $p_1, p_2, \ldots , p_{n}$ be any labeling of the points of $P$. Color each segment of the triangle
 $\Delta(p_1, p_2, p_3)$ with color $c_1$, and for each $j\in \{4,5,\ldots ,n\}$ use $c_j$ to color each segment in $\{p_jp_i~:~i=1,2,\ldots ,j-1\}$. Since this 
defines a proper coloring of $\pp$ with $n-2$ colors, we are done. 
\end{proof}

%%%%%%%%%%%%%%%%%%%%%%%%%%%%%%%%%%%%%%%%%%%%% %%%%%%%%%%%%%%%

Proposition~\ref{p:many} is not used in the rest of the paper, but it is relevant because illustrates the tightness of Theorem~\ref{thm:main}, and at the same time 
the difficulty of proving it by cumputer search.  

\begin{proposition}\label{p:many} Let $\ss:=\{t_1^1t_2^3, b_1 t_2^3, b_1 t_2^2, b_1 t_2^1, b_1 a_5, b_2 a_5, b_3 a_1, b_3 a_2, b_3 a_3, b_3 a_4, b_3 a_5, b_4 a_1, b_4 a_2, b_4 a_3, b_4 a_4\-, b_4 a_5, b_5 t_1^2, b_5 t_1^1, b_5 t_2^1, b_5 a_1\}$. For $x_1x_2\in \xx \setminus \ss$, there are many ways to color $D(X)$ with 14 colors in which
the color of $x_1x_2$ is not assigned to any other segment of $\xx$.  
\end{proposition}

\begin{proof} Let $x_1x_2\in \xx \setminus \ss$. 
It is not hard to see that $x_1x_2$ is a side (no diagonal) of some convex $5$-gon $X_5\subset X$. Similarly, 
it is easy to find a valid $3$-coloring $\gamma'$ of the 10 segments of $\xx_5$ in such a way that the color $\gamma'(x_1x_2)$ 
does not appear in any other segment of $\xx_5$. In fact, there are exactly 2 ways to do this.  
Let $x_6, x_7, \ldots , x_{16}$ be any labeling of the points of $X\setminus X_5$. Following the same argument as in the proof of Proposition~\ref{p:n-2}, 
we can extend $\gamma'$ to a  coloring $\gamma$ of $D(X)$ by adding a new star $S_i$ of color $c_i$ with 
apex $x_i$ for each $x_i\in X\setminus X_5$.  
\end{proof}

%%%%%%%%%%%%%%%%%%%%%%%%%%%%%%%%%%%%%%%%%%%%% %%%%%%%%%%%%%%%

\begin{proposition}\label{p:closed}  
Let $P'$ be a proper subset of $P$, and 
 let $\gamma$ be an optimal coloring of $D(P)$. Then the following hold:
\begin{itemize}
\item[(i)] If $|P'|\geq 3$ and $\chi(D(P))=|P|-2$, then $\chi(D(P'))=|P'|-2$.
\item[(ii)] If $P'$ is separable with respect to $\gamma$ and $|\gamma(P')|=|P'|,$
then $\chi(D(P)) \geq \chi(D(P\setminus P')) + |P'|$.
\item[(iii)] If $S_1,\ldots , S_r$ are different stars of $\gamma$ with apices $v_1, \ldots ,v_r$, respectively. Then 
 $\chi(D(P\setminus \{v_1, \ldots , v_r\}))=\chi(D(P))-r$.
\end{itemize}

\end{proposition}
\begin{proof} By Proposition~\ref{p:n-2}, in order to show $(i)$ it suffices to show that $\chi(D(P'))\geq |P'|-2$. Seeking a contradiction, suppose that there exists a proper 
coloring $\beta'$ of $D(P')$ such that $|\beta'(P')|<|P'|-2$.  Then, following the same argument as in the proof of Proposition~\ref{p:n-2}, we can extend $\beta'$ to a 
 coloring $\beta$ of $D(P)$ by adding a new star $S_i$ of color $c_i$ with apex $p_i$ for each $p_i\in P\setminus P'$. 
Then the number of colors of $\beta$ is less than $|P'|-2+|P\setminus P'|=|P|-2$, contradicting the hypothesis that $\chi(D(P))=|P|-2$.

Let $P'$ and $\gamma$ be as in~$(ii)$. Since $P'$ is a separable subset of $P$ with respect to $\gamma$, then $\gamma(P')$ and $\gamma(P\setminus P')$ are disjoint subsets
of $\gamma(P)$ and so $\chi(P) = |\gamma(P)|\geq |\gamma(P\setminus P')|+ |\gamma(P')|\geq \chi(P\setminus P') + |P'|$. This proves $(ii)$.

For $i=1,\ldots ,r$,  let $S_i$ and $v_i$ be as in~$(iii)$ and let $Q=P\setminus \{v_1, \ldots , v_r\}$. Since the restriction of $\gamma$ to $D(Q)$
is a  coloring of $D(Q)$ with at most $\chi(D(P))-r$ colors, then $\chi(D(Q))\leq \chi(D(P))-r$. On the other hand, if $\chi(D(Q))<\chi(D(P))-r$ then
we can proceed as in the proof of $(i)$ to obtain a %(contradiction) 
 coloring $\beta$ of $D(P)$ with less than $\chi(D(Q))+r$ colors, which is a contradiction.  
\end{proof}

\begin{proposition}\label{p:5apices} Let $\gamma$ be a 4-coloring of $D(C_5)$. If  $\gamma^{\star}(C_5)=5$, then $C_5$ has two points $p$ and $q$
such that $\{pq\}$ is a $2-$star of $\gamma$ and neither $p$ nor $q$ is an apex of any other star of $\gamma$. 
\end{proposition}
\begin{proof}
We start by noting that any chromatic class of $\gamma$ is formed by at most $5$ segments of $\mc C_5$. From $|\mc C_5|=10$ and the hypothesis that $\gamma$ 
has 4 chromatic classes it follows that $\gamma$ has at most two $2-$stars. On the other hand, since each point of $C_5$ is an apex of a 
star of $\gamma$, then $\gamma$ must contain at least one $2-$star, say $S=\{pq\}$. 
We now show that if $S'$ is another star of $\gamma$ with apex $v\in \{p,q\}$, 
then $\mc C_5$ has a segment that cannot be colored by $\gamma$, contradicting the assumption that $\gamma$ is a coloring of $D(C_5)$.
 
Suppose first that $S'$ is a $2-$star. Then, $S'=\{vw\}$ for some 
$w\in C_5\setminus \{p,q\}$. Let $x$ and $y$ be the two points in $C_5\setminus \{p,q,w\}$. By hypothesis, for each $z\in \{x,y\}$, $\gamma$ has a star $S_z$ 
with apex $z$. As $S$ and $S'$ are the only $2-$stars of $\gamma$, then $S_x\neq S_y$, and so $S, S', S_x,$ and $S_y$ are the 4 chromatic classes of 
$\gamma$. Since each of them is a star, then the segment joining the two points in $C_5\setminus \{x,y,v\}$ cannot be colored by $\gamma$.
 
Suppose now that $S'$ is a
star and let $z_1,z_2,$ and $z_{3}$ be the points of $C_5\setminus \{p,q\}$. Again, for each $z_i$ 
we know that $\gamma$ has a star $S_{z_i}$ with apex $z_{i}$. Since $\gamma$ has exactly 4 chromatic classes, then two of $S_{z_1}, S_{z_2}, S_{z_3}$ 
must be the same. W.l.o.g. we asume $S_{z_1}=S_{z_2}$. Then $S_{z_1}$ must be a $2-$star, and so 
$S, S', S_{z_1},$ and $S_{z_3}$ are the 4 chromatic classes of $\gamma$. Since each of them is a star, then the segment joining the
 two points in $C_5\setminus \{v,z_1,z_3\}$ cannot be colored by $\gamma$.
\end{proof}

%%%%%%%%%%%%%%%%%%%%%%%%%%%%%%%%%%%%%%%%%%%%% %%%%%%%%%%%%%%%
\begin{proposition}\label{p:gamma(A)=3} 
If $\gamma$ is a 3-coloring of $D(C_5)$, then $\gamma^{\star}(C_5)\leq 2$. 
\end{proposition}
\begin{proof} Let $v_1, v_2,\ldots ,v_5$ be the points of $C_5$, so that they appear in this cyclic order in $\overline{C_5}$. Since each color of $\gamma$
is in at most two segments of $\overline{C_5}$, then the 3 colors $c_1, c_2, c_3$ of $\gamma$ appear in $\overline{C_5}$. w.l.o.g. assume
that $\gamma(v_1v_2)=c_1$, $\gamma(v_2v_3)=\gamma(v_3v_4)=c_2$ and $\gamma(v_4v_5)=\gamma(v_5v_1)=c_3$.

Suppose first that $v_1v_2$ is a $2-$star of $\gamma$. Then, $\gamma(v_2v_4)=c_2$ and  $\gamma(v_4v_1)=c_3$, and so none of $v_3, v_4, v_5$ can be
an apex of $\gamma$, implying that $\gamma^{\star}(C_5)\leq 2$.  
If $v_1v_2$ is not a $2-$star of $\gamma$, then each colour of $\gamma$ appears in at least two segments of $\mc C_5$, and so each star
of $\gamma$ has a unique apex. Since if the 3 chromatic classes of $\gamma$ are stars, then some of $v_1v_4$ or $v_2v_4$ 
cannot be colored by $\gamma$, we can assume that $\gamma$ has at most two stars, and hence $\gamma^{\star}(C_5)\leq 2$. 
 \end{proof}
   
%%%%%%%%%%%%%%%%%%%%%%%%%%%%%%%%%%%%%%%%%%%%% %%%%%%%%%%%%%%%
\begin{thm}[Theorem 1 in~\cite{mario}]\label{thm:mario} If $l, k\in {\mathbb Z}^+$ and $l\geq \max\{3,k\}$, then 
$$\chi(D(C_{k,l}))=k+l-\left\lfloor \sqrt{2l+\frac{1}{4}} - \frac{1}{2}\right\rfloor.$$
\end{thm}

\begin{cor}\label{c:neq6} Let $\gamma$ be a coloring of $D(X)$. Then 
\begin{itemize}

\item[($i$)] $|\gamma(T_1\cup T_2)|\geq 4$. 
\item[($ii$)] $|\gamma(A\cup B)|\geq 8$.
\item[($iii$)] $|\gamma(B\cup T_2)|\geq 6$.
\item[($iv$)] $|\gamma(A\cup T_1)|\geq 6$.
\item[($v$)] If $|\gamma(T_1\cup T_2)|\geq 6$, then $|\gamma(X)|\geq 14$.  
\item[($vi$)] If $|\gamma(A\cup B)|\geq 10$, then  $|\gamma(X)|\geq 14$.

\end{itemize}
\end{cor}
\begin{proof} Since $T_1\cup T_2\sim C_{3,3}$, $A\cup B\sim C_{5,5}$ and 
$B\cup T_2 \sim A\cup T_1 \sim C_{3,5}$ then $|\gamma(T_1\cup T_2)|\geq 4$, 
\noindent $|\gamma(A\cup B)|\geq 8$, $|\gamma(B\cup T_2)|\geq 6$ and 
$|\gamma(A\cup T_1)|\geq 6$, by Theorem~\ref{thm:mario}.

Since $T_1\cup T_2$ is a separable subset of $X$, then $|\gamma(X)|\geq |\gamma(A\cup B)|+|\gamma(T_1\cup T_2)|$. 
If  $|\gamma(T_1\cup T_2)|\geq 6$, then {\it (ii)} implies  {\it(v)}. Similarly, if $|\gamma(A\cup B)| \geq 10$, then {\it (i)} implies {\it (vi)}. 
\end{proof}
 
Most of our proofs are by case analysis, depending on the cardinality of the following sets $\gamma(A), \gamma(B), 
\gamma(T_1)$ and $\gamma(T_2)$. For the rest of the paper we will use the following additional notation.

\begin{itemize}
\item[(N1)] If $|\gamma(A)|=3$, Proposition~\ref{p:gamma(A)=3} implies that $A$ has 3 points that are not apices of $\gamma(A)$. We will denote by 
$a_i, a_j, a_k$ these 3 points of $A$, and w.l.o.g. we assume that $i<j<k$. If $|\gamma(B)|=3$ we define $b_p, b_q, b_r$ with $p<q<r$, similarly. 

\item[(N2)] If $|\gamma(A)|=4$ and $\gamma^{\star}(A) = 5$, Proposition~\ref{p:5apices} implies that  $A$ has 2 points defining a $2-$star of $\gamma(A)$
and are such that none of them is apex of any other star of $\gamma(A)$. We will denote by 
$a_i$ and $a_j$ these 2 points of $A$, and w.l.o.g. we assume that $i<j$. If $|\gamma(B)|=4$ and $\gamma^{\star}(B) = 5$, we define $b_p$ and $b_q$ with $p<q$, similarly.  

\item[(N3)] If $|\gamma(T_i)|=2$ for $i\in \{1,2\}$, we let $e_i$ denote the segment of $T_i$ that has a different color than the other two segments of $T_i$.
\end{itemize}   

%%%%%%%%%%%%%%%%%%%%%%%%%%%%%%%%%%%%%%%%%%%%% %%%%%%%%%%%%%%%
\begin{proposition}\label{p:BT1} $\chi(D(A\cup T_2))=6$ and $\chi(D(B\cup T_1))=6$. 
\end{proposition}

\begin{proof} Since $A\cup T_2 \sim B\cup T_1$, it is enough to show that $\chi(D(B\cup T_1))=6$. 
Since $\chi(D(B\cup T_1))\leq 6$ by Proposition~\ref{p:n-2}, we need to show that $\chi(D(B\cup T_1))\geq 6$.
Let $\gamma$ be a coloring of $D(B\cup T_1)$. Clearly, $|\gamma(B)|\geq 3$, $|\gamma(T_1)|\geq 1$ and 
$|\gamma(B\cup T_1)|\geq |\gamma(B)| + |\gamma(T_1)|$. Then either $|\gamma(B)|\leq 4$ or we are done. 

If $|\gamma(B)|=3$, then $b_{p}t^1_1, b_{q}t^2_1$ and $b_{r}t^3_1$ receive distinct color, where $b_p, b_q$ and $b_r$ are as in (N1).
 Since none of these 3 colors appears in any segment of $\gamma(B)$, we are done.    
 
Suppose now that $|\gamma(B)|=4$. Then $|\gamma(T_1)|=1$, as otherwise we are done. 
If $B$ contains a point $b$ that is not an apex of $\gamma(B)$, then $\gamma(bt^1_1)$ is the $6$th required color. Then, $\gamma^{\star}(B)=5$ and hence
either $\gamma(b_{p}t^1_1)$ or $\gamma(b_{q}t^3_1)$ is the $6$th required color, where   $b_p$ and $b_q$ are as in (N2).    
\end{proof}

%%%%%%%%%%%%%%%%%%%%%%%%%%%%%%%%%%%%%%%%%%%%%%%%%%%%%%%%%%%%%%%%%%%%%%%% 
%%%%%%%%%%%%%%%%%%%%%%%%%%%%%%%%%%%%%%%%%%%%%%%%%%%%%%%%%%%%%%%%%%%%%%%%
%%%%%%%%%%%%%%%%%%%%%%%%%%%%%%%%%%%%%%%%%%%%%%%%%%%%%%%%%%%%%%%%%%%%%%%%
%%%%%%%%%%%%%%%%%%%%%%%%%%%%%%%%%%%%%%%%%%%%%%%%%%%%%%%%%%%%%%%%%%%%%%%%
%%%%%%%%%%%%%%%%%%%%%%%%%%%%%%%%%%%%%%%%%%%%%%%%%%%%%%%%%%%%%%%%%%%%%%%%
%%%%%%%%%%%%%%%%%%%%%%%%%%%%%%%%%%%%%%%%%%%%%%%%%%%%%%%%%%%%%%%%%%%%%%%%  

\section{Technical claims behind the proof of Theorem~\ref{thm:main}}\label{s:lemmas}  

As we mentioned before, the bulk of this paper is the proof that 
\begin{equation}\label{e:coco}
\chi(D(Q))\geq |Q|-2 \mbox{ for any } Q\subseteq X \mbox{ with } |Q|\geq 3.
\end{equation}
 
In this section we do this task. As we shall see, the ``symmetry" of $X$ and the good number of separable subsets that
$X$ contains will play a central role in this task. In particular, in this section we will proof that Inequality~(\ref{e:coco})
holds for about 15 representative (and strategic) subsets of $X$ of cardinality at most $11$.

Roughly speaking, our basic proof technique is the following: given $Q\subset X$, a coloring $\gamma$ of $\qq$, and a nonempty subset 
$\{f_1, \ldots , f_m\}$ of $\qq$ for which the set of colors $\{\gamma(f_1), \ldots ,\gamma(f_m)\}$ is known, we then proceed to find 
an ordered sequence $g_1, \ldots , g_l$ of segments in $\qq \setminus \{f_1, \ldots , f_m\}$ such that 
the number of colors in $\{\gamma(f_1), \ldots ,\gamma(f_m)\} \cup \{\gamma(g_1), \ldots ,\gamma(g_l)\}$ is equal to $|Q|-2$.  
We emphasize that for each $i\in \{1,\ldots ,l\}$, the determination of the unknown color $\gamma(g_i)$ depends on the colors previously fixed, namely
 $$\{\gamma(f_1), \ldots ,\gamma(f_m)\} \cup \{\gamma(g_1), \ldots ,\gamma(g_{i-1})\}.$$ 
 
\begin{note} When we have already fixed $|Q|-3$ colors in this process, and so it remains to prove the needed of the last color, we often write 
 $\gamma(g_j) \circeq c$ to mean that $\gamma(g_j)$ must be equal to the color $c$, 
 as otherwise $\gamma(g_j)$ is precisely the last required color. 
 \end{note}

%%%%%%%%%%%%%%%%%%%%%%%%%%%%%%%%%%%%%%%%%%%%% %%%%%%%%%%%%%%%

\begin{lemma}\label{l:TXT} If $Q$ is a subset of $T_1\cup I \cup T_2$ with $I\in \{A, B\}$ and $|Q|\geq 3$, then $\chi(D(Q))=|Q|-2$. 
\end{lemma}

\begin{proof} By Proposition~\ref{p:closed}~($i$), it is enough to show the assertion for $Q=T_1\cup I\cup T_2$. 
By Proposition~\ref{p:n-2}, all we need to show is that $\chi(D(Q)) \geq 9$. 
We only discuss the case $I=A$. The case $I=B$ can be handled in a similar way. Thus we assume $Q=T_1\cup A\cup T_2$. 

Let $\gamma$ be an optimal coloring of $D(Q)$. Since $|\gamma(A)|\geq 3$, then we can assume $|\gamma(T_1\cup T_2)|\in \{4,5\}$, as otherwise we are done. 
Similarly, since $|\gamma(Q)| \geq |\gamma(A)|+|\gamma(T_1\cup T_2)|$, then $|\gamma(A)|\in \{3,4\}$.

We claim that if $|\gamma(T_i)|\geq 3$ for some $i\in \{1,2\}$, then $\chi(D(Q))\geq 9$. Indeed, since $T_i$ is a separable subset of $Q$ with respect to $\gamma$,
then $|\gamma(D(Q))|\geq |\gamma(D(Q\setminus T_i))|+|\gamma(D(T_i))|\geq |\gamma(D(Q\setminus T_i))|+3$ by Proposition~\ref{p:closed}~$(ii)$. Since
$|\gamma(D(Q\setminus T_i))|\geq 6$ by Theorem~\ref{thm:mario}, if $i=2$ (respectively, Proposition~\ref{p:BT1} if $i=1$). Then we may assume
that $1\leq |\gamma(T_1)|, |\gamma(T_2)|\leq 2$. 

\vskip 0.1cm
{\sc Case 1}. Suppose that $|\gamma(A)|=3$. Let $a_i, a_j, a_{k}\in A$ be as in (N1). 

(1.1) If $|\gamma(T_1)|=1=|\gamma(T_2)|$, then $|\gamma(T_1\cup T_2)|=5$ and one of $a_{i}t^1_1$ or $a_{j}t^1_2$ provides the 9th color.  

(1.2) If $|\gamma(T_1)|=1$ and $|\gamma(T_2)|=2$, then none of $\gamma(a_{i}t^1_1), \gamma(a_{j}t^2_1),$ and $\gamma(a_{k}t^3_1)$
belongs to $\gamma(A)\cup \gamma(T_1)$, and so $|\gamma(A\cup T_1)|\geq 7$. This inequality together with $|\gamma(T_2)|=2$
 imply the required inequality. 
  
(1.3) If $|\gamma(T_1)|=2$ and $|\gamma(T_2)|=1$, then either $l(e_1)a_i$ or $r(e_1)t^2_2$ provides the $7$th color, where $e_1$ is as in (N3). The last two required colors are 
$\gamma(t^1_2a_{j})$ and  $\gamma(t^3_2a_{k})$. 

(1.4) Suppose that $|\gamma(T_1)|=2=|\gamma(T_2)|$. We need to show the existence of 2 additional co\-lors. Let $e_1$ and $e_2$ be as in (N3), let 
$f_1:=l(e_1)l(e_2)$ and $f_2:=r(e_1)r(e_2)$. We assume w.l.o.g. that $\gamma(T_1)=\{c_4, c_5\}, \gamma(T_2)=\{c_6, c_7\}, \gamma(e_1)=c_4$ and $\gamma(e_2)=c_7$.

We note that if $\gamma(h)\notin \gamma(T_2)$ for some $h\in a_{k}*T_2$, then $h$ and either
$l(e_1){a_i}$ or $r(e_1){a_j}$ provide the required two colors. Then the color of any segment in $a_{k}*T_2$ must be $c_6$ or $c_7$. For $l\in \{1,2,3\}$, let $h_l=a_{k}t^l_2$.

$\bullet$ Suppose that $e_2=t^1_2t^2_2$. Then $\gamma(t^1_2t^3_2)=\gamma(t^2_2t^3_2)=c_6$, $\gamma(h_3)\circeq c_6$ and $\gamma(h_1)\circeq c_7$.  
Suppose first that $\gamma(h_2)=c_7$. If $\gamma(f_2)\neq c_4$, then $f_2$ and $t^1_2a_{j}$ provide the required colors. Then $\gamma(f_2)=c_4$, 
and $f_1$ together with either $l(f_1){a_i}$ or $r(f_1){a_j}$ provide the required colors. 
Suppose now that $\gamma(h_2)=c_6$. We note that $\gamma(t^1_2a_{i})$ and $\gamma(t^2_2a_{j})$ cannot be $c_6$, and moreover, exactly one of them must be $c_7$
as otherwise we are done. Since if $\gamma(t^1_2a_{i})=c_7$, then $t^2_2a_{j}$ together with either $l(e_1){a_i}$ or $r(e_1)t^3_2$ provide the two required colors. Then 
$\gamma(t^2_2a_{j})=c_7$, and so $t^1_2a_{i}$  together with either $l(e_1){t^2_2}$ or $r(e_1){t^3_2}$ provide the two required colors.

$\bullet$ Suppose that $e_2=t^2_2t^3_2$. Then $\gamma(t^1_2t^2_2)=\gamma(t^1_2t^3_2)=c_6$, $\gamma(h_1)= c_6$ and $\gamma(h_3)= c_7$.
 We note that neither $\gamma(t^1_2a_{i})$ nor $\gamma(t^2_2a_{j})$ can be $c_7$, and moreover, exactly one of them must be $c_6$ 
as otherwise we are done. Suppose first that $\gamma(t^1_2a_{i})=c_6$. Then $\gamma(h_2)\circeq c_7$, and so $t^2_2a_{j}$  together with either
$l(e_1)a_i$ or $r(e_1)t^3_2$ provide the two required colors. Suppose now that $\gamma(t^2_2a_{j})=c_6$. 
 Note that if $\gamma(h_2)=c_6$, then $t^1_2a_{j}$ together with either 
 $l(e_1)a_{i}$ or $r(e_1)t^2_2$ provide the two required colors. Then $\gamma(h_2)=c_7$, and hence 
$t^1_2a_{i}$ and some of $f_1$ or $f_2$ give the two required colors. 

$\bullet$ Suppose that $e_2=t^1_2t^3_2$. Then $\gamma(t^1_2t^2_2)=\gamma(t^2_2t^3_2)=c_6$, $\gamma(h_1)= c_7$ and $\gamma(h_3)= c_7$.
  The required colors are given by $t^1_2a_{j}$ and some segment of $l(e_1)a_{i}$ or $r(e_1)t^3_2$.

\vskip 0.1cm
{\sc Case 2}. Suppose that $|\gamma(A)|=4$. 
Since $|\gamma(T_1)|=1=|\gamma(T_2)|$ imply $|\gamma(T_1\cup T_2)|\geq 5$, and we know that $|\gamma(A\cup T_1 \cup T_2)|\geq |\gamma(A)|+|\gamma(T_1 \cup T_2)|$, 
we can assume that some of $|\gamma(T_1)|\geq 2$ or $|\gamma(T_2)|\geq 2$ holds.   

\vskip 0.1cm
(2.1) Suppose $\gamma^{\star}(A) < 5$. Let $a$ be a point of $A$ that is not an apex of $\gamma(A)$.
  
$\bullet$ If $|\gamma(T_1)|=2$ and $|\gamma(T_2)|=1$, then $at^3_2$ together with either $l(e_1)t^1_2$ or $r(e_1)t^2_2$ give the two required colors.

$\bullet$ If $|\gamma(T_1)|=1$ and $|\gamma(T_2)|=2$, then $at^1_1$ together with either $t_1^2l(e_2)$ or $t_1^3r(e_2)$ give the two required colors.
 
$\bullet$  Suppose that $|\gamma(T_1)| = 2 = |\gamma(T_2)|$. Then we need to show the existence of one additional color.
 Let $f_1 := l(e_1)l(e_2)$ and $f_2 := r(e_1)r(e_2)$. We note that  $\{\gamma(f_1), \gamma(f_2)\} = \{\gamma(e_1), \gamma(e_2)\}$ 
 as otherwise we are done. Then $\gamma(f_1) \neq \gamma(e_i)$ for some $i \in \{1, 2\}$. Let $v:=f_1 \cap e_i$. Then $v\in \{t_1^1, t_1^2, t_2^1, t_2^2\}$. 
Since $v\neq t_2^2$ implies that  $\gamma(av)$ is  the required color, we can assume $v=t_2^2$ and hence we must have $e_i=t_2^2t_2^3$. Then  
  $\gamma(t_2^2r(e_1)) \in \{\gamma(f_1), \gamma(f_2)\}$ as otherwise we are done. If $\gamma(t_2^2r(e_1)) = \gamma(f_1)$ 
  (respectively, $\gamma(t_2^2r(e_1)) = \gamma(f_2)$), then $\gamma(l(e_1)a)$ (respectively, $\gamma(t_2^3a)$) is the required color.

\vskip 0.1cm  
(2.2) Suppose that $\gamma^{\star}(A) = 5$. Let $a_{i}$ and $a_{j}$ be as in (N2).

$(2.2.1)$ Suppose that $|\gamma(T_1)|=2$ and $|\gamma(T_2)|=1$. Clearly, $\gamma(t^1_2a_{i}), \gamma(t^3_2a_{j})\notin \gamma(T_2)$, and moreover, 
exactly one of these two must be $\gamma(a_ia_j)$ as otherwise we are done. If $\gamma(t^1_2a_{i})=\gamma(a_ia_j)$, then
$t^3_2a_{j}$ together with either $l(e_1)t^1_2$ or $r(e_1)t^2_2$ give the two required colors. Then $\gamma(t^3_2a_{j})=\gamma(a_ia_j)$, and so
$t^1_2a_{i}$ together with either $l(e_1)t^2_2$ or $r(e_1)t^3_2$ give the two required colors.

\vskip 0.1cm
$(2.2.2)$ Suppose that $|\gamma(T_1)|=1$ and $|\gamma(T_2)|=2$. We need to show the existence of two additional colors. 
Assume w.l.o.g. that $\gamma(T_2)=\{c_6, c_7\}$ and that $\gamma(e_2)=c_7$.

$\bullet$ Suppose that $e_2=t^2_2t^3_2$. Then $\gamma(t^1_2t^2_2)=\gamma(t^1_2t^3_2)=c_6$. 
Clearly, either $\gamma(t^2_1t^2_2)$ or $\gamma(t^3_1t^3_2)$ is the $8$th color $c_8$ and the other one must be $c_7$, as otherwise we are done. 
If $\gamma(t^2_1t^2_2)=c_7$ and $\gamma(t^3_1t^3_2)=c_8$, then either $a_it_1^1$ or $a_jt_1^2$ provides the $9$th color, and if 
 $\gamma(t^3_1t^3_2)=c_7$ and $\gamma(t^2_1t^2_2)=c_8$, then $\gamma(t^1_1a_{i})\circeq \gamma(a_ia_j),~\gamma(t^2_1a_{j})\circeq c_8$ and 
$\gamma(t^3_1t^2_2)\circeq c_7$. These imply that $\gamma(t^3_2a_{j})$ is the $9$th color. 

$\bullet$ Suppose that $e_2=t^1_2t^2_2$. Then $\gamma(t^1_2t^3_2)=\gamma(t^2_2t^3_2)=c_6$.  
Clearly, either $\gamma(t^2_1t^1_2)$ or $\gamma(t^3_1t^2_2)$ is the $8$th color $c_8$ and the other one must be $c_7$, as otherwise we are done. 
Then $\gamma(t^1_1a_{i})\circeq \gamma(a_ia_j)$.
 If $\gamma(t^2_1t^1_2)=c_7$ and $\gamma(t^3_1t^2_2)=c_8$, then $\gamma(t^2_1a_{j})$ is the $9$th color. Suppose now 
 that $\gamma(t^3_1t^2_2)=c_7$ and  $\gamma(t^2_1t^1_2)=c_8$. Then $\gamma(t^2_1a_{j})\circeq c_8,~\gamma(t^3_1t^1_2)\circeq c_7$ and 
 $\gamma(t^3_1t^3_2)\circeq c_6$. These imply that $\gamma(t^2_2a_{j})$ is the $9$th color.  

$\bullet$ Suppose that $e_2=t^1_2t^3_2$. Then $\gamma(t^1_2t^2_2)=\gamma(t^2_2t^3_2)=c_6$. Clearly, either $\gamma(t^2_1t^1_2)$ or $\gamma(t^3_1t^3_2)$  
is the $8$th color $c_8$ and the other one must be $c_7$, as otherwise we are done. Then $\gamma(t^1_1a_{i})\circeq \gamma(a_ia_j)$.
If $\gamma(t^2_1t^1_2)=c_7$ and $\gamma(t^3_1t^3_2)=c_8$, then $\gamma(t^2_1a_{j})$ is the $9$th color. Suppose now that $\gamma(t^3_1t^3_2)=c_7$ and 
$\gamma(t^2_1t^1_2)=c_8$, then $\gamma(t^2_1a_{j})\circeq c_8$ and $\gamma(t^3_1t^1_2)\circeq c_7$. These imply that $\gamma(t^3_2a_{j})$ is the $9$th color.

\vskip 0.1cm
$(2.2.3)$ Suppose that $|\gamma(T_1)|=2=|\gamma(T_2)|$. As above, let $f_1:=l(e_1)l(e_2)$ and $f_2:=r(e_1)r(e_2)$. 
We need to show the existence of one additional color, say $c_9$. 
Since if $|\gamma(T_1\cup T_2)|\geq 5$ there is nothing to prove, we assume $|\gamma(T_1\cup T_2)|=4$. 
Suppose that $\gamma(T_1)=\{c_4, c_5\}, \gamma(T_2)=\{c_6, c_7\}, \gamma(e_1)=c_4$ and $\gamma(e_2)=c_7$. Then
either $\{\gamma(f_1), \gamma(f_2)\}=\{c_4, c_7\}$ or we are done.  Let $u_1:=l(e_1)$ and $u_2:=r(e_1)$. 

$\bullet$ Suppose that $e_2=t^1_2t^3_2$. Then $\gamma(t^1_2t^2_2)=\gamma(t^2_2t^3_2)=c_6$. 
If $\gamma(f_1)=c_7$ and $\gamma(f_2)=c_4$, then either $u_1a_{i}$ or $t_2^3a_{j}$ must be colored with $c_9$. Suppose now that
$\gamma(f_1)=c_4$ and $\gamma(f_2)=c_7$. Then either $\gamma(u_2t^1_2)\in \{c_4, c_7\}$ or $\gamma(u_2t^1_2)$ is the $9$th color. 
If $\gamma(u_2t^1_2)=c_4$, then either $u_1a_{i}$ or $t_2^1a_{j}$ must be colored with $c_9$, and if $\gamma(u_2t^1_2)=c_7$, then 
either $a_{i}t_2^1$ or $a_{j}t_2^3$ must be colored with $c_9$.

$\bullet$ Suppose that $e_2=t^1_2t^2_2$. Then $\gamma(t^1_2t^3_2)=\gamma(t^2_2t^3_2)=c_6$. If $\gamma(f_1)=c_7$ and 
$\gamma(f_2)=c_4$, then either $\gamma(u_1t^2_2)\in \{c_4, c_7\}$ or we are done. 
If $\gamma(u_1t^2_2)=c_7$ then either $u_1a_i$ or $t_2^1a_j$ must be colored with $c_9$, and if
$\gamma(u_1t^2_2)=c_4$ then $\gamma(u_1a_i)\circeq \gamma(a_{i}a_{j})$ and either $t_2^2a_j$ or $t_2^3u_2$ is colored with $c_9$. 
Suppose now that $\gamma(f_1)=c_4$ and $\gamma(f_2)=c_7$. Then either $\gamma(u_2t^1_2)\in \{c_4, c_7\}$ or we are done. 
If $\gamma(u_2t^1_2)=c_4$ then either $u_1a_{i}$ or $t^1_2a_{j}$ must be colored with $c_9$, and if $\gamma(u_2t^1_2)=c_7$, then 
$\gamma(a_{i}t^1_2)\circeq \gamma(a_ia_j), \gamma(a_{j}t^2_2)\circeq c_6$, and so $\gamma(u_2t^3_2)=c_9$.

$\bullet$ Suppose that $e_2=t^2_2t^3_2$. Then $\gamma(t^1_2t^2_2)=\gamma(t^1_2t^3_2)=c_6$. If $\gamma(f_1)=c_7$ and 
$\gamma(f_2)=c_4$, then either $u_1a_{i}$ or $t^3_2a_{j}$ must be colored with $c_9$. Suppose now that 
that $\gamma(f_1)=c_4$ and $\gamma(f_2)=c_7$. Then either $\gamma(u_2t^2_2)\in \{c_4, c_7\}$ or we are done. 
If $\gamma(u_2t^2_2)=c_4$, then $\gamma(u_1a_{i})\circeq \gamma(a_ia_j), \gamma(t^2_2a_{j})\circeq c_6$, and so $\gamma(u_1t^1_2)=c_9$. If
$\gamma(u_2t^2_2)=c_7$, then $\gamma(t^3_2a_{j})\circeq \gamma(a_ia_j), \gamma(t^2_2a_{i})\circeq c_6, \gamma(u_2t^1_2)\circeq c_4$, 
and so $\gamma(u_1a_{i})=c_9$.
\end{proof}

%%%%%%%%%%%%%%%%%%%%%%%%%%%%%%%%%%%%%%%%%%%%% %%%%%%%%%%%%%%%
\begin{cor}\label{c:3,4} 
Let $\gamma$ be a  coloring of $D(X)$. If $I\in \{A,B\}$ and  $|\gamma(I)|\geq 5$ then $|\gamma(D(X))|\geq 14$. 
\end{cor}

\begin{proof} Since $I$ is a separable subset of $X$, then Lemma~\ref{l:TXT} imply $|\gamma(D(X))|\geq |\gamma(X\setminus I)|+|\gamma(I)|\geq 9+5$. 
\end{proof}

%%%%%%%%%%%%%%%%%%%%%%%%%%%%%%%%%%%%%%%%%%%%% %%%%%%%%%%%%%%%
\begin{lemma}\label{thm:luis} Let $A'\subset A, B'\subset B$ and $T'\subset (T_1\cup T_2)$ be such that $|A'|=|B'|=|T'|=3$. 
If $Q$ is a subset of $A'\cup T'\cup B'$ with $|Q|\geq 3$, then $\chi(D(Q))=|Q|-2$. 
\end{lemma}
\begin{proof} By Proposition~\ref{p:closed}($i$), it is enough to show the assertion for $Q:=A'\cup T'\cup B'$. 
By Proposition~\ref{p:n-2}, all we need to show 
is that $\chi(D(Q))\geq 7$. On the other hand, it is not hard see that if $T'$ is concave up (respectively, concave down) then 
$Q\sim A'\cup T_2 \cup B'$ (respectively, $Q\sim A'\cup T_1 \cup B'$). Then, w.l.o.g. we can assume that $Q$ is either $A'\cup T_1 \cup B'$
or $A'\cup T_2 \cup B'$. We only analyze the case $T'=T_1$. The case $T'=T_2$ can be handled in a similar way. 

Let $\gamma$ be a  coloring of $D(Q)$. Clearly, each of $A', B'$ and $T_1$ is a separable subset of $Q$. 
 Since $|\gamma(T_1\cup B')|\geq 4$ by Lemma~\ref{l:TXT}, we can assume $1\leq |\gamma(A')|\leq 2$, as otherwise we are done. 
Similarly, from $(A'\cup B' )\sim C_{3,3}\sim (A'\cup T_1)$ and  $|\gamma(C_{3,3})|\geq 4$ (by Theorem~\ref{thm:mario}), we can conclude
that $1\leq |\gamma(T_1)|, |\gamma(B')|\leq 2$. Let $a_i, a_j, a_k$ (resp. $b_p, b_q, b_r$) be the points of $A'$ (respectively, $B'$) with $i<j<k$ (resp. $p<q<r$).
 
\vskip 0.1cm
{\sc Case 1}. Suppose that $|\gamma(T_1)|=1$. 

$(1.1)$ Suppose that $|\gamma(A')|=1$. Then, $\gamma$ assigns distinct colors to $a_{i}t^1_1,$ $a_{j}t^2_1,$ and $a_{k}t^3_1$ and hence 
$|\gamma(A'\cup T_1)|\geq 5$. Since $|\gamma(B')|\geq 2$ implies the required inequality, we can assume that $|\gamma(B')|=1$.
Then either $a_ib_p$ or $t^1_1b_q$ provides the 7th color.   

$(1.2)$ Suppose that $|\gamma(A')|=2$ and $|\gamma(B')|=1$. Let $a'$ be the segment of $A'$ that has a different color than the other two segments of $A'$.
Since $|\gamma(T_1)|=1$, we need to show the existence of 3 additional colors.
Clearly, $c_1:=\gamma(b_pt^1_1)$ and $c_2:=\gamma(b_rt_1^3)$ are 2 new colors. Then $\gamma(t_1^1b_q)\circeq c_1$ and so either
$l(a')b_p$ or $r(a')t_1^2$ provides the 7th color.    
 
$(1.3)$ Suppose that $|\gamma(A')|=2$ and $|\gamma(B')|=2$. Let $a'$ (resp. $b'$) be the segment of $A'$ (resp. $B'$) that has a different color than the other two 
segments of $A'$ (resp. $B'$). Since $|\gamma(T_1)|=1$, we need to show the existence of 2 additional colors, say $c_6$ and $c_7$. 
Clearly, a segment of $l(b')t_1^1$ or $r(b')t_1^3$ provides the $6$th color $c_6$ and the other one must be colored with $\gamma(b')$, as otherwise we are done. If
$\gamma(l(b')t_1^1)=c_6$ then either $l(a')t_1^2$ or $r(a')t_1^3$ is colored with $c_7$. If  $\gamma(r(b')t_1^3)=c_6$ 
then either $l(a')t_1^1$ or $r(a')t_1^2$ is colored with $c_7$.

\vskip 0.1cm
{\sc Case 2}. Suppose that $|\gamma(T_1)|=2$. Let $e_1$ be as in (N3).
Suppose that $\gamma(T_1)=\{c_1, c_3\}$ and that $\gamma(e_1)=c_1$.
 
$(2.1)$ Suppose that $|\gamma(A')|=1$. If $|\gamma(B')|=1$ then $|\gamma(A'\cup B')|\geq 5$, and so 
$|\gamma(Q)|\geq 7$, as required. Thus, we can assume that $|\gamma(B')|=2$. Let $b'$ be the segment of $B'$ that has a different color than the other two 
segments of $B'$. Then, we need to show the existence of 2 additional colors, say $c_6$ and $c_7$.

 Clearly, a segment in $\{a_il(b'), a_kr(b')\}$ is colored with the $6$th color $c_6$ and the other one must be colored with $\gamma(b')$, as otherwise we are done. 
Let  $[]$ be the quadrilateral formed by $a_i, a_k, l(b')$ and $r(b')$. 

Suppose that $\gamma(a_il(b'))=c_6$ and $\gamma(a_kr(b'))=\gamma(b')$. If $T_1$ is not on the left of $[]$, then either $l(e_1)a_j$ or $r(e_1)a_k$ provides the 7th color. 

Thus we can assume that $T_1$ lies on the left of $[]$. Then $\gamma(a_jl(b'))\circeq c_6$, $\gamma(l(e_1)a_i)\circeq c_1$, $\gamma(r(e_1)a_i)\circeq c_1$ and
$\gamma(r(e_1)a_j)\circeq c_6$. Then either $a_it_1^m$ or $l(e_1)l(b')$  provides the 7th color, where $t_1^m$ is the point in $T_1$ but not in $e_1$. 

Suppose now that $\gamma(a_il(b'))=\gamma(b')$ and $\gamma(a_kr(b'))=c_6$. If $T_1$ is not on the right of $[]$, then either 
 $\gamma(l(e_1)a_i)$ or $\gamma(r(e_1)a_j)$ provides the 7th color. Thus we can assume that $T_1$ 
 lies on the right of $[]$. Then $\{\gamma(l(e_1)a_j),  \gamma(r(e_1)a_k)\}=\{c_1, c_6\}$, as otherwise we are done. 
 
$\bullet$  If  $\gamma(l(e_1)a_j)=c_1$ and $\gamma(r(e_1)a_k)=c_6$, then $a_jr(b')$ provides the 7th color. 

$\bullet$ If  $\gamma(l(e_1)a_j)=c_6$ and $\gamma(r(e_1)a_k)=c_1$, then $\gamma(a_jr(b'))\circeq c_6$ and $\gamma(l(e_1)a_k)\circeq c_1$. If 
$t_1^3\in e_1$, then $t_1^3r(b')$ provides the 7th color. If $t_1^3\notin e_1$, then $e_1=t_1^1t_1^2$ and $\gamma(t_1^1r(b'))\circeq c_6$, and
 either $t_1^2r(b')$ or $t_1^3a_k$ provides the 7th color. 

$(2.2)$ Suppose that $|\gamma(A')|=2$ and $|\gamma(B')|=1$. Let $a'$ be the segment of $A'$ 
that has a different color than the other two segments of $A'$.
We need to show the existence of 2 additional colors, say $c_6$ and $c_7$. Clearly, a segment in 
$\{l(a') b_p, r(a')b_r\}$ is colored with the $6$th color $c_6$ and the other one must be 
colored with $\gamma(a')$, as otherwise we are done. Let  $[]$ be the quadrilateral formed by $b_p, b_r, l(a')$ and $r(a')$.  

Suppose that $\gamma(l(a')b_p)=\gamma(a')$ and $\gamma(r(a')b_r)=c_6$. Suppose first that  $T_1$ lies inside of $[]$. 

$\bullet$ If  $e_1 = t_1^1t_1^2$, then $\gamma(t_1^1b_p) \circeq c_1$ and $\gamma(r(a')t_1^2) \circeq c_6$. Then either $b_qt_1^2$ or $b_rt_1^3$ is colored with $c_7$.

$\bullet$  If $e_1 = t_1^1t_1^3$, then $\gamma(t_1^1b_q) \circeq c_1$. Then either  $b_pt_1^1$ or $b_qt_1^3$ is colored with $c_7$.

$\bullet$ If $e_1 = t_1^2t_1^3$, then $\gamma(t_1^3b_q) \circeq c_1, \gamma(r(a')t_1^2) \circeq c_6$ and $\gamma(t_1^1b_p) \circeq c_3$. Then either 
$b_qt_1^2$ or $b_rt_1^3$ is colored with $c_7$.

Suppose now that $T_1$ lies in  the exterior of $[]$. Then $T_1$ lies on the right of $[]$ and $\gamma(r(a')b_q) \circeq c_6$.

$\bullet$ If  $e_1 = t_1^1t_1^2$, then $\gamma(b_rt_1^1) \circeq c_1$ and  $\gamma(b_rt_1^3)\circeq c_3$. Then $\gamma(b_rt_1^2) \in \{c_1, c_3\}$ or we are done. If
 $\gamma(b_rt_1^2) = c_1$, then either $r(a')t_1^2$ or $b_qt_1^1$ must be colored with $c_7$. If $\gamma(b_rt_1^2) = c_3$, then 
 $\gamma(r(a')t_1^2) \circeq c_6, \gamma(r(a')t_1^3) \circeq c_6, \gamma(b_qt_1^2) \circeq c_1$ and so $b_pt_1^1$ must be colored with $c_7$.

$\bullet$ If  $e_1 = t_1^1t_1^3$, then $\gamma(t_1^1b_r) \circeq c_1$ and $\gamma(t_1^3b_r) \circeq c_1$. Then either  $b_pt_1^1$ or $r(a')t_1^3$  is colored with $c_7$.

$\bullet$  If  $e_1 = t_1^2t_1^3$, then $\gamma(t_1^3b_r) \circeq c_1$, 
$\gamma(r(a')t_1^2) \circeq c_6, \gamma(t_1^1b_p) \circeq c_3$ and $\gamma(t_1^2b_r) \circeq c_1$. Then either $r(a')t_1^3$ or $b_qt_1^2$ is colored with $c_7$.

Suppose now that $\gamma(l(a')b_p)=c_6$ and $\gamma(r(a')b_r)=\gamma(a')$. 

$\bullet$  If $e_1=t_1^1t_1^2$, then $\gamma(t_1^3b_r)\circeq c_3, \gamma(t_1^1b_q)\circeq c_1,  \gamma(t_1^2b_q)\circeq c_1,  \gamma(t_1^2b_r)\circeq c_3$ and
$\gamma(t_1^3 r(a'))\circeq \gamma(a')$. These imply that either $t_1^2l(a')$ or $t_1^1b_p$ must be colored with $c_7$. 

$\bullet$  If $e_1=t_1^1t_1^3$, then either $t_1^1b_q$ or $t_1^3b_r$ is colored with $c_7$. 

$\bullet$  If $e_1=t_1^2t_1^3$, then $\gamma(t_1^3b_r)\circeq c_1, \gamma(t_1^1b_q)\circeq c_3,  \gamma(t_1^2b_q)\circeq c_3,  \gamma(t_1^2b_r)\circeq c_1$ and
$\gamma(t_1^3 r(a'))\circeq \gamma(a')$. These imply that either $t_1^1b_p$ or $t_1^2l(a')$ is colored with $c_7$. 

$(2.3)$ Suppose that $|\gamma(A')|=2$ and $|\gamma(B')|=2$. Since $|\gamma(T_1)|=2$, we need to show the existence of 1 additional color. Let 
$a'$ (resp. $b'$) be the segments of $A'$ (resp. $B'$) that has a different color than the other two segments of $A'$ (resp. $B'$).
We note that $\{\gamma(a'), \gamma(b')\}=\{\gamma(l(a')l(b')), \gamma(r(a')r(b'))\}$ as otherwise we are done. 

Let $[]$ be the convex quadrilateral formed be the endpoints of $a'$ and $b'$, and let $x$ (resp. $y$) be the endpoint of $a'$ (resp. $b'$) that is
 incident with both colors $\gamma(a')$ and $\gamma(b')$. Then, either $x=r(a')$ or $y=r(b')$ holds. Suppse that $T_1$ lies in the interior of $[]$. 
 If $t_1^3\in e_1$, then one of $xl(e_1)$ or $yt_1^3$ provides the required color. If $t_1^3\notin e_1$, then $e_1=t_1^1t_1^2$ and one of $xt_2^2$ or $yt_1^1$ provides the required color.  Thus, we can assume that $T_1$ lies outside of $[]$. Let $t_1^m$ be the point
  of $T_1$ that does not belong to $e_1$. 

$\bullet$ Suppose that $x$ is the endpoint of $a'$ that is closest to $t^2_1$. Then the triangle $\Delta(e_1,x)$ must be colored with $c_1$ and
$\gamma(xt_1^m)=c_3$, as otherwise we are done. These imply that either $l(e_1)l(b')$ or $r(e_1)r(b')$ provides the required color.

$\bullet$ Suppose finally that $x$ is the endpoint of $a'$ that is farthest to $t^2_1$. This implies that $y$ is the endpoint of $b'$ that is closest to $t^2_1$. 

$-$ If $e_1=t^1_1t_1^3$, then $\gamma(\Delta(e_1,y))=c_1$ and either $l(a')t_1^1$ or $r(a')t_1^3$ provides the 7th color.

$-$ If $e_1=t^1_1t_1^2$, then $\gamma(t_1^1y)\circeq c_1$ and either $\gamma(t_1^2y)\in \{c_1, c_3\}$ or we are done. If $\gamma(yt_1^2)=c_1$ 
(respectively, $\gamma(yt_1^2)=c_3$) then either $l(a')t_1^1$ or $r(a')t_1^2$ (respectively, $l(a')t_1^2$ or $r(a')t_1^3$) provides the required color. 

$-$ If $e_1=t^2_1t^3_1$, then $\gamma(t_1^3y)\circeq c_1$ and either $\gamma(yt^2_1)\in \{c_1, c_3\}$ or we are done. If $\gamma(yt^2_1)=c_1$ 
(respectively, $\gamma(yt^2_1)=c_3$) then either $l(a')t_1^2$ or $r(a')t_1^3$ (respectively,  $l(a')t^1_1$ or $r(a')t^2_1$) provides the 7th color.
\end{proof} 

\begin{lemma}\label{l:/TT} If $a\in A, b\in B$ and $Q:=T_1 \cup \{a, b\} \cup T_2$, then $\chi(D(Q)|\geq 6$.   
\end{lemma}

\begin{proof}  Let $\gamma$ be an optimal coloring of $\qq$. By rotating $Q$ an angle $\pi$ around the origin, if necessary, we may assume w.l.o.g that 
$T_1\cup T_2$ lies on the right semiplane of the line spanned by $ab$. Similarly, we may assume that 
$\gamma(ab)=c_0$. Since if $|\gamma(T_1\cup T_2)|\geq 5$ there is nothing to prove, we can assume that 
$|\gamma(T_1\cup T_2)|=4$ by Corollary~\ref{c:neq6}. Thus, we need to show the existence of one additional color, say $c_6$.

We note that if $|\gamma(T_1)|=1=|\gamma(T_2)|$, then $\gamma(T_1),~\gamma(T_2),~\gamma(t^1_1t^1_2),~\gamma(t^2_1t^2_2)$ and $\gamma(t^3_1t^3_2)$
are pairwise distinct, contradicting that $|\gamma(T_1\cup T_2)|=4$. Thus, we can assume that $|\gamma(T_1)|+|\gamma(T_2)|\geq 3$.

$(1)$  Suppose that $|\gamma(T_{j})|=2$ and $|\gamma(T_{3-j})|=1$ for some $j\in \{1,2\}$. Then there is a unique $e\in\{t^1_1t^1_2, t^2_1t^2_2,\- t^3_1t^3_2\}$
such that $\gamma(e)\notin \gamma(T_1)\cup  \gamma(T_2)$. 

$(1.1)$ Suppose that $j=1$. Then $|\gamma(T_1)|=2$ and $|\gamma(T_2)|=1$. Note that if $e=t^1_1t^1_2$, then one of $\gamma(t^3_2a)$ or $\gamma(t^2_2b)$ 
must be $c_6$. Similarly, if $e=t^2_1t^2_2$, then one of $\gamma(t^1_2a)$ or $\gamma(t^3_2b)$ must be $c_6$. Thus we may assume that $e=t^3_1t^3_2$.
Then $\gamma(t^1_2a)\circeq c_0$, $\gamma(t^3_2b)\circeq \gamma(e)$, and either $l(e_1)t^1_2$ or $r(e_1)t^2_2$ must be colored with $c_6$. 

$(1.2)$ Suppose that $j=2$. Then $|\gamma(T_1)|=1$ and $|\gamma(T_2)|=2$. Note that if $e=t^2_1t^2_2$, then one of $\gamma(t^1_1a)$ or $\gamma(t^3_1b)$ 
must be $c_6$. Similarly, if $e=t^3_1t^3_2$, then one of $\gamma(t^2_1a)$ or $\gamma(t^1_1b)$ must be $c_6$. Thus we may assume that $e=t^1_1t^1_2$. 
Then $\gamma(t^3_1b)\circeq c_0$, $\gamma(t^1_1a)\circeq \gamma(e)$, and so either $l(e_2)t^2_1$ or $r(e_2)t^3_1$ must be colored with  $c_6$. 

$(2)$ Suppose that $|\gamma(T_1)|=|\gamma(T_2)|=2$.
 Suppose that $\gamma(T_1)=\{c_1, c_3\}, \gamma(T_2)=\{c_2, c_4\}, \gamma(e_1)=c_1$ and $\gamma(e_2)=c_4$.
 From $|\gamma(T_1\cup T_2)|=4$ it follows that there is a monochromatic triangle $\Delta$
 with vertices in the endpoints of $e_1$ and $e_2$. Then exactly one of $e_1\in \Delta$ or $e_2\in \Delta$ holds, and so $\gamma(\Delta)\in \{c_1, c_4\}$.   

$\bullet$ Suppose $e_1\in \Delta$. Then $\gamma(\Delta)=c_1$. If $e_1=t^1_1t^3_1$, then either $\gamma(t^1_1a)$ or $\gamma(t^3_1b)$ must be $c_6$. 
If $e_1=t^2_1t^3_1$, then either $\gamma(t^2_1a)$ or $\gamma(t^3_1b)$ must be $c_6$. Similarly, if 
$e_1=t^1_1t^2_1$, then $\gamma(t^1_1b)\circeq c_0$, $\gamma(t^1_1a)\circeq c_0,~\gamma(t^2_1b)\circeq c_3$, and so $\gamma(t^3_1a)=c_6$.

$\bullet$ Suppose $e_2\in \Delta$. Then $\gamma(\Delta)=c_4$.  If $e_2=t^1_2t^3_2$, then either $\gamma(t^1_2a)$ or $\gamma(t^3_2b)$ must be $c_6$. 
If $e_2=t^1_2t^2_2$, then either $\gamma(t^1_2a)$ or $\gamma(t^2_2b)$ must be $c_6$. Similarly, if $e_2=t^2_2t^3_2$ then $\gamma(t^3_2b)\circeq c_0$, $\gamma(t^3_2a)\circeq c_0,~\gamma(t^2_2a)\circeq c_2$, and so $\gamma(t^1_2b)=c_6$.
\end{proof}

%%%%%%%% %%%%%%%% %%%%%%%% %%%%%%%% %%%%%%%% %%%%%%%% %%%%%%%% %%%%%%%%

 \begin{cor}\label{c:AB>=8} Let $\gamma$ be an optimal coloring of $D(X)$. If there are $a\in A$ and $b\in B$
such that $|\gamma(A\cup B \setminus \{a, b\})|\geq 8$ and $\gamma(A\cup B \setminus \{a, b\})\cap \gamma (T_1\cup \{a, b\}\cup T_2) = \emptyset$, then  
$|\gamma(X)|\geq 14$.  
\end{cor}

\begin{proof} From $\gamma(A\cup B \setminus \{a, b\})\cap \gamma (T_1\cup \{a, b\}\cup T_2) = \emptyset$ it follows that 
$|\gamma(X)|\geq |\gamma(A\cup B \setminus \{a, b\})| + |\gamma (T_1\cup \{a, b\}\cup T_2)|\geq 8+6$, due to Lemma~\ref{l:/TT}.  
 \end{proof}  
  
 %%%%%%%% %%%%%%%% %%%%%%%% %%%%%%%% %%%%%%%% %%%%%%%% %%%%%%%% %%%%%%%%
 
\begin{lemma}\label{l:/TTb} Let $a\in A, b\in B$ and $Q:=T_1 \cup \{a, b\} \cup T_2$. If $\gamma$ is a coloring of $D(Q)$ such that
$\gamma(ab)\neq \gamma(g)$ for any vertex (segment) $g$ of $\qq\setminus \{ab\}$, then $|\gamma(Q)|\geq 7$.   
\end{lemma}

\begin{proof}  By rotating $Q$ an angle $\pi$ around the origin, if necessary, we may assume w.l.o.g that 
$T_1\cup T_2$ lies on the right semiplane of the line spanned by $ab$. Similarly, we may assume that 
$\gamma(ab)=c_0$. By Corollary~\ref{c:neq6}($i$) we know that $|\gamma(T_1\cup T_2)|\geq 4$. 
We proceed similarly as in the proof of Lemma~\ref{l:/TT}.

\noindent (1) Suppose that $|\gamma(T_1)|=|\gamma(T_2)|=1$. Assume w.l.o.g. that $\gamma(T_1)=c_1$ and $\gamma(T_2)=c_2$. Then 
$\gamma(t^1_1t^1_2)=c_3, \gamma(t^2_1t^2_2)=c_4$, $\gamma(t^3_1t^3_2)=c_5$, and so either $\gamma(t^1_1b)$ or
 $\gamma(t^1_2a)$ must be the required color. 

\noindent (2)  Suppose that $|\gamma(T_{j})|=2$ and $|\gamma(T_{3-j})|=1$ for some $j\in \{1,2\}$. From $|\gamma(T_1\cup T_2)|\geq 4$ we
know that there exists $e\in\{t^1_1t^1_2, t^2_1t^2_2, t^3_1t^3_2\}$ 
such that $\gamma(e)\notin \gamma(T_1)\cup  \gamma(T_2)$. 

\noindent (2.1) Suppose that $j=1$. Let $\gamma(T_1)=\{c_1,c_3\}, \gamma(T_2)=\{c_2\}, \gamma(e_1)=c_1$ and $\gamma(e)=c_4$.
 If $e=t^1_1t^1_2$, then $\gamma(t^3_2a)$ and $\gamma(t^2_2b)$ are the required colors, and if 
 $e=t^2_1t^2_2$ then $\gamma(t^1_2a)$ and $\gamma(t^3_2b)$ are the required colors.

Thus we can assume that $e=t^3_1t^3_2$. Then $\gamma(t^1_2a)=c_5$ and $\gamma(t^3_2b)\circeq c_4$.
If $e_1=t^1_1t^2_1$ then $\gamma(t^1_1b)\circeq c_1$ and so $\gamma(t^2_1t^2_2)$ is the required color.
 Then we can assume that $e_1=t^l_1t^3_1$ for some $l\in \{1,2\}$. Then $\gamma(t^l_1t^2_2)\circeq c_1$, 
 $\gamma(t^3_1t^2_2)\circeq c_1$, $\gamma(bt^3_1)\circeq c_4$ and  $\gamma(at^3_2)\circeq c_5$. 
These imply that $\gamma(t^l_1t^1_2)$ is the required color.

\noindent (2.2) Suppose that $j=2$. Let $\gamma(T_1)=\{c_1\}, \gamma(T_2)=\{c_2,c_4\},  \gamma(e_2)=c_2$, and $\gamma(e)=c_3$.
 If $e=t^2_1t^2_2$, then $\gamma(t^1_1a)$ and $\gamma(t^3_1b)$
are the required. Similarly, if $e=t^3_1t^3_2$ then $\gamma(t^2_1a)$ and $\gamma(t^1_1b)$ are the required colors.

Thus we may assume that $e=t^1_1t^1_2$. Then $\gamma(t^3_1b)=c_5$ and  $\gamma(t^1_1a)\circeq c_3$.
If $e_2=t^2_2t^3_2$ then $\gamma(t^2_1t^2_2)\circeq c_2$, and so $\gamma(at^3_2)$ is the required. Then 
  $e_2=t^1_2t^l_2$ for some $l\in \{2,3\}$ and so $\gamma(t^2_1t^1_2)\circeq c_2, \gamma(t^2_1t^l_2)\circeq c_2$
  and $\gamma(at^1_2)\circeq c_3$. These imply that $\gamma(bt^1_1)\circeq c_5$ and so $\gamma(t^3_1t^l_2)$  is the required.  

\noindent (3) Suppose that $|\gamma(T_1)|=2=|\gamma(T_2)|$. Let $\gamma(T_1)=\{c_1, c_3\}, \gamma(T_2)=\{c_2, c_4\}, \gamma(e_1)=c_1,$ and $\gamma(e_2)=c_2$. 
We need to show the existence of two additional colors, say $c_5$ and $c_6$.  

\noindent (3.1) Suppose that $e_1=t^1_1t^2_1$. Then one of $\gamma(bt^1_1)=c_1$ or $\gamma(at^2_1)=c_1$ holds, as otherwise we are done. 

$\bullet$ Suppose that $\gamma(bt_1^1)=c_1$. Then $\gamma(at^2_1)=c_5$. 
Note that if $e_2=t_2^lt^3_2$ for some $l\in \{1,2\}$, then $\gamma(bt_2^l)\circeq c_2$, $\gamma(bt^3_2)\circeq c_2$ and $\gamma(at^3_2)\circeq c_5$. Since these imply
$\gamma(t^2_1t^l_2)=c_6$, we can assume $e_2=t^1_2t^2_2$. Then $\gamma(bt^1_2)\circeq c_2$, $\gamma(bt^2_2)\circeq c_2$, $\gamma(bt^3_2)\circeq c_4$ and
 $\gamma(t^2_1t^1_2)\circeq c_5$, and so $\gamma(at^2_2)$ must be $c_6$. 

$\bullet$ Suppose that $\gamma(at_1^2)=c_1$. Then $\gamma(bt^1_1)=c_5$. 
Note that if $e_2=t^1_2t^l_2$ for some $l\in \{2,3\}$, then $\gamma(at^1_2)\circeq c_2$, $\gamma(bt^l_2)\circeq c_5$ and $\gamma(at^1_1)\circeq c_1$.
Since these imply
 $\gamma(t^2_1t^l_2)=c_6$, we can assume $e_2=t^2_2t^3_2$. Then $\gamma(at^3_2)\circeq c_2$, $\gamma(bt^2_2)\circeq c_5$, $\gamma(at^1_1)\circeq c_1$, and so 
 $\gamma(t^2_1t^2_2)=c_6$. 

\noindent (3.2) Suppose that $e_1=t^1_1t^3_1$. Then either $\gamma(at^1_1)=c_1$ or $\gamma(bt^3_1)=c_1$  holds, as otherwise we are done. 

$\bullet$ Suppose that $\gamma(at_1^1)=c_1$. Then $\gamma(bt^3_1)=c_5$. If $t_2^3\in e_2$, then 
 $\gamma(t_2^3a)\circeq c_2, \gamma(l(e_2)b)\circeq c_5$, $\gamma(t^3_1l(e_2))\circeq c_5$, $\gamma(at^3_1)\circeq c_1$, and 
$\gamma(bt^1_1)=c_6$. If $t_2^3\notin e_2$, then $e_2=t_2^1t_2^2$, and so  $\gamma(t_2^1a)\circeq c_2, \gamma(t_2^2b)\circeq c_5$, $\gamma(t^3_1t_2^2)\circeq c_5$, 
$\gamma(at^3_1)\circeq c_1$, and  $\gamma(bt^1_1)=c_6$.

$\bullet$ Suppose now that $\gamma(b t_1^3)=c_1$. Then $\gamma(at^1_1)=c_5$. Note that if $e_2=t^1_2t^l_2$ for some $l\in \{2,3\}$, then $\gamma(bt^l_2)\circeq c_2$,
 $\gamma(at^1_2)\circeq c_5$, $\gamma(bt^1_1)\circeq c_1$, and so $\gamma(t^3_1t^1_2)=c_6$. Suppose now that $e_2=t^2_2t^3_2$.
 Then $\gamma(bt^2_2)\circeq c_2$ and $\gamma(bt^3_2)\circeq c_2$, $\gamma(at^3_2)\circeq c_5$,  $\gamma(bt^1_1)\circeq c_1$, 
 and hence $\gamma(t^3_1t^2_2)=c_6$.

 \noindent (3.3) Suppose that $e_1=t^2_1t^3_1$.  Then either $\gamma(at^2_1)=c_1$ or $\gamma(bt^3_1)=c_1$  holds, as otherwise we are done. 

$\bullet$ Suppose that $\gamma(at^2_1)=c_1$. Then $\gamma(bt^3_1)=c_5$. Note that if $e_2=t^1_2t^l_2$ for some $l\in \{2,3\}$, then $\gamma(at^1_2)\circeq c_2$,
 $\gamma(bt^l_2)\circeq c_5$, $\gamma(at^3_1)\circeq c_1$, $\gamma(at^1_1)\circeq c_3$, $\gamma(bt^2_1)\circeq c_5$ and so $\gamma(t^3_1t^l_2)=c_6$.  
 Similarly, if $e_2=t^2_2t^3_2$ then $\gamma(at^3_2)\circeq c_2$, $\gamma(bt^2_2)\circeq c_5$,  $\gamma(at^3_1)\circeq c_1$, $\gamma(at^1_1)\circeq c_3$,
 $\gamma(bt^2_1)\circeq c_5$, and so $\gamma(t^3_1t^2_2)=c_6$.  
 
 $\bullet$ Suppose that $\gamma(bt^3_1)=c_1$.  Then $\gamma(at^2_1)=c_5$.  Note that if $e_2=t^l_2t^3_2$ for some $l\in \{1,2\}$, then 
 $\gamma(bt^l_2)\circeq c_2$, $\gamma(bt^3_2)\circeq c_2$ and $\gamma(at^3_2)\circeq c_5$. Since these imply
 $\gamma(t^2_1t^l_2)=c_6$, we can assume $e_2=t^1_2t^2_2$. Then $\gamma(bt^1_2)\circeq c_2$, $\gamma(bt^2_2)\circeq c_2$,
  $\gamma(bt^3_2)\circeq c_4$, $\gamma(t^2_1t^1_2)\circeq c_5$, and so $\gamma(at^2_2)=c_6$. 
 \end{proof}
 
 %%%%%%%%%%%%%%%%%%%%%%%%%%%%%%%%%%%%%%%%%%%%% %%%%%%%%%%%%%%%
 
 \begin{lemma}\label{l:3vs4}
Let $a\in A, b \in B$ and let $T_j$ be the triangle in $\{T_1, T_2\}$ that is closest to $ab$. Let $Q:=T_j\cup \{a, b,t\} $ with $t \in T_{3-j}$ 
and let $\gamma$ be a coloring of $D(Q)$. If $\gamma(ab)\neq \gamma(xt)$ for some fixed $x\in \{a,b\}$ and $\gamma(\ell)\notin \{\gamma(ab), \gamma(xt)\}$
for any $\ell\in \qq \setminus \{ab, xt\}$, then $|\gamma(Q)|\geq 6$.
\end{lemma}

\begin{proof} By rotating $Q$ an angle $\pi$ around the origin and relabeling the points of $Q$, if necessary, we may assume w.l.o.g. that 
$T_j\cup \{t\}$ lies on the right semiplane of the line spanned by $ab$. Then $T_j=T_1$ and $t\in T_2$. 
Let $x\in \{a,b\}$ be as in the statement of lemma, let $y$ be such that $\{x,y\}=\{a,b\}$, and let $\Delta:=\Delta( a,b,t)$. For brevity, let $c_0=\gamma(ab)$, $c_1=\gamma(xt)$ and $c_2=\gamma(yt)$.
Then $\gamma(\Delta)=\{c_0, c_1, c_2\}$. Since $|\gamma(Q)|\geq |\gamma(T_1)|+|\gamma(\Delta)|$, we can
assume $|\gamma(T_1)|\in \{1,2\}$, as otherwise we are done.

\noindent (1) Suppose that $\gamma(T_1)=\{c_4\}$. Then we can always connect the corners of $T_1$ with the corners of $\Delta$ 
(without creating any intersections) by means of three pairwise disjoint segments $s_1, s_2$ and $s_3$. 
Since $\gamma(s_1), \gamma(s_2), \gamma(s_3), c_0, c_1, c_4$ are pairwise distinct, we are done. 
  
\noindent (2) Suppose now that $\gamma(T_1)=\{c_4,c_5\}$. Let $\gamma(e_1)=c_4$. Clearly, $c_0, c_1, c_2, c_4$ and $c_5$ are pairwise distinct.  

\noindent $\bullet$ Suppose that $e_1=t^1_1t^3_1$. Then $\{c_2, c_4\}=\{\gamma(t^1_1y), \gamma(t^3_1t)\}$ or the $6$th color is provided by some of
$t^1_1y$ or $t^3_1t$. That equality implies that some of $t^1_1x$ or $t^3_1x$ provides the $6$th color. 

\noindent $\bullet$ Suppose that $e_1=t^1_1t^2_1$. If $a=x$,
 then $\gamma(t^1_1x)\circeq c_4$, $\gamma(t^2_1x)\circeq c_4$ and either $\gamma(t^1_1y)$ or $\gamma(t^2_1t)$ is the $6$th color. Then $b=x$ and so 
 $\gamma(t^1_1x)\circeq c_4$, $\gamma(t_1^2t)\circeq c_2$, $\gamma(t^1_1y)\circeq c_4$,  $\gamma(t_1^2y)\circeq c_2$, and 
 either $\gamma(t^2_1x)$ or $\gamma(t_1^3t)$ is the $6$th color.

\noindent $\bullet$ Suppose that $e_1=t^2_1t^3_1$. If $a=x$, then $\gamma(t^2_1x)\circeq c_4$, $\gamma(t^3_1x)\circeq c_4$,
 $\gamma(t^3_1y)\circeq c_2$, $\gamma(t^3_1t)\circeq c_2$, and either $\gamma(t^1_1x)$ or $\gamma(t^2_1y)$ is the $6$th color.
Then $b=x$ and so $\gamma(t^3_1x)\circeq c_4$,  $\gamma(t^2_1y)\circeq c_2$, $\gamma(t^2_1t)\circeq c_2$ and $\gamma(t^3_1t)\circeq c_4$ 
 These imply that either $\gamma(t^2_1x)$ or $\gamma(t^1_1y)$ is the $6$th color.
\end{proof}

%%%%%%%%%%%%%%%%%%%%%%%%%%%%%%%%%%%%%%%%%%%%% %%%%%%%%%%%%%%%
 
\begin{lemma}\label{l:4vs5}
Let $a\in A, b \in B$ be such that $T_1$ lies on the right semiplane of the line spanned by $ab$. Let $Q:=T_{1}\cup \{a, b,t_2^i, t_2^j\}$ with $1\leq i<j\leq 3$, and
 let $\Delta':=\Delta(a, t_2^i, t_2^j)$. If $\gamma$ is a  coloring of $D(Q)$, $|\gamma(\Delta')|=1$,  
$\gamma(ab)\neq \gamma(bt_2^j)$ and $\gamma(\ell)\notin \{\gamma(ab), \gamma(bt_2^j)\}$ for any $\ell\in \qq \setminus \{ab, bt_2^j\}$, then $|\gamma(Q)|\geq 7$.
\end{lemma}
\begin{proof} For brevity, let $c_0=\gamma(ab)$, $c_1=\gamma(bt_2^j)$, $c_2=\gamma(at_2^j)$ and $\Delta=\overline{Q}$. 
Then $\gamma(\Delta)=\{c_0, c_1, c_2\}$. 
Let $c_3=\gamma(at^2_1)$, $c_4=\gamma(bt^1_1)$ and $c_5=\gamma(t^3_1t_j^2)$. Clearly, $c_0, c_1, \ldots , c_5$ are pairwise distinct.
We need to show the existence of one additional color.  If $|\gamma(T_1)|=1$, then $\gamma(T_1)\notin \{c_0, c_1, \ldots ,c_5\}$ is
the required color. Then we may assume that $|\gamma(T_1)|\geq 2$ and $\gamma(T_1)\subset \{c_3, c_4, c_5\}$,
 as otherwise we are done. We also note that if  $|\gamma(T_1)|=3$,  then $\gamma(T_1)=\{c_3, c_4, c_5\}$ and $bt_2^i$ provides the required color.
Then we can assume that $|\gamma(T_1)|=2$. If $\gamma(T_1)=\{c_3, c_4\}$, then $\gamma(bt_2^i)\circeq c_5$, $\gamma(at^3_1)\circeq c_3$, 
$\gamma(t^2_1t^3_1)\circeq c_3$,  $\gamma(t^1_1t^2_1)\circeq c_4$ and $\gamma(t^1_1t^3_1)\circeq c_4$. These imply that either 
$\gamma(at^1_1)$ or $\gamma(bt^2_1)$ is the required color. Similarly, if $\gamma(T_1)=\{c_3, c_5\}$, then $\gamma(t^1_1t^3_1)\circeq c_5$, 
$\gamma(bt_2^i)\circeq c_4$, $\gamma(t^1_1t^2_1)\circeq c_3$, $\gamma(at^1_1)\circeq c_3$ and so $\gamma(t^2_1t_2^i)$  is the required color.
Finally, if $\gamma(T_1)=\{c_4, c_5\}$, then $\gamma(bt_2^i)$  is the required color. 
\end{proof}

%%%%%%%%%%%%%%%%%%%%%%%%%%%%%%%%%%%%%%%%%%%%% %%%%%%%%%%%%%%%
Let us introduce some notation that we will use in the rest of the paper. Let $Q\subseteq X$ and let $e\in {\mathcal Q}$. If no segment of $\qq$ crosses $e$, 
we will say that $e$ is {\em clean} in $\qq$.  Similarly, if $p\in X$ is not incident with  $e$, then $d(p,e)$ will be the euclidian distance from $p$ to $e$.  

%%%%%%%%%%%%%%%%%%%%%%%%%%%%%%%%%%%%%%%%%%%%% %%%%%%%%%%%%%%%

\vskip 0.1cm
\begin{definition}\label{d:U}
Let $\gamma$ be a  coloring of $D(X)$ and let $\ell$ be a segment in $A*B$. Then, (1) we shall use $U_{\ell}$ to denote the subset
 (may be empty) of $T_1\cup T_2$ such that $v\in U_{\ell}$ if and only if the two segments that join $v$ with the endpoints of $\ell$ are colored with the 
 same color $c$ and $c\neq \gamma(\ell)$; (2) if $|U_{\ell}|<6$, we let $v_0, v_1, \ldots ,v_s$ be such that $(T_1\cup T_2)\setminus U_{\ell}=\{v_0, v_1, \ldots ,v_s\}$ 
 and  $d(v_i, \ell)<d(v_j, \ell)$ for $i<j$; (3) if $|U_{\ell}|<6$, $\Delta_{\ell}:=\Delta(\ell, v_0)$, and $\ell_1$ and $\ell_2$ will be the other sides of $\Delta_{\ell}$;   
(4) if $|U_{\ell}|<5$, $\Delta_i:=\Delta(\ell_i, v_1)$ for $i=1,2$.  
\end{definition}

%%%%%%%%%%%%%%%%%%%%%%%%%%%%%%%%%%%%%%%%%%%%% %%%%%%%%%%%%%%%

\begin{proposition}\label{p:starvgood} Let $\Delta$ be a triangle with corners in $A\cup B$, let $\ell$ be a side of $\Delta$ such that $\ell\in A*B$, and let
 $Q:=T_1\cup  \Delta \cup T_2$.  If $\gamma$ is an optimal coloring of $D(Q)$ and $\Delta$ contains $T_1\cup T_2$ in its interior, then $\ell$ belongs to a star of 
 $\gamma(Q)$ or $|\gamma(Q)|\geq |Q|-2$.      
\end{proposition}

\begin{proof}  We note that $\ell$ is clean in $\qq$ and that $|Q|=9$. Then we must show that $|\gamma(Q)|\geq 7$. Since 
$|\gamma(Q)|\geq |\gamma(T_1\cup T_2)|+|\gamma(\Delta)|\geq 4+|\gamma(\Delta)|$, then either $|\gamma(\Delta)|\in \{1,2\}$ or we are done.  
\vskip 0.1cm
\noindent{$\bullet$} Suppose that $|\gamma(\Delta)|=1$. Assume first that $\Delta$ has a vertex in $a\in A$
and the other two in $b_p, b_q\in B$ with $p<q$. Since $T_1\cup T_2$ lies in the interior of $\Delta$, then $p\in \{1,2,3\}$ and $q\in \{4,5\}$.
By applying Lemma~\ref{l:/TTb} to $T_1\cup \{a, b_p\}\cup T_2$ we obtain $|\gamma(Q)|\geq 7$, as desired.
The case in which $\Delta$ has an endpoint in $b\in B$ and the other two in $A$ can be handled in a similar way.

\vskip 0.1cm
\noindent{$\bullet$} Suppose that  $|\gamma(\Delta)|=2$. Then two sides, say $\ell_1$ and $\ell_2$, of $\Delta$ receive the same color and they form a star of $\gamma|(Q)$.  
We can assume that $\ell\notin \{\ell_1, \ell_2\}$, as otherwise $\ell$ is in a star of $\gamma(Q)$, as claimed. Let $v$ be the common endpoint of $\ell_1$ and $\ell_2$. 
Since $\ell_1$ and $\ell_2$ are clean in $\qq$, the chromatic class $\gamma(\ell_1)=\gamma(\ell_2)$
is a star of $\gamma(Q)$ with apex in $v$, and hence $|\gamma(Q\setminus \{v\})|=|\gamma(Q)|-1$ by Proposition~\ref{p:closed}~($iii$). 
Since $\ell\in A*B$, then $|\gamma(Q\setminus \{v\})|\geq 6$ by Lemma~\ref{l:/TT} and so $|\gamma(Q)|\geq 7$ as required. 
\end{proof}

%%%%%%%%%%%%%%%%%%%%%%%%%%%%%%%%%%%%%%%%%%%%% %%%%%%%%%%%%%%%

\begin{lemma}\label{l:TDT} Let $\Delta_0$ be a triangle with vertices in $A\cup B$, and let 
$Q\subseteq T_1 \cup \Delta_0 \cup T_2$ with $|Q|\geq 3$. If $\gamma$ is an optimal coloring of $D(Q)$, then $|\gamma(Q)|\geq |Q|-2$.   
\end{lemma}

\begin{proof} By Proposition~\ref{p:closed}~($i$), it is enough to verify the case in which $Q:=T_1 \cup \Delta_0 \cup T_2$. Then, we need to show that 
$|\gamma(Q)|\geq 7$. Since $|\gamma(Q)|\geq |\gamma(T_1\cup T_2)|+|\gamma(\Delta_0)|\geq 4+|\gamma(\Delta_0)|$, we can assume $|\gamma(\Delta_0)|\in \{1,2\}$. 

Suppose first that the 3 points of $\Delta_0$ are in $A$. Since
 $|\gamma(Q)|\leq 6$ implies $D(T_1\cup A\cup T_2)\leq |\gamma(Q)|+2\leq 8$, and this last contradicts Lemma~\ref{l:TXT}, 
 we conclude that $|\gamma(Q)|\geq 7$. An analogous reasoning shows that $|\gamma(Q)|\geq 7$ if  $\Delta_0\subset B$.  Then $\Delta_0$ 
 has at least one vertex in each of $A$ and $B$.
Let $\ell, \ell', \ell''$ be the sides of $\Delta_0$, and assume w.l.o.g. that $\ell, \ell'\in A*B$. Then $\ell''$ has both endpoints in either $A$ or $B$. 

{\sc Case 1}. Suppose that $T_1\cup T_2$ lies in the interior of $\Delta_0$. By Proposition~\ref{p:starvgood} we can assume that $\ell=ab$ belongs to a star with apex $v\in \{a,b\}$, as otherwise we are done. If $v\in \ell''$ then Lemma~\ref{l:/TT} implies $|\gamma(Q\setminus \{v\})|\geq 6$ and so  $|\gamma(Q)|\geq 7$. Similarly, if $v\notin \ell''$ then Lemma~\ref{l:TXT} and Proposition~\ref{p:closed}~$(i)$ imply that $|\gamma(Q\setminus \{v\})|\geq 6$, and so $|\gamma(Q)|\geq 7$.   
 
%%%%%%%%%%%%%%%%%%%%%%%%%%%%%%%%%%%%%%%%%%%%% %%%%%%%%%%%%%%%

{\sc Case 2}. Suppose that $T_1\cup T_2$ lies in the exterior of $\Delta_0$. Assume w.l.o.g. that $\ell=ab$ (with $a\in A$ and $b\in B$) is closer to  $O=(0,0)$ than $\ell'$. Then, 
$\ell'$ and $\ell''$ are clean in $\qq$. 

%%%%%%%%%%%%%%%%%%%%%%%%%%%%%%%%%%%%%%%%%%%%% %%%%%%%%%%%%%%%
\begin{claim}\label{cl:star} If $\gamma(Q)$ has a star with apex in a corner of $\Delta_0$, then $|\gamma(Q)|\geq 7$. 
\end{claim}
\noindent{\em Proof of Claim}~\ref{cl:star}. Suppose that $v$ is a corner of $\Delta_0$ that is apex of a star of $\gamma(Q)$. If such $v\in \ell''$, then 
$|\gamma(Q\setminus \{v\})|\geq 6$ by Lemma~\ref{l:/TT}, and so $|\gamma(Q)|\geq 7$. Similarly, if $v\notin \ell''$ then 
$|\gamma(Q\setminus \{v\})|\geq 6$ by Lemma~\ref{l:TXT} and Proposition~\ref{p:closed}~($i$), and hence  $|\gamma(Q)|\geq 7$. 
\halnospace

\vskip 0.1cm
\noindent (2.1) Suppose that $|\gamma(\Delta_0)|=1$. Then $\ell, \ell', \ell''$ receive the same color, and so
 $\gamma$ and $T_1\cup \{a, b\}\cup T_2$ satisfy the hypotheses of Lemma~\ref{l:/TTb}. Therefore, $|\gamma(Q)|\geq 7$.  
 \vskip 0.1cm
\noindent (2.2) Suppose that $|\gamma(\Delta_0)|=2$. Then, two sides of $\Delta_0$ receive the same color, say $c_2$. Suppose that 
$\gamma(\Delta_0)=\{c_1, c_2\}$. By Corollary~\ref{c:neq6}($i$), $|\gamma(T_1\cup T_2)|\geq 4$. Since if $|\gamma(T_1\cup T_2)|= 5$ there is nothing to prove,
we can assume that $\gamma(T_1\cup T_2)=\{c_3, c_4, c_5, c_6\}$.  
\vskip 0.1cm
\noindent (2.2.1) Suppose that $\gamma(\ell)=c_1$. Since $\ell'$ and $\ell''$ are clean in $\qq$ and $\gamma(\ell')=c_2=\gamma(\ell'')$, then $c_2$ must be a star of 
$\gamma(Q)$ with apex $v=\ell'\cap \ell''$ and we are done by Claim~\ref{cl:star}.
\vskip 0.1cm
\noindent (2.2.2) Suppose that $\gamma(\ell)=c_2$.  

%%%%%%%%%%%%%%%%%%%%%%%%%%%%%%%%%%%%%%%%%%%%% %%%%%%%%%%%%%%%
\begin{claim}\label{cl:l3=c2} We can assume that $\gamma(\ell')=c_1$ and $\gamma(\ell)=\gamma(\ell'')=c_2$. 
\end{claim}
\noindent{\em Proof of Claim}~\ref{cl:l3=c2}. Suppose that $\gamma(\ell)\neq \gamma(\ell'')$. From $\gamma(\ell)=c_2$ it follows that $\gamma(\ell'')= c_1$, and so 
$\gamma(\ell')=c_2$.
 We recall that $\ell''$ has its endpoints in either $A$ or $B$, and that it is clean in $\qq$. It is easy to see that each segment of colored with $c_1$ intersects to $\ell$ and so we can recolor $\ell$ with $c_1$ without affecting the essential properties of $\gamma(Q)$ (with the roles of $c_1$ and $c_2$ 
 in $\Delta_0$ interchanged).  \halnospace

%%%%%%%%%%%%%%%%%%%%%%%%%%%%%%%%%%%%%%%%%%%%% %%%%%%%%%%%%%%%

\begin{claim}\label{cl:holds6} If $\gamma(Q)$ has three distinct stars with apices in $T_1\cup T_2$, then Lemma~\ref{l:TDT} holds. 
\end{claim}
\noindent{\em Proof of Claim}~\ref{cl:holds6}. Let $u_1, u_2, u_3\in T_1\cup T_2$ be the apices of such stars of $\gamma(Q)$. Then 
$|\gamma(Q\setminus \{u_1, u_2, u_3\})|\geq 4$ by Lemma~\ref{thm:luis}, and hence $|\gamma(Q)|\geq 7$.  
\halnospace

\vskip 0.1cm

Let $\ell'$ be as above. Let us define $U_{\ell'}\subset T_1\cup T_2$ as in Definition~\ref{d:U}. From the choice of $\ell'$ it is not hard to see that each point of $U_{\ell'}$ is the 
apex of a proper star of $\gamma(Q)$. By Claim~\ref{cl:holds6}, we may asssume that $|U_{\ell'}|\leq 2$. 
Let $\{v'_0, \ldots, v'_s\}, \Delta_{\ell'}, \ell'_1, \ell'_2$ and $\Delta'_1, \Delta'_2$ be the objects corresponding to $\ell'$ described in Definition~\ref{d:U}. Since $|U_{\ell'}|\leq 2$ all these are well-defined. Let $Q'=Q\setminus U_{\ell'}$. By Proposition~\ref{p:closed}~($iii$), it is enough to show $|\gamma(Q')|\geq 7- |U_{\ell'}|$. 
We note that $\Delta_{\ell'}$ is a separable subset of $Q'$ with respect to $\gamma$. Then 
 $|\gamma(Q')|\geq |\gamma(Q'\setminus \Delta_{\ell'})|+|\gamma(\Delta_{\ell'})|$ by Proposition~\ref{p:closed}~($ii$).

%%%%%%%%%%%%%%%%%%%%%%%%%%%%%%%%%%%%%%%%%%%%% %%%%%%%%%%%%%%%
\begin{claim}\label{cl:xDelta'} $\gamma(\Delta_{\ell'})=\{c_1\}$ or Lemma~\ref{l:TDT} holds. 
\end{claim}
\noindent{\em Proof of Claim}~\ref{cl:xDelta'}. By Claim~\ref{cl:l3=c2} $\gamma(\ell')=c_1$, and so $c_1\in \gamma(\Delta_{\ell'})$. 
Since $Q'\setminus \Delta_{\ell'} \subset T_1\cup I\cup T_2$ for some $I\in \{A,B\}$,  Lemma~\ref{l:TXT} implies 
$|\gamma(Q'\setminus \Delta_{\ell'})|\geq |Q'\setminus \Delta_{\ell'}|-2$, and so $|\gamma(Q')|\geq |Q'\setminus \Delta_{\ell'}|-2+|\gamma(\Delta_{\ell'})|$. We note that 
$|Q'\setminus \Delta_{\ell'}|=9-|U_{\ell'}|-|\Delta_{\ell'}|$. 

If $|\gamma(\Delta_{\ell'})|=|\Delta_{\ell'}|=3$, then $|\gamma(Q')|\geq (9-|U_{\ell'}|-|\Delta_{\ell'}|)-2+|\Delta_{\ell'}|=7-|U_{\ell'}|$, as required. Then we can assume that
$|\gamma(\Delta_{\ell'})|\in \{1,2\}$. Since $|\gamma(\Delta_{\ell'})|=1$
implies $\gamma(\Delta_{\ell'})=\{c_1\}$, as claimed, we must have $|\gamma(\Delta_{\ell'})|=2$. 
From $v'_0\notin U_{\ell'}$ we know that $\gamma(\ell'_1)\neq\gamma(\ell'_2)$. Since $|\gamma(\Delta_{\ell'})|=2$,
we must have $c_1\in \{\gamma(\ell'_1), \gamma(\ell'_2)\}$. In any case, $c_1$ is a star with apex in $\Delta_0$. This last 
and Claim~\ref{cl:star} imply $|\gamma(Q)|\geq 7$.       
\halnospace

%%%%%%%%%%%%%%%%%%%%%%%%%%%%%%%%%%%%%%%%%%%%% %%%%%%%%%%%%%%%
 \begin{claim}\label{cl:ch3} Let $f$ and $g$ be the segments that join $v'_s$ with the endpoints of $\ell''$. Then $\gamma(f)= c_2=\gamma(g)$ or Lemma~\ref{l:TDT} holds. 
\end{claim}
\noindent{\em Proof of Claim}~\ref{cl:ch3}. Let $[]$ be the convex quadrilateral formed by the 3 points of $\Delta_0$ 
together with $v'_s$. Then either $f$ or $g$ is a side of $[]$ and the other one is a diagonal of $[]$. Suppose that $f$ (resp. $g$) is a side (resp. diagonal) 
of $[]$. We start by showing that $|\gamma([])|\in \{2,4\}$ implies $|\gamma(Q')|\geq 7- |U_{\ell'}|$, as required.

Suppose that $|\gamma([])|=2$. Then $\gamma([])=\{c_1,c_2\}$ and $\gamma(f)=c_2$. Then $\gamma(g)\neq c_2$ implies that $c_2$
is a star of $\gamma(Q)$ with apex in a corner $v$ of $\Delta_0$. From this  and Claim~\ref{cl:star} we can have that $|\gamma(Q)|\geq 7$ and so 
$|\gamma(Q')|\geq 7- |U_{\ell'}|$. 

If $|\gamma([])|=4$ then $[]$ is a separable subset of $Q'$, and  
Lemma~\ref{l:TXT} implies $|\gamma(Q'\setminus [])|\geq 9-|U_{\ell'}|-|[]|-2$, and hence 
$|\gamma(Q')|\geq |\gamma(Q'\setminus [])|+|\gamma([])|=(9-|U_{\ell'}|-|[]|-2)+|[]|=7-|U_{\ell'}|$. Thus we can assume that $\gamma([])=\{c_1, c_2, c_3\}$ for some $c_3\in \gamma(Q')\setminus \{c_1,c_2\}$. From the difinition of $f$ it is not hard to see that 
$\ell', \ell'',f$ are consecutive in $[]$ and appear in this cyclic order. Then $\ell'\cap f=\emptyset$ and $\gamma(\ell')=c_1$ imply 
$\gamma(f)\neq c_1$. 

Let $h$ be the side of $[]$ that joins $v'_s$ with $v=\ell\cap \ell'$. From Claim~\ref{cl:xDelta'} and $v'_0\neq v'_s$ it follows that $\gamma(h)\neq c_1$. Similarly, 
 $h\cap \ell''=\emptyset$ and $\gamma(\ell'')=c_2$ imply $\gamma(h)\neq c_2$. These and $\gamma([])=\{c_1, c_2, c_3\}$ imply $\gamma(h)= c_3$. 
 From $v'_s\notin U_{\ell'}$ we have $\gamma(f)\neq c_3$ and so $\gamma(f)=c_2$. We note that if $\gamma(g)\neq c_2$, then $c_2$ 
 must be a star of $\gamma(Q)$ with apex in $\Delta_0$, and so we are done by Claim~\ref{cl:star}. 
\halnospace
%%%%%%%%%%%%%%%%%%%%%%%%%%%%%%%%%%%%%%%%%%%%% %%%%%%%%%%%%%%%

We are ready to complete the proof of $(2.2.2)$ and so the proof of Lemma~\ref{l:TDT}. Let $a'\in A$ and $b'\in B$ be such that $\ell'=a'b'$. 
We recall that $\Delta'_l=\Delta(\ell'_l,v'_1)$ for $l\in \{1,2\}$. From $v'_1\notin U_{\ell'}$ we know that $\gamma(v'_1a')\neq \gamma(v'_1b')$ and so 
$|\gamma(\Delta'_l)|=3=|\Delta'_l|$ for some $l\in \{1,2\}$.
 
Let $f, g$ and $[]$ be as in the proof of Claim~\ref{cl:ch3}. Then $f$ is the side of $[]$ that joins $v'_s$ with $\ell''$ and $\gamma(f)=c_2=\gamma(g)$. 
Since $f$ is disjoint from any segment in $\Delta'_l$, then $c_2\notin \gamma(\Delta'_l)$. From this and the fact that $\Delta'_l$ is a separable subset of $Q'$
we know that $|\gamma(Q')|\geq |\gamma(Q'\setminus \Delta'_l)|+|\gamma(\Delta'_l)|$.

 Let $v=\ell\cap \ell'$. If $v\in \Delta'_l$, then $Q'\setminus \Delta'_l \subset T_1\cup I\cup T_2$  for some $I\in \{A,B\}$, then Lemma~\ref{l:TXT}  implies 
 $|\gamma(Q'\setminus \Delta'_l)|\geq |Q'\setminus \Delta'_l|-2$. If $v\notin \Delta'_l$, then 
$Q'\setminus \Delta'_l \subset T_1\cup \{a,b\} \cup T_2$ with 
$\ell=ab$, then Lemma~\ref{l:/TTb} and Proposition~\ref{p:closed}~($ii$) imply $|\gamma(Q'\setminus \Delta'_l)|\geq |Q'\setminus \Delta'_l|-2$.
In any case, we have $|\gamma(Q')|\geq |Q'\setminus \Delta'_l|-2+|\gamma(\Delta'_l)|$. Since  
 $|Q'\setminus \Delta'_l|=9-|U_{\ell'}|-|\Delta'_l|$ we have $|\gamma(Q')|\geq (9-|U_{\ell'}|-|\Delta'_l|) - 2+|\gamma(\Delta'_l)|=7-|U_{\ell'}|$, as required.
\end{proof}

%%%%%%%%%%%%%%%%%%%%%%%%%%%%%%%%%%%%%%%%%%%%% %%%%%%%%%%%%%%%

\begin{lemma}\label{l:1T3} Let $\Delta_U$ be a triangle with corners in $U\in \{A,B\}$, $y\in (A\cup B)\setminus U$ and $Q\subseteq T_1 \cup \Delta_U \cup \{y\} \cup T_2$.
If $\gamma$ is an optimal coloring of $D(Q)$, then $|\gamma(Q)|\geq |Q|-2$.   
\end{lemma}
\begin{proof} By Proposition~\ref{p:closed}~($i$), it is enough to verify the case in which $Q=T_1 \cup \Delta_U \cup \{y\} \cup T_2$. 
Then, we need to show that $\gamma(Q)\geq 8$. We can assume that $\gamma$ does not have a star with apex in a point $v\in \Delta_U  \cup \{y\}$. Indeed, if such $v$ exists the required inequality follows by applying Lemma~\ref{l:TDT} and Proposition~\ref{p:closed}~($iii$) to $Q\setminus \{v\}$. Let $u_{i}, u_{j}, u_{k}$ be the corners of $\Delta_U$ and suppose that $i<j<k$. 

Let $\Delta:=\Delta(u_{i}, u_{k}, y)$. Then $u_{j}$ lies in the interior of $\Delta$. Note that $\Delta_U$ is a separable subset of $Q$ and that its segments are clean in $\qq$. 
Since $\gamma(\Delta_U)=2$ implies that $\gamma$ has a star with apex
in a point of $\Delta_U$, we must have that $|\gamma(\Delta_U)|\in \{1,3\}$. If $|\gamma(\Delta_U)|=3$ then $|\gamma(Q)|\geq |\gamma(Q\setminus \Delta_U)|+3$. Since  
$|\gamma(Q\setminus \Delta_U)|\geq 5$ by Lemma~\ref{l:TXT}, we can assume  $|\gamma(\Delta_U)|=1$.

\vskip 0.1cm
{\sc Case 1}. Suppose that $T_1\cup T_2$ lies in the interior of $\Delta$. Then $\Delta=\overline{Q}$. Since $\gamma(u_{i}y)=\gamma(u_{k}y)$ implies that color $\gamma(u_{i}y)$ induces a star
  with apex in $y$,  we can assume $\gamma(u_{i}y)\neq \gamma(u_{k}y)$. Then $|\gamma(\Delta)|=3$ and so $|\gamma(Q)|\geq |\gamma(Q\setminus \Delta)|+3$. 
Since $|\gamma(Q\setminus \Delta)|\geq 5$ by Lemma~\ref{l:TXT}, we are done.  

\vskip 0.1cm
{\sc Case 2}. Suppose that $T_1\cup T_2$ is in the exterior of $\Delta$. Then we must have that $i=1, j=2$ and $k=3$. Let $\ell=u_{i}y$ and $c_1=\gamma(\ell)$. 
We note that if a segment $f\in \{u_i, y\}* (T_1\cup T_2)$ intersects $\ell$, then $f$ intersects each of $u_jy$ and $u_ky$. Then, w.l.o.g. 
we can assume that $\gamma(u_jy)=c_1=\gamma(u_ky)$. Since $c_1$ cannot be a star, there is a segment $g\in u_i*(T_1\cup T_2)$ such that
$\gamma(g)=c_1$. Let $\Delta':=\Delta(\ell, g)$. If $\gamma(\Delta')\neq \{c_1\}$, then $T_1\cup \{u_j, y\}\cup T_2$ satisfy the hypotheses of Lemma~\ref{l:/TTb} and so
$|\gamma(Q\setminus \{u_i, u_k\})|\geq 7$. Since $\gamma(u_iu_k)$ is not one of these 7 colors, we are done. Thus, we can assume $\gamma(\Delta')=\{c_1\}$.

\vskip 0.1cm
For $\ell=u_{i}y$, let $U_{\ell}$ be as in Definition~\ref{d:U}. From the assertions in previous paragraph it is not hard to see that each point of $U_{\ell}$ is the apex of a 
proper star of $\gamma(Q)$.  By using the argumet in the proof of Claim~\ref{cl:holds6}, we can asssume that $|U_{\ell}|\leq 2$.  
Let $\{v_0, \ldots, v_s\}, \Delta_{\ell}, \ell_1, \ell_2$ and $\Delta_1, \Delta_2$ be as in Definition~\ref{d:U}. Since $|U_{\ell}|\leq 2$, all these are well-defined.
Let $Q'=Q\setminus U_{\ell}$. By Proposition~\ref{p:closed}~($iii$), it is enough to show that $|\gamma(Q')|\geq 8- |U_{\ell}|$. 

From the definition of $U_{\ell}$ we know that $\Delta'\subset Q'$. Suppose first that $\Delta'=\Delta_{\ell}$, and let  
$[]$ be the convex quadrilateral formed by the 3 points of $\Delta$ together with $v_s$. Then $[]=\overline{Q'}$.
From $v_s\notin U_{\ell}$ and $\gamma(\Delta_{\ell})=\{c_1\}$ it follows that $|\gamma([])|=4$. Then $|\gamma(Q')|\geq |\gamma(Q'\setminus [])|+4$. On the other hand, 
Lemma~\ref{l:TXT} implies $|\gamma(Q'\setminus [])|\geq (10-|U_{\ell}|)-4-2=4-|U_{\ell}|$, and hence $|\gamma(Q')|\geq 8 - |U_{\ell}|$.   
Suppose finally that $\Delta'\neq\Delta_{\ell}$. From $v_0\notin U_{\ell}$ it follows that $\gamma(\ell_1)\neq \gamma(\ell_2)$, and so 
$|\gamma(\Delta_{\ell})|=3$. Since $\Delta_{\ell}$ is a separable subset of $Q'$, then $|\gamma(Q')|\geq |\gamma(Q'\setminus \Delta_{\ell})|+3$. 
By Lemma~\ref{l:TXT} we know that $|\gamma(Q'\setminus \Delta_{\ell})|\geq |Q'\setminus \Delta_{\ell}|-2=5-|U_{\ell}|$, as required.  
\end{proof}

%%%%%%%%%%%%%%%%%%%%%%%%%%%%%%%%%%%%%%%%%%%%% %%%%%%%%%%%%%%%

%%%%%%%%%%%%%%%%%%%%%%%%%%%%%%%%%%%%%%%%%%%%% %%%%%%%%%%%%%%%
%%%%%%%%%%%%%%%%%%%%%%%%%%%%%%%%%%%%%%%%%%%%% %%%%%%%%%%%%%%%

\begin{proposition}\label{p:starvgood2} Let $[]$ be a convex quadrilateral with corners in $A\cup B$, let $\ell$ be a side of $[]$ such that $\ell\in A*B$, and let
 $Q\subseteq T_1\cup  [] \cup T_2$ with $|Q|\geq 3$. If $\gamma$ is an optimal coloring of $D(Q)$ and $[]$ contains $T_1\cup T_2$ in its interior, then $\ell$ belongs to a star of 
 $\gamma(Q)$ or $|\gamma(Q)|\geq |Q|-2$.      
\end{proposition}

\begin{proof} By Proposition~\ref{p:closed}~($i$), it is enough to verify the case in which $Q:=T_1 \cup [] \cup T_2$. We note that $\ell$ is clean in $\qq$ and that
$|\gamma([])|\geq 2$. We need to show that $|\gamma(Q)|\geq 8$. Since $|\gamma(Q)|\geq |\gamma(T_1\cup T_2)|+|\gamma([])|\geq 4+|\gamma([])|$, 
we can assume $|\gamma([])|\in \{2,3\}$. 

Let $U_{\ell}$ be as in Definition~\ref{d:U}. From the choice of $\ell$ it is not hard to see that each point of $U_{\ell}$ is the apex of a proper star of $\gamma(Q)$.
Again, by using the argumet in the proof of Claim~\ref{cl:holds6}, we can asssume that $|U_{\ell}|\leq 2$. Then, if 
$\{v_0, \ldots, v_s\}, \Delta_{\ell}, \ell_1, \ell_2$ and $\Delta_1, \Delta_2$ are as in Definition~\ref{d:U}, they are well-defined because $|U_{\ell}|\leq 2$.
Let $Q'=Q\setminus U_{\ell}$. By Proposition~\ref{p:closed}~($iii$), it is enough to show $|\gamma(Q')|\geq 8- |U_{\ell}|$. 

Let $\ell'$ be the opposite side of $\ell$ in $[]$, and note that $c_1:=\gamma(\ell)\in \gamma(\Delta_{\ell})$. 
Since the set of points forming $\Delta_{\ell}$ is a separable subset of $Q'$, then 
 $|\gamma(Q')|\geq |\gamma(Q'\setminus \Delta_{\ell})|+|\gamma(\Delta_{\ell})|$ by Proposition~\ref{p:closed}~($ii$).

Because $Q'\setminus \Delta_{\ell} \subset T_1\cup \{a,b\} \cup T_2$ for $ab=\ell'$, then Lemma~\ref{l:/TT} and 
Proposition~\ref{p:closed}~($ii$) imply $|\gamma(Q'\setminus \Delta_{\ell})|\geq |Q'\setminus \Delta_{\ell}|-2$, and so $|\gamma(Q')|\geq |Q'\setminus \Delta_{\ell}|-2+|\gamma(\Delta_{\ell})|$. We note that  $|Q'\setminus \Delta_{\ell}|=10-|U_{\ell}|-|\Delta_{\ell}|$. 

If $|\gamma(\Delta_{\ell})|=|\Delta_{\ell}|$, then $|\gamma(Q')|\geq (10-|U_{\ell}|-|\Delta_{\ell}|)-2+|\Delta_{\ell}|=8-|U_{\ell}|$, as required. 
Then we must have $|\gamma(\Delta_{\ell})|\in \{1,2\}$. 

\vskip 0.1cm
\noindent{$\bullet$}  Suppose that $|\gamma(\Delta_{\ell})|=1$. From $v_1\notin U_{\ell}$ it follows that $|\gamma(\Delta_l)|=3$ for some $l\in \{1,2\}$.
Then $\Delta_l$ is a separable subset of $Q'$, and hence
 $|\gamma(Q')|\geq |\gamma(Q'\setminus \Delta_l)|+|\gamma(\Delta_l)|$. By  Lemma~\ref{l:/TT} and Proposition~\ref{p:closed}~($i$) we know that
 $|\gamma(Q'\setminus \Delta_l)|\geq |Q'\setminus \Delta_l|-2=(10-|U_{\ell}|)-2-|\Delta_l|$, and hence  
 $|\gamma(Q'\setminus \Delta_l)|\geq (10-|U_{\ell}|)-2-|\Delta_l|+3=8-|U_{\ell}|$, as required.

\vskip 0.1cm
\noindent{$\bullet$}  Suppose that $|\gamma(\Delta_{\ell})|=2$.  From $v_1\notin U_{\ell}$ we know that $\gamma(\ell_1)\neq\gamma(\ell_2)$, and hence 
$c_1\in \{\gamma(\ell_1), \gamma(\ell_2)\}$. In any case, we have that $c_1$ must be a star with apex in an endpoint of $\ell$, as required. 
\end{proof}

%%%%%%%%%%%%%%%%%%%%%%%%%%%%%%%%%%%%%%%%%%%%% %%%%%%%%%%%%%%%

\begin{lemma}\label{l:2T2} Let $a_{i}, a_{j}\in A$, $b_{p}, b_{q}\in B$ and $Q:=T_1 \cup \{a_{i}, a_{j}, b_{p}, b_{q}\} \cup T_2$. If  $\gamma$ is an optimal coloring of $D(Q)$, 
then $|\gamma(D(Q))|\geq 8$.   
\end{lemma}

\begin{proof}   Let $[]$ be the convex quadrilateral
defined by $a_{i}, a_{j}, b_{p}, b_{q}$. W.l.o.g suppose that $i<j$ and $p<q$. We may assume that $\gamma$ does not have a star with apex in a corner
$v$ of $[]$. Indeed, if such $v$ exists then we can deduce the required inequality by applying Lemma~\ref{l:TDT} and Proposition~\ref{p:closed}~($iii$) to 
$Q\setminus \{v\}$. Similarly, we can assume that $T_1\cup T_2$ lies in the exterior of $[]$, as otherwise $|\gamma(Q)|\geq 8$ by Proposition~\ref{p:starvgood2}. 

From $\gamma(a_{i}a_{j})\neq \gamma(b_{p}b_{q})$, we know that $\gamma([])\geq 2$.
Since $\gamma(Q)\geq \gamma(T_1\cup T_2)+\gamma([])$ and $\gamma(T_1\cup T_2)\geq 4$, then either $\gamma([])\in \{2,3\}$ or we are done.
Let $\ell_1, \ell_2, \ell_3, \ell_4$ be the sides of $[]$ and suppose that they appear in this cyclic (clockwise) order. W.l.o.g. let $\ell_1$ (resp. $\ell_3$) be the farthest (resp. closest) segment of $\{a_{i}b_{p}, a_{j}b_{q}\}$ to $(0,0)$. For $i\in \{1,3\}$, let $c_i=\gamma(\ell_i)$. Clearly, $c_1\neq c_3$.

\vskip 0.1cm
{\sc Case 1}. Suppose that  $\gamma([])=2$. By rotating $Q$ an angle $\pi$ around the origin and relabeling the 10 points of $Q$, if necessary, we may assume w.l.o.g that 
$T_1\cup T_2$ lies on the right semiplane of the line spanned by $\ell_3$. Then, either $\gamma(\ell_2)=c_1, \gamma(\ell_4)=c_3$ or 
$\gamma(\ell_2)=c_3, \gamma(\ell_4)=c_1$ hold. 

\vskip 0.1cm
{\sc Case 1.1}. Suppose that $|\gamma(T_2)|=1$. Let $\gamma(T_2)=\{c_2\}$. We first analyze the case $\gamma(\ell_2)=c_1$ and $\gamma(\ell_4)=c_3$.  
Since $\gamma$ has no apices in $[]$, then $\gamma(a_{j}b_{p})=c_1$ and so
any segment colored with $c_1$ and distinct from $a_{j}b_{p}$ must be incident with $a_{i}$. Moreover, we can assume that any segment incident with  
$a_{i}$ is colored with color $c_1$. Indeed, if necessary, we can slightly modify $\gamma$ by recoloring with $c_1$ all such segments. It is not hard to see that 
this change does not affect the essential properties of $\gamma(Q)$.

Clearly, $c_2\notin \{c_1, c_3\}$. Let $\Delta:=\Delta(t^1_2, t^3_2,a_j)$.
From $|\gamma(T_2)|=1$ it follows that $|\gamma(\Delta)|\in \{2,3\}$. Since $\Delta$ is a separable subset of $Q$, then 
$|\gamma(Q)|\geq |\gamma(Q\setminus \Delta)|+ |\gamma(\Delta)|$. Since Lemma~\ref{l:TDT} 
implies $|\gamma(Q\setminus \Delta)|\geq |Q\setminus \Delta|-2$, we can assume that $|\gamma(\Delta)|=2$ or we are done. Let $c_4$ be such that 
$\gamma(\Delta)=\{c_2,c_4\}$. Then the sides $a_{j}t^1_2$ and $a_{j}t^3_2$ of $\Delta$ are both colored with $c_4\notin \{c_1, c_2, c_3\}$. From
$\gamma(a_{i}t^3_2)=c_1$, and the fact that $t^1_2t^3_2$ and $a_{i}t^3_2$ are the only segments intersectig both $a_{j}t^1_2$ and $a_{j}t^3_2$, it follows that 
$c_4$ must be a star with apex $a_{j}\in []$, and we are done. 

Suppose now that $\gamma(\ell_2)=c_3$ and $\gamma(\ell_4)=c_1$. As $\gamma$ has no apices in $[]$, then 
$\gamma(a_ib_q)=c_1$. As before, we can assume that any segment incident with $b_{p}$ is colored with color $c_1$. 
Let $\Delta:=\Delta(t^2_2, t^3_2,b_q)$.  Then, $\Delta$ is a separable subset of $Q$ and so 
$|\gamma(Q)|\geq |\gamma(Q\setminus \Delta)|+ |\gamma(\Delta)|$. Since Lemma~\ref{l:TDT} 
implies $|\gamma(Q\setminus \Delta)|\geq |Q\setminus \Delta|-2$ and $|\gamma(T_2)|=1$, then $|\gamma(\Delta)|=2$ or we are done.
Then $b_{q}t^2_2$ and $b_{q}t^3_2$ are both colored with $c_4\notin \{c_1, c_2, c_3\}$. If $|\gamma(T_1)|=1$, then $t_1^2t_2^1$ and 
$t_1^3t_2^2$ provide the $6$th and $7$th color, and either $a_i t_1^1$ or $a_jt_2^3$ provides the $8$th color. 

Suppose now that $|\gamma(T_1)|=2$ with $\gamma(T_1)=\{c_5, c_6\}$. Clearly, $\{c_1, c_2, \ldots , c_6\}$ are pairwise distinct. 
Let $e_1$ be as in (N3) and suppose that $\gamma(e_1)=c_6$. If $\gamma(a_jt_2^3)\neq c_3$, then $\gamma(a_jt_2^3)$ is the $7$th color, and
 either $l(e_1)t_2^1$ or $r(e_1)t_2^2$ provides the $8$th color. Thus we may assume that $\gamma(a_jt_2^3)=c_3$. Then some of $l(e_1)t_2^1$ or $r(e_1)t_2^2$ provides the $7$th color $c_7$, and the other one must be colored with $c_6$, as otherwise we are done. If $\gamma(l(e_1)t_2^1)=c_6$ and $\gamma(r(e_1)t_2^2)=c_7$, 
 then $a_it_2^1$ provides the $8$th color. Then, we must have $\gamma(r(e_1)t_2^2)=c_6$ and $\gamma(l(e_1)t_2^1)=c_7$, and hence 
 $\gamma(a_it_2^1)\circeq c_7$, $\gamma(l(e_1)t_2^2)\circeq c_6$ and
$\gamma(r(e_1)t_2^3)\circeq c_4$. If $e_1=t_1^2t_1^3$, then either $a_it_1^1$ or $b_qt_1^2$ provides the $8$th required color. Similarly,
if $e_1\neq t_1^2t_1^3$, then $b_qt_1^1$  provides the $8$th required color. 

\vskip 0.1cm
{\sc Case 1.2}. Suppose that $|\gamma(T_1)|=1$ and $|\gamma(T_2)|=2$. Let $\gamma(T_1)=\{c_2\}$ and note that $c_2\notin \{c_1, c_3\}$. 
 
\vskip 0.1cm
\noindent {(1.2.1)}  Suppose that $\gamma(\ell_2)=c_3$ and $\gamma(\ell_4)=c_1$. Since $\gamma$ does not have a star with apex in $[]$, then
 $\gamma(a_{i}b_{q})=c_1$ and as before, we can assume that any segment incident with $b_{p}$ is colored with color $c_1$.

Let $\Delta:=\Delta(t^1_1, t^3_1, b_q)$. From $|\gamma(T_1)|=1$ it follows that $|\gamma(\Delta)|\in \{2,3\}$. 
We note that any segment of $Q$ intersecting $\Delta$ is incident with a point of 
$\Delta\cup \{b_{p}\}$, and so $|\gamma(Q)|\geq |\gamma(Q\setminus \Delta)|+ |\gamma(\Delta)|$. Since Lemma~\ref{l:TDT}
implies $|\gamma(Q\setminus \Delta)|\geq |Q\setminus \Delta|-2$, then we must have $|\gamma(\Delta)|=2$, or we are done. Let $c_4$ be such that 
$\gamma(\Delta)=\{c_2,c_4\}$. Then $b_{q}t^1_1$ and $b_{q}t^3_1$ are both colored with $c_4\notin \{c_1, c_2, c_3\}$. 
Since any segment that is incident with $b_{p}$ is colored with $c_1$, then $c_4$ must be a star with apex $b_{q}$ in a corner of $[]$, and so we are done. 

\vskip 0.1cm 
\noindent {(1.2.2)}  Suppose that $\gamma(\ell_2)=c_1$ and $\gamma(\ell_4)=c_3$. Since $\gamma$ does not have a star with apex in $[]$, then
 $\gamma(a_{j}b_{p})=c_1$ and as before, we can assume that any segment incident with $a_{i}$ is colored with color $c_1$.

Let $\Delta:=\Delta(t^1_1, t^2_1,a_{j})$. From $|\gamma(T_1)|=1$ it follows that $|\gamma(\Delta)|\in \{2,3\}$. 
We note that any segment of $Q$ intersecting $\Delta$ is incident with a point of 
$\Delta\cup \{a_{i}\}$, and so $|\gamma(Q)|\geq |\gamma(Q\setminus \Delta)|+ |\gamma(\Delta)|$. Since Lemma~\ref{l:TDT} 
implies $|\gamma(Q\setminus \Delta)|\geq |Q\setminus \Delta|-2$, then we must have $|\gamma(\Delta)|=2$, or we are done. Let $c_4$ be such that 
$\gamma(\Delta)=\{c_2,c_4\}$. Then $a_{j}t^1_1$ and $a_{j}t^2_1$ are both colored with $c_4\notin \{c_1, c_2, c_3\}$.

Suppose that $\gamma(T_2)=\{c_5, c_6\}$, and let $e_2$ be as in (N3). Suppose that $\gamma(e_2)=c_5$. Clearly, some of $t_1^2l(e_2)$ or $t^3_1r(e_2)$ provides a new color, say $c_7$, and hence  $\gamma(b_{p}t^1_1)\circeq c_3$.  

If $\gamma(t^2_1l(e_2))=c_7$, then $\gamma(t^3_1r(e_2))\circeq c_5$ and $b_{q}t^3_1$ provides the $8$th color. Thus, we may assume that $\gamma(t^3_1r(e_2))=c_7$, and so
$\gamma(t^2_1l(e_2))\circeq c_5$, $\gamma(b_{q}t^3_1)\circeq c_7$, $\gamma(b_{q}r(e_2))\circeq c_7$, $\gamma(t^2_1r(e_2))\circeq c_5$ and
$\gamma(t^1_1l(e_2))\circeq c_4$. If $t^3_2\in e_2$, then $a_{j}t^3_2$ provides the $8$th color. Similarly, if $t^3_2\notin e_2$ then 
some of  $b_{q}t^3_2$ or $a_{j}t^2_2$ provides the $8$th color. This proves Case 1.2.

 \vskip 0.1cm
{\sc Case 1.3}. Suppose that $|\gamma(T_1)|=2= |\gamma(T_2)|$. Let $\gamma(T_1)=\{c_2, c_4\}$ and $\gamma(T_2)=\{c_5, c_6\}$. For $j\in \{1,2\}$
let $e_j$ be as in (N3), and let $\gamma(e_1)=c_4,~\gamma(e_2)=c_6$.   
Clearly, $c_1, c_2, \ldots ,c_6$ are pairwise distinct. In order to show this case, we need the following four claims.

\begin{claim}~\label{cl:3vs4} 
 If $u,v\in T_2$ are such that 
 $\gamma(Q\setminus \{u,v\}) \subseteq \gamma(Q)\setminus \gamma(T_2)$, then $|\gamma(Q)|\geq 8$.
\end{claim}

\noindent{\em Proof of Claim}~\ref{cl:3vs4}.  Suppose that $T_{2}=\{u,v,w\}$. Since $|\gamma(T_2)|=2$ and 
$\gamma(Q\setminus \{u,v\}) \subseteq \gamma(Q)\setminus \gamma(T_2)$, it is enough to show that $|\gamma(Q')|\geq 6$ for $Q'=Q\setminus \{u,v\}$.
We note that Lemma~\ref{thm:luis} guarantees $|\gamma(Q'\setminus \{w\})|\geq 5$. 

Suppose first that $\gamma(\ell_2)=c_1$ and $\gamma(\ell_4)=c_3$. Since $\gamma$ does not have a star with apex in $[]$, then
 $\gamma(a_{j}b_{p})=c_1$ and as before we can assume that any segment incident with $a_{i}$ is colored with $c_1$. 
 
 Let $c_0=\gamma(b_{q}w)$. We note that if $\gamma(a_{j}w)=c_0$, then $c_0$ is star with apex $w$, and so $|\gamma(Q')|\geq |\gamma(Q'\setminus \{w\})|+1=6$ by Proposition~\ref{p:closed}~($iii$) and Lemma~\ref{thm:luis}. Thus, we may assume that 
$\gamma(a_{j}w)\neq c_0$. Since $c_0\notin \gamma(T_1)\cup \{c_1, \gamma(a_{j}w)\}$, then $c_0=c_3$ or $c_0$ is the $6$th required color.
Since $c_3$ cannot be a star of $\gamma$, then $\gamma(b_{p}w)\circeq c_3$. By applying Lemma~\ref{l:3vs4} to $Q''=Q'\setminus \{a_{i}, b_{q}\}$ 
 we have $|\gamma(Q')|\geq |\gamma(Q'')|\geq 6$, as required. 

 Suppose now that $\gamma(\ell_2)=c_3$ and $\gamma(\ell_4)=c_1$. Since $\gamma$ does not have a star with apex in $[]$, then
 $\gamma(a_{i}b_{q})=c_1$ and as before we can assume that any segment incident with $b_{p}$ is colored with $c_1$. 
 Let $c_0=\gamma(a_{j}w)$. If $\gamma(b_{q}w)=c_0$, then $c_0$ is star with apex $w$, and so $|\gamma(Q')|\geq |\gamma(Q'\setminus \{w\})|+1=6$ by Proposition~\ref{p:closed}~($iii$) and Lemma~\ref{thm:luis}. Thus, we may assume that 
$\gamma(b_{q}w)\neq c_0$. Since $c_0\notin \gamma(T_1)\cup \{c_1, \gamma(b_{q}w)\}$, then $c_0=c_3$ or $c_0$ is the $6$th required color.
Since $c_3$ cannot be a star of $\gamma$, then $\gamma(a_{i}w)\circeq c_3$. By applying Lemma~\ref{l:3vs4} to $Q''=Q'\setminus \{a_{j}, b_{p}\}$ 
 we have $|\gamma(Q')|\geq |\gamma(Q'')|\geq 6$, as required. 
\halnospace

\vskip 0.1cm
\noindent {(1.3.1)}  Suppose that $\gamma(\ell_2)=c_1$ and $\gamma(\ell_4)=c_3$. Since $\gamma$ does not have a star with apex in $[]$, then
 $\gamma(a_{j}b_{p})=c_1$ and as before we can assume that any segment incident with $a_{i}$ is colored with color $c_1$.

\begin{claim}\label{cl:thereisapex} $\gamma(Q)$ has a proper star with apex $u\in T_2$ or $|\gamma(Q)|\geq 8$. 
 \end{claim}
 \noindent{\em Proof of Claim}~\ref{cl:thereisapex}. Suppose that $t^2_2$ is not an apex of $\gamma(Q)$. Then $e_2=t_2^2t_2^k$ for some $k\in\{1,3\}$.
 Let $l$ be such that $\{l,k\}=\{1,3\}$ and let $\Delta:=\Delta(e_2, a_{j})$. 
  Additionally, we suppose that $t_2^l$ is not an apex of $\gamma(Q)$. This implies that $\gamma(a_{j}t^2_2)=c_5$ and, as a consequence, 
  any segment of $\qq$ coloured with $c_5$ must be incident with a corner of $\Delta$.  Since $e_2$ and $a_{j}t^2_2$ are sides of $\Delta$, $|\gamma(\Delta)|\geq 2$. 

We now note that no color of $\gamma(\Delta)$ belongs to $\gamma(Q')$ for $Q'=Q\setminus \Delta$. Then $|\gamma(Q)|\geq |\gamma(Q')|+|\gamma(\Delta)|$. Since
$|\gamma(Q')|\geq 5$ by Lemma~\ref{l:TDT}, then $|\gamma(\Delta)|=2$ or we are done. Then we must have that $\gamma(a_{j}t_2^k)=c_6$. 
From this last it follows that $c_6$ must be a star of $\gamma$ with apex $t^k_2\in T_2$, as claimed.
\halnospace  

\vskip 0.1cm
By Claim~\ref{cl:thereisapex} we know that $\gamma(Q)$  has a star with apex $u\in T_2$. By Proposition~\ref{p:closed}~($iii$), it is enough to show
$|\gamma(Q')|\geq 7$ for $Q'=Q\setminus \{u\}$. Let $\{t_2^l,t_2^k\}=T_2\setminus \{u\}$ with $l<k$, and let $\Delta':=\Delta(t_2^l, t_2^k,a_{j})$. 
Since $\Delta'$ is a separable subset of $Q'$, we may assume that $|\gamma(\Delta')|\in \{1,2\}$ or we are done by Lemma~\ref{thm:luis}.

Suppose that $|\gamma(\Delta')|=1$. Then either $l(e_1)a_{j}$ or $r(e_1)t_2^l$ provides the $6$th color of $\gamma(Q')$, and so 
$\gamma(b_{q}t_2^k)\circeq c_3$. This and the fact that
 $b_{q}$ cannot be apex of any star of $\gamma$ imply $\gamma(b_{p}t_2^k)\circeq c_3$. By applying Lemma~\ref{l:4vs5} to 
 $Q''=Q'\setminus \{a_{i}, b_{q}\}$ we have $|\gamma(Q')|\geq |\gamma(Q'')|\geq 7$, as required. 

Suppose now that $|\gamma(\Delta')|=2$. If $\gamma(a_{j}t_2^l)=\gamma(a_{j}t_2^k)$ then 
$a_{j}$ is an apex of $\gamma(Q')$ and we are done by Lemma~\ref{l:TDT}.  Then 
$t_2^lt_2^k$ must have the same color $c'_5$ that some of $a_{j}t_2^l$ or $a_{j}t_2^k$. From this last it is easy to see that
$c'_5$ must be a star of $\gamma$ with apex $v\in \{t_2^l, t_2^k\}$, and so the required inequality follows by applying Claim~\ref{cl:3vs4} to
$Q\setminus \{u,v\}$. This proves (1.3.1).

\vskip 0.1cm
\noindent {(1.3.2)} Suppose that $\gamma(\ell_2)=c_3$ and $\gamma(\ell_4)=c_1$. Since $\gamma$ does not have a star with apex in $[]$, then
 $\gamma(a_{i}b_{q})=c_1$ and as before we can assume that any segment incident with $b_{p}$ is colored with color $c_1$.

\begin{claim}~\label{cl:T1<2star} 
 If  $u,v\in T_1$ are such that $\gamma(Q\setminus \{u,v\}) \subseteq \gamma(Q)\setminus \gamma(T_1)$, then $|\gamma(Q)|\geq 8$.
\end{claim}

\noindent{\em Proof of Claim}~\ref{cl:T1<2star}. Suppose that $T_{1}=\{u,v,w\}$. Since $|\gamma(T_1)|=2$ and 
$\gamma(Q\setminus \{u,v\}) \subseteq \gamma(Q)\setminus \gamma(T_1)$, it is enough to show that $|\gamma(Q')|\geq 6$ for $Q'=Q\setminus \{u,v\}$.
 
 Let $c_7:=\gamma(b_{q}w)$ and note that $c_7\notin \{c_1, c_3, c_5, c_6\}$. Thus, we need to show the existence of one additional color. Since 
 Lemma~\ref{thm:luis} implies that $|\gamma(Q\setminus\{u,v,w\})|\geq 5$, we can assume that $w$ is not an apex of $\gamma(Q')$. 
 This last implies that $\gamma(a_iw)\neq c_7$, and so $\gamma(a_iw)\circeq c_3$ and $\gamma(a_jw)\circeq c_3$. Then we can modify 
 $\gamma$, if necessary, by recoloring with $c_3$ all segments in $a_{i}*(T_2\cup \{w\})$. It is not hard to see this modification does 
 not affect the essential properties of $\gamma(Q')$.  
  
We may assume that $c_5$ is a thrackle of $\gamma(Q')$. Indeed, if $c_5$ is a star with apex $x\in T_2$, then Proposition~\ref{p:closed}~$(iii)$ and 
Lemma~\ref{thm:luis} imply $|\gamma(Q')|\geq|\gamma(Q'\setminus \{x\})|+1\geq 6$, as required. 
Then $e_2=t_2^2t_2^k$ for some $k\in \{1,3\}$ and $\gamma(a_{j}t^2_2)=c_5$. These imply $\gamma(a_{j}t_2^k)\circeq c_6$ and so
$c_6$ must be a star with apex $t_2^k$. By Proposition~\ref{p:closed}~$(iii)$ and 
Lemma~\ref{thm:luis} we have $|\gamma(Q')|\geq|\gamma(Q'\setminus \{t^2_k\})|+1\geq 6$, as required. 
\halnospace

\begin{claim}\label{cl:thereisapexinT1} $\gamma(Q)$ has a proper star with apex $u\in T_1$ or $|\gamma(Q)|\geq 8$. 
 \end{claim}
 \noindent{\em Proof of Claim}~\ref{cl:thereisapexinT1}. Suppose that $t_1^2$ is not an apex of $\gamma(Q)$. Then $e_1=t_1^2t_1^k$ for some $k\in\{1,3\}$.
 Let $l$ be such that $\{l,k\}=\{1,3\}$ and let $\Delta:=\Delta(e_1,b_{q})$. 
 Additionally, we may assume that $t_1^l$ is not an apex of $\gamma(Q)$. This implies that $\gamma(b_{q}t_1^2)=c_2$ and, as a consequence, 
  any segment of $\qq$ coloured with $c_2$ must be incident with a corner of $\Delta$.  Since $e_1$ and $b_{q}t_1^2$ are sides of $\Delta$, $|\gamma(\Delta)|\geq 2$. 

We now note that no color of $\gamma(\Delta)$ belongs to $\gamma(Q')$ for $Q'=Q\setminus \Delta$. Then $|\gamma(Q)|\geq |\gamma(Q')|+|\gamma(\Delta)|$. Since
$|\gamma(Q')|\geq 5$ by Lemma~\ref{l:TDT}, then $|\gamma(\Delta)|=2$ or we are done. Then we must have that $\gamma(b_{q}t_1^k)=c_4$. 
From this last it follows that $c_4$ must be a star of $\gamma$ with apex $t_1^k\in T_1$, as claimed.
\halnospace

\vskip 0.1cm
By Claim~\ref{cl:thereisapexinT1} we know that $\gamma(Q)$ has a star with apex $u\in T_1$. Let $t_1^l,t_1^k\in T_1\setminus \{u\}$ with $l<k$, and let 
$\Delta':=\Delta(t_1^l, t_1^k, b_{q})$. Since $\Delta'$ is a separable subset of $Q'$ with respect to $\gamma$, we may assume that
  $|\gamma(\Delta')|\in \{1,2\}$ or we are done. Moreover, it is not hard to see that $|\gamma(\Delta')|=2$ implies that some chromatic class of $\Delta'$ is a proper star of
  $\gamma(Q)$ with apex $v\in \Delta'$. If $v=b_{q}$ the required result follows by applying Lemma~\ref{l:TDT} and Proposition~\ref{p:closed}~($iii$) to 
  $Q\setminus \{v\}$. Similarly, if 
$v\in \{t_1^l, t_1^k\}$ we are done by Claim~\ref{cl:T1<2star}. Thus, we can assume that $|\gamma(\Delta')|=1$ and that neither $t_1^l$ nor $t_1^k$ is an 
apex of $\gamma(Q)$. By Proposition~\ref{p:closed}~($iii$), it is enough to show $|\gamma(Q')|\geq 7$ for $Q'=Q\setminus \{u\}$. 
We remark that $c_1, c_3, \gamma(\Delta'), c_5$ and $c_6=\gamma(e_2)$ are pairwise distinct.

 Clearly, some of $t_1^k l(e_2)$ or $b_qr(e_2)$ provides the $6$th of $\gamma(Q')$. If 
$\gamma(b_qr(e_2))$ is such $6$th color, then $\gamma(t_1^kl(e_2))\circeq c_6$ and either $t_1^la_i$ or $t_1^k a_{j}$ provides the $7$th required color.
 Thus we may assume that $\gamma(t_1^kl(e_2))$ is the $6$th color, say $c'_6$. Then $\gamma(b_qr(e_2))\circeq c_6$, 
$\gamma(t_1^ll(e_2))=c'_6$, $\gamma(t_1^la_{i})\circeq c_3$, $\gamma(t_1^la_{j})\circeq c_3$, $\gamma(t_1^ka_{j})\circeq c'_6$
 and $\gamma(b_ql(e_2))\circeq c_6$. The $7$th color is given by $t_2^3a_{j}$ if $t_2^3\in e_2$, and by either  
 $t_2^2a_{j}$ or $b_qt_2^3$ when $t^3_2\notin e_2$. This proves (1.3.2).

\vskip 0.1cm
{\sc Case 2}. Suppose that  $\gamma([])=3$. Then one of $\gamma(\ell_2)\notin \{c_1,c_3\}$ or $\gamma(\ell_4)\notin \{c_1,c_3\}$ holds. Suppose that
 $c_2=\gamma(\ell_2)\notin \{c_1,c_3\}$. Then $\gamma(\ell_4)\in \{c_1,c_3\}$.  Note that if $\gamma(\ell_4)=c_1$, then $\ell_3$ can be recolored with $c_2$ 
 and we are done by Case 1. Then we can assume $\gamma(\ell_4)=c_3$. Since any segment that intersects $\ell_2$ also intersect the diagonal $\ell'_2$ of $[]$ that goes from
 $\ell_1\cap \ell_4$ to $\ell_2\cap \ell_3$, we can assume that $\gamma(\ell'_2)=c_2$. Similarly, if $\ell'_4$ is the diagonal of $[]$ that goes from
 $\ell_1\cap \ell_2$ to $\ell_3\cap \ell_4$, we can assume that $\gamma(\ell'_4)=c_3$. Since $c_1$ cannot be a star of $\gamma(Q)$, then $c_1$ must be 
 a triangle $\Delta$ with a side in $\ell_1$.

\vskip 0.1cm
Let $U_{\ell_1}$ be as in Definition~\ref{d:U}. From the definition of $\ell_1$ it is not hard to see that each point of $U_{\ell_1}$ is the apex of a proper star of $\gamma$.
Thus, we may asssume that $|U_{\ell}|\leq 2$, as otherwise Proposition~\ref{p:closed}~($iii$) and Lemma~\ref{thm:luis} imply $|\gamma(Q)|\geq 8$. 
Let $\{v_0, \ldots, v_s\}$ be as in Definition~\ref{d:U}. We remark that $s\geq 4$ because $|U_{\ell_1}|\leq 2$.
 Let $Q'=Q\setminus U_{\ell_1}$. By Proposition~\ref{p:closed}~($iii$), it is enough to show $|\gamma(Q')|\geq |Q'|-2$. 
 
 Let $v_i\in \{v_0, \ldots, v_s\}$ be the corner of $\Delta$ in $T_1\cup T_2$.
 Suppose first that $v_i=v_0$, let $\Delta_j:=\Delta(v_0v_1,x_k)$ with $x_k=\ell_1\cap \ell_k$ and $k\in \{2,4\}$, and let $c_4=\gamma(v_0v_1)$.
Cearly, $c_4\notin \{c_1, c_2, c_3\}$, $x_2=a_i,$ and $x_4=b_p$. It is not hard to see that  $|\gamma(Q')|\geq |\gamma(Q'\setminus\Delta_{j})|+|\gamma(\Delta_{j})|$. Since 
$|\gamma(Q'\setminus \Delta_{j})|\geq |Q'\setminus \Delta_{j}|-2$ by Lemma~\ref{l:TDT}, we can assume that $|\gamma(\Delta_j)|=2$ and so that $\gamma(\Delta_j)=\{c_1, c_4\}$. From this last it follows that
$\gamma(v_1x_2)=c_4=\gamma(v_1x_4)$, contradicting that $v_1\notin U_{\ell_1}$. Then we can assume that $v_i\neq v_0$. Since $v_0\notin U_{\ell_1}$, the triangle 
$\Delta_{\ell_1}:=\Delta(\ell_1, v_0)$ satisfies $|\gamma(\Delta_{\ell_1})|=3$. Suppose that some $c_j\in \{c_2, c_3\}$ is in $\gamma(\Delta_{\ell_1})$. Since any segment of color $c_j$ intersects to $\ell_3$, we can recolor $\ell_3$ (if necessary) with $c_j$. Then $\gamma(Q'\setminus \Delta_{\ell_1})=\gamma(Q') \setminus \{c_1,c_k\}$ with 
$\{c_j, c_k\}=\{c_2,c_3\}$ and $\ell_3$ is the only segment of $Q'\setminus \Delta_{\ell_1}$ colored with $c_j$. Moreover, we note that in such a case 
$Q'\setminus \Delta_{\ell_1}$ and $\ell_3$ satisfy the conditions of  Lemma~\ref{l:/TTb}, and so we can conclude that $|\gamma(Q'\setminus \Delta_{\ell_1})|\geq |Q'\setminus\Delta_{\ell_1}|-1$, or equivalently, $|\gamma(Q')|\geq |Q'|-2$.  Then we can assume that neither $c_2$ nor $c_3$  belongs to $\gamma(\Delta_{\ell_1})$. 
From this fact it is easy to see that $|\gamma(Q')|\geq |\gamma(Q'\setminus\Delta_{\ell_1})|+|\gamma(\Delta_{\ell_1})|$. Since $|\gamma(\Delta_{\ell_1})|=3$, and 
 $ |\gamma(Q'\setminus\Delta_{\ell_1})|\geq |Q'\setminus \Delta_{\ell_1}|-2$ by Lemma~\ref{l:/TT}, then $|\gamma(Q')|\geq |Q'|-2$. The case in which
$\gamma(\ell_4)\notin \{c_1,c_3\}$ can be handled in a similar way.
\end{proof}

%%%%%%%%%%%%%%%%%%%%%%%%%%%%%%%%%%%%%%%%%%%%%%%%%%%%%%%%%%%%%%%%%%%%%%%%
%%%%%%%%%%%%%%%%%%%%%%%%%%%%%%%%%%%%%%%%%%%%%%%%%%%%%%%%%%%%%%%%%%%%%%%%

\begin{lemma}\label{l:2T3} Let  $\Delta_U$ be a triangle with corners in $U \in \{A,B\}$, 
$v_p, v_q\in (A\cup B)\setminus U$ and $Q:=T_1 \cup \Delta_U \cup \{v_p, v_q\} \cup T_2$. If $\gamma$ is an optimal coloring of $D(Q)$, then $|\gamma(D(Q))|\geq 9$.   
\end{lemma}

\begin{proof} Let $u_i, u_j, u_k$ be the corners of  $\Delta_U$ with $i<j<k$. W.l.o.g.
we assume that $p<q$.  We may assume that $\gamma$ does not have a star 
with apex $w\in  \Delta_U$ (resp. $w\in \{v_p, v_q\}$). Indeed, if such an apex $w$ exists, then we can deduce the required inequality by applying Lemma~\ref{l:2T2} 
(resp. Lemma~\ref{l:1T3}) and Proposition~\ref{p:closed}~($iii$) to $Q\setminus \{w\}$. From Lemma~\ref{l:TXT} we know that  
$|\gamma(Q\setminus \Delta_U)|\geq 6$. From this and the fact that $\Delta_U$ is a separable triangle of $Q$, we have that
$|\gamma(Q)|\geq |\gamma(Q\setminus \Delta_U)|+|\gamma(\Delta_U)|\geq 6+ |\gamma(\Delta_U)|$ and so we may assume that $|\gamma(\Delta_U)|\in \{1,2\}$, 
as otherwise we are done. Moreover, since $|\gamma(\Delta_U)|=2$ implies that some vertex $w$ of $\Delta_U$ is an apex of $\gamma(Q)$, we may assume 
 $|\gamma(\Delta_U)|=1$. Let $c_0=\gamma(\Delta_U)$ and let $[]$ be the convex quadrilateral defined by 
$\{u_i, u_k, v_p, v_q\}$. Since $c_0, \gamma(u_iv_p), \gamma(u_kv_q)$ are pairwise distinct, then $|\gamma([])|\geq 3$. See Figure~\ref{fig:4cases}.  

\vskip 0.1cm
{\sc Case 1}. Suppose that $T_1\cup T_2$ lies in the interior of $[]$. Then $[]=\overline{Q}$. If $\gamma([])=4$ then $|\gamma(Q)|\geq |\gamma(Q\setminus [])|+4$. Since  
$|\gamma(Q\setminus [])|\geq 5$ by Lemma~\ref{l:TXT}, we can assume that $|\gamma([])|=3$. 

Let $c_1=\gamma(v_pv_q)$. Let $f\in \{u_iv_p, u_kv_q\}$ be such that $\gamma(f)=c_1$. We note that $f$ exists and is uniquely determined, because $|\gamma([])|=3$. 
Since $\gamma$ has no stars with an apex in $[]$, the triangle $\Delta(f, v_pv_q)$ must be colored with $c_1$. Let $\ell$ be the opposite segment to $f$ in $[]$, let 
$\ell':=u_iu_k$ and let $c_2=\gamma(\ell)$. Then $\gamma([])=\{c_0, c_1, c_2\}$.

Let $U_{\ell}$ be as in Definition~\ref{d:U}. From the choice of $\ell$ it is not hard to see that each point of $U_{\ell}$ is the apex of a proper star of $\gamma$.
Thus, we may asssume that $|U_{\ell'}|\leq 2$, as otherwise Proposition~\ref{p:closed}~($iii$) and Lemma~\ref{thm:luis} imply $|\gamma(Q)|\geq 9$. 
Let $\{v_0, \ldots, v_s\}, \Delta_{\ell}, \ell_1, \ell_2$ and $\Delta_1, \Delta_2$ be as in Definition~\ref{d:U}. We remark that they are well-defined because $|U_{\ell}|\leq 2$. 
Let $Q'=Q\setminus U_{\ell}$. By Proposition~\ref{p:closed}~($iii$), it is enough to show $|\gamma(Q')|\geq 9- |U_{\ell}|$. 

By an argument totally analogous to the one used in the proof of Claim~\ref{cl:xDelta'}, we can assume that $\gamma(\Delta_{\ell})=\{c_2\}$. 
From $v_1\notin U_{\ell}$ we know that $|\gamma(\Delta_i)|=3=|\Delta_i|$ for some $i\in \{1,2\}$. From the position of the points of $\Delta_i$ 
it is not hard  to see that any segment colored with some color in $\gamma(\Delta_i)$ is incident with a point of $\Delta_i$, and hence
 $|\gamma(Q')|\geq |\gamma(Q'\setminus \Delta_i)|+|\gamma(\Delta_i)|$. 
 
 If $\Delta_i\cap \Delta_U\neq \emptyset$, then $Q'\setminus \Delta_i \subset Q \setminus (\Delta_i\cap \Delta_U)$ and so Lemma~\ref{l:2T2} and 
Proposition~\ref{p:closed}~($ii$) imply $|\gamma(Q'\setminus \Delta_i)|\geq |Q'\setminus \Delta_i|-2$.  If $\Delta_i\cap \Delta_U=\emptyset$, then 
$Q'\setminus \Delta_i \subset Q \setminus (\Delta_i\cap \{v_p, v_q \})$ and so Lemma~\ref{l:1T3} and Proposition~\ref{p:closed}~($i$) 
imply $|\gamma(Q'\setminus \Delta_i)|\geq |Q'\setminus \Delta_i|-2$. In any case, we have $|\gamma(Q')|\geq |Q'\setminus \Delta_i|-2+|\gamma(\Delta_i)|$. Since  
 $|Q'\setminus \Delta_i|=11-|U_{\ell}|-|\Delta_i|$ we have $|\gamma(Q')|\geq (11-|U_{\ell}|-|\Delta_i|) - 2+|\gamma(\Delta_i)|=9-|U_{\ell}|$, as required.

\begin{figure}[t]
	\centering
	\includegraphics[width=1\textwidth]{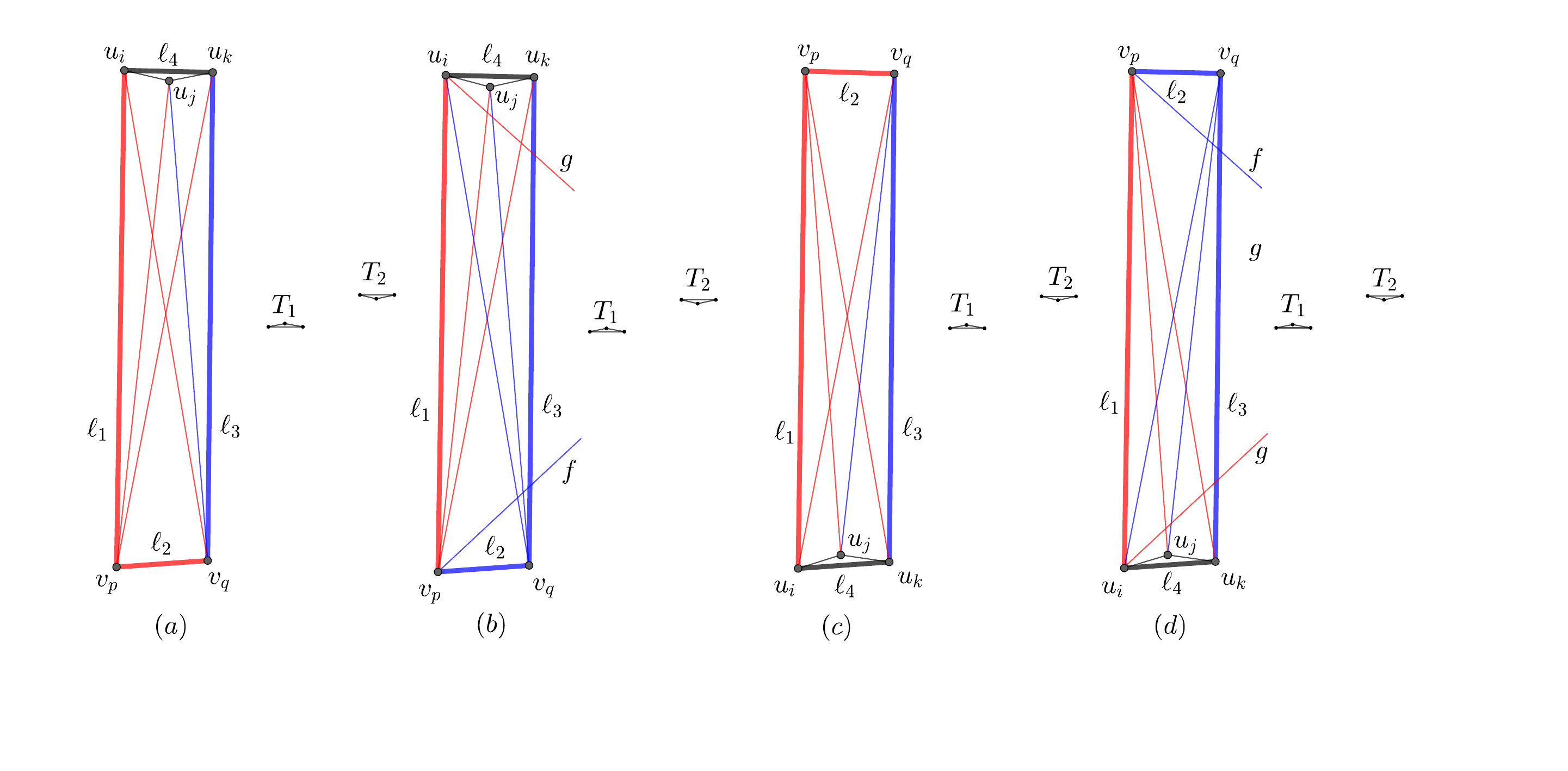}
	\caption{\small Here $T_1\cup T_2$ lies in the exterior of the convex quadrilateral $[]$ formed by the thick segments.
	 W.l.o.g. we may assume that if $S~=~\{u_i, u_j, u_k, v_p, v_q\}$, then
	the segments with both endpoints in $S$ are coloured by $\gamma$ as illustrated in some of these four cases. In (b) $f\in v_p*(T_1\cup T_2)$ and 
	$g\in u_i*(T_1\cup T_2)$. The existence of such $f$ (resp. $g)$ it follows from the assumption that $v_q$ (resp. $v_p$) cannot be an apex of $[]$. Similarly,
	we can assume the existence of the corresponding $f$ and $g$ in (d).}
	\label{fig:4cases}
\end{figure}

{\sc Case 2}. Suppose that $T_1\cup T_2$ lies in the exterior of $[]$. Then the index $i$ of $u_i$ is at most 3, and hence $T_1\cup T_2$ must lie on the
right side of $[]$, as depicted in Figure~\ref{fig:4cases}. Let $\ell_1=u_iv_p, \ell_2=v_pv_q, \ell_3=v_qu_k$ and
 $\ell_4=u_iu_k$ be the sides of $[]$ and let $S=\{u_i, u_j, u_k, v_p, v_q\}$. 
 By Lemma\ref{thm:luis} we may assume that $\gamma$ has at most 2 stars with apices in $T_1\cup T_2$.  

%%%%%%%%%%%%%%%%%%%%%%%%%%%%%%%%%%%%%%%%%%%%%%%%%%%%%%%%%%%%%%%%%%%%%%%%

\begin{claim}\label{cl:4cases}
W.l.o.g. we may assume that the segments with both endpoints in $S$ are coloured by $\gamma$ as in some of the four cases illustrated in Figure~\ref{fig:4cases}.
\end{claim}

\noindent{\em Proof of Claim}~\ref{cl:4cases}. Suppose first that $\Delta_U\subset A$. Thus the points of $S$ are accommodated as depicted in Figure~\ref{fig:4cases}~($a)-(b$).
Note that if $\ss_1=\{v_pu_j, v_pu_k, v_qu_i\}$, then $\ell_1\cup \ss_1$ form a thrackle and moreover, if $\ell\in \qq\setminus \{\ell_1\}$ is such
that $\ell_1\cap \ell\neq \emptyset$, then $\ell \cap \ell'\neq \emptyset$ for each $\ell'\in \ss_1$. From these it follows that we can recolour each segment of $\ss_1$ with color
$\gamma(\ell_1)$ without affecting the essential properties of $\gamma(Q)$. We now consider two subcases, depending on whether or not $\gamma(\ell_1)=\gamma(\ell_2)$.

Suppose that $\gamma(\ell_1)=\gamma(\ell_2)$. Then $\gamma(u_jv_q)\notin \{\gamma(\ell_1), \gamma(\ell_4)\}$. We note that any
$\ell\in \qq\setminus \{u_jv_q\}$ with $\gamma(\ell)=\gamma(u_jv_q)$ must intersect $\ell_3$. Then, if necessary, 
we can recolour $\ell_3$ with color $\gamma(u_jv_q)$ without affecting the essential properties of $\gamma(Q)$, and so we are in the case depicted in 
Figure~\ref{fig:4cases}~($a$).

Suppose now that $\gamma(\ell_1)\neq\gamma(\ell_2)$. We note that if $\ss_2=\{v_qu_i, v_qu_j, v_qu_k\}$, then $\{\ell_2\}\cup \ss_2$ is a thrackle. Moreover, note that if 
$\gamma(\ell)\notin  \{\gamma(\ell_1), \gamma(\ell_4)\}$ and $\ell\cap \ell_2\neq \emptyset$, then $\ell$ intersects to each segment of $\ss_2$, and so we can recolour each segment of $\ss_2$ with color $\gamma(\ell_2)$, lying in the case depicted in Figure~\ref{fig:4cases}~($b$). 

An analogous reasoning shows that if $\Delta_U\subset B$, then $\gamma(S)$ is as in some of the cases depicted in Figure~\ref{fig:4cases}~($c)-(d)$.  
\halnospace

From now on, we let $c_1=\gamma(\ell_1)$ and $c_3=\gamma(\ell_3)$. We recall that $c_0=\gamma(\Delta_U)$. 

\vskip 0.1cm
{\sc Case 2.1}. Suppose that the segments in $\ss$ are coloured by $\gamma$ as in Figure~\ref{fig:4cases}~($a$). By recoloring, if necessary,
we may assume that each segment incident with $v_p$ is coloured with $c_1$. We remark that this assumption does not affect the essential 
properties of $\gamma(Q)$. 

\vskip 0.1cm
{\sc Case 2.1.1}. Suppose that $|\gamma(T_1)|=1$. Let $\Delta:=\Delta(t^1_1, t^3_1,v_q)$. We note that any segment 
that crosses $\Delta$ is incident with $v_p$ and so has color $c_1$. Then $\Delta$ is a separable subset of $Q$ with respect to $\gamma$ and so we can assume that 
$|\gamma(\Delta)|\in \{1,2\}$, as otherwise we are done by Lemma~\ref{l:1T3}. This assumption and $|\gamma(T_1)|=1$ imply that $|\gamma(\Delta)|=2$. From this last it follows 
that $\gamma$ has a star with apex $v_q\in []$, and so we are done by applying  Lemma~\ref{l:1T3} and Proposition~\ref{p:closed}~($iii$) 
to $Q\setminus \{v_q\}$.

{\sc Case 2.1.2}. Suppose that $|\gamma(T_1)|=2$ and $|\gamma(T_2)|=1$. We remark that the 5 colors in $C=\{c_0, c_1, \gamma(T_2)\}\cup \gamma(T_1)$ are pairwise
distinct. Let $c_6=\gamma(t^3_2u_k), c_7=\gamma(t^1_2u_j)$ and $c_8=\gamma(t^2_2v_q)$ 
and note that the colors in $C\cup \{c_6, c_7, c_8\}$ are pairwise distinct.  Then $\gamma(l(e_1)u_i)\circeq \gamma(e_1)$, $\gamma(r(e_1)u_j)\circeq c_7$, 
$\gamma(t^1_2u_k)\circeq c_6$ and $\gamma(t^3_2v_q)\circeq c_8$. From these it follows that $\gamma(r(e_1)t^2_2)$ is the
$9$th color.

\vskip 0.1cm
{\sc Case 2.1.3}. Suppose that $|\gamma(T_1)|=2=|\gamma(T_2)|$. Let $e_2$ be as in (N3). 
\vskip 0.1cm
\noindent $\bullet$ Suppose that $\gamma(Q)$ has 2 stars with apices $u,v\in T_1$. It is enough to show that $|\gamma(Q\setminus \{u,v\})|\geq 7$.
 Let $w$ be such that $T_1=\{u,v,w\}$.
By Lemma~\ref{l:TXT} we may assume that no point of $T_2\cup \{w\}$ is an apex of $\gamma$. We note that the 5 
colors in $C=\{c_0, c_1, c_3\}\cup \gamma(T_2)$ are pairwise distinct. Then $c_6=\gamma(wu_k)$ must be the 
$6$th color, and so $\gamma(wv_q)\circeq c_3$. Suppose first that $t^3_2\in e_2$. Since $c_3$ cannot be a star, then $\gamma(\Delta)=\gamma(e_2)$ where 
$\Delta:=\Delta(e_2, v_q)$. Then either $l(e_2)w$ or $t_2^3u_k$ provides the $7$th color.
Suppose now that $t^3_2\notin e_2$. Then $e_2=t^1_2t^2_2$, and either $\gamma(v_qt^3_2)$ is the $7$th required color or 
$\gamma(v_qt^3_2)=\gamma(t^1_2t^3_2)=\gamma(t^2_2t^3_2)$. Since these equalities imply that $t^3_2\in T_2$ is an apex of $\gamma$, we are done.

\vskip 0.1cm
\noindent $\bullet$ Suppose that $\gamma(Q)$ has exactly 1 star with apex in $T_1$. Let $u$ be the only apex of $\gamma$ in $T_1$, and let $t_1^l, t_1^m$ be such that 
$T_1=\{u,t_1^l, t_1^m\}$ with $l<m$. By Proposition~\ref{p:closed}~($iii$), it is enough to show that $|\gamma(Q\setminus \{u\})|\geq 8$. Let $\Delta:=\Delta(t_1^l, t_1^m,v_q)$. We note that any segment that crosses $\Delta$ is incident with $v_q$ and so has color $c_1$. 
Then $\Delta$ is a separable subset of $Q$ with respect to $\gamma$, and so we can assume that $|\gamma(\Delta)|\in \{1,2\}$, as otherwise we are done by Lemma~\ref{l:1T3}. 
Moreover, since $|\gamma(\Delta)|=2$ implies that $\gamma$ has an apex in a corner of $\Delta$, and this is impossible by previous assumptions, 
we must have $|\gamma(\Delta)|=1$. Then $c_7=\gamma(t_1^lu_k)$ is the $7$th color, and so $\gamma(t_1^mu_k)\circeq c_7$,
$\gamma(t_1^ml(e_2))\circeq \gamma(e_2)$ and $\gamma(v_qr(e_2))\circeq c_3$. Then $t_1^lu_j$ provides the $8$th required color. 
    
\vskip 0.1cm
\noindent $\bullet$ Suppose that no point of $T_1$ is an apex of $\gamma(Q)$. Then either $e_1=t^1_1t^2_1$ or $e_1=t^2_1t^3_1$. Let $c_2$ (resp. $c_4$) be such that
$\{c_2, \gamma(e_1)\}=\gamma(T_1)$ (resp. $\{c_4, \gamma(e_2)\}=\gamma(T_2)$). We remark that the 7 colors in $C=\{c_0, \ldots, c_4, \gamma(e_1), \gamma(e_2)\}$ 
are pairwise distinct. Since $c_2$ cannot be a star of $\gamma(Q)$ and any segment incident with $v_p$ is colored with 
$c_1$, then $\gamma(v_qt^2_1)=c_2$.  

\noindent $-$ Suppose that $e_1=t^1_1t^2_1$. From $\gamma(v_qt^2_1)=c_2$ it follows that $c_8:=\gamma(t^3_1u_k)$ is the $8$th color. 
We may assume that $\gamma(t^1_1v_q)\in \{\gamma(e_1), c_3\}$, as otherwise $\gamma(t^1_1v_q)$ is the required color.
If $\gamma(t^1_1v_q)=\gamma(e_1)$, then $\gamma(t^2_1u_k)\circeq c_8$, $\gamma(t^2_1u_j)\circeq c_3$, and hence either
$l(e_2)t^3_1$ or $r(e_2)v_q$ provides the $9$th color. Suppose now that $\gamma(t^1_1v_q)=c_3$. 
 Since $c_3$ cannot be a star of $\gamma(Q)$ with apex in $v_q$, then neither $l(e_2)v_q$ nor $r(e_2)v_q$ is colored $c_3$.   
Then $\gamma(v_ql(e_2))\circeq\gamma(e_2)\circeq \gamma(v_qr(e_2))$. If $t^3_2\in e_2$, then either $t^3_1l(e_2)$ or $t^3_2u_k$ provides
 the $9$th color, and if $t^3_2\notin e_2$ then $\gamma(t^3_1t^1_2)\circeq c_8$, $\gamma(t^3_2v_q)\circeq c_4$ and so the $9$th required color is provided by 
 $t^2_2u_k$.
 
\noindent $-$ Suppose that $e_1=t^2_1t^3_1$. From $\gamma(v_qt^2_1)=c_2$ it follows that $c_8:=\gamma(t^1_1u_k)$ is the $8$th color. 
 We may assume that $\gamma(t^3_1v_q)\in \{\gamma(e_1), c_3\}$, as otherwise $\gamma(t^3_1v_q)$ is the required color. 
 If $\gamma(t^3_1v_q)=\gamma(e_1)$, then $\gamma(t^2_1l(e_2))\circeq \gamma(e_2)$, $\gamma(r(e_2)v_q)\circeq c_3$, $\gamma(t^1_1u_j)\circeq c_8$
and so $t^2_1u_k$ provides the $9$th color. Suppose now that $\gamma(t^3_1v_q)=c_3$. Then $\gamma(t^1_1u_j)\circeq c_8$. Since $c_3$ cannot be a star of $\gamma(Q)$,
then $\gamma(v_qr(e_2))\circeq \gamma(e_2)$, $\gamma(t^3_1l(e_2))\circeq \gamma(e_1)$ and $t^2_1u_k$ provides the $9$th color.

%%%%%%%%%%%%%%%%%%%%%%%%%%%%%%%%%%%%%%%%%%%%%%%%%%%%%%%%%%%%%%%%%%%%%%%%
%%%%%%%%%%%%%%%%%%%%%%%%%%%%%%%%%%%%%%%%%%%%%%%%%%%%%%%%%%%%%%%%%%%%%%%%

\vskip 0.1cm
{\sc Case 2.2}. Suppose that the segments in $\ss$ are coloured by $\gamma$ as in Figure~\ref{fig:4cases}~($b$). 

\vskip 0.1cm
{\sc Case 2.2.1}. Suppose that $|\gamma(T_1)|=1=|\gamma(T_2)|$. Let $c_6:=\gamma(t^1_1t^1_2)$, $c_7:=\gamma(t^2_1t^2_2)$ and $c_8:=\gamma(t^3_1t^3_2)$.
Clearly, the 8 colors $c_0, c_1, c_3, c_6, c_7, c_8, \gamma(T_1), \gamma(T_2)$ are pairwise distinct. Then either $t^1_1u_j$ or $t^1_2u_k$ provides the $9$th color.  

\vskip 0.1cm
{\sc Case 2.2.2}. Suppose that $|\gamma(T_1)|=1$ and $|\gamma(T_2)|=2$. Let $\gamma(T_1)=\{c_2\}$, $\gamma(T_2)=\{c_4, \gamma(e_2)\}$, 
$c_7:=\gamma(t^3_1u_k)$ and $c_8:=\gamma(t^2_1u_j)$. Then the $8$ colors $c_0, c_1, c_2, c_3, c_4, c_7, c_8, \gamma(e_2)$ 
are pairwise distinct. Then $\gamma(t^1_1u_j)\circeq c_8$, $\gamma(t^2_1u_k)\circeq c_7$, $\gamma(t_1^3l(e_2))\circeq \gamma(e_2)$, 
$\gamma(t_1^3r(e_2))\circeq \gamma(e_2)$ and $\gamma(t^3_1v_p)\circeq c_3$. Then $r(e_2)v_q$ provides the $9$th color. 

\vskip 0.1cm
{\sc Case 2.2.3}. Suppose that $|\gamma(T_1)|=2$ and $|\gamma(T_2)|=1$. Let $\gamma(T_1)=\{c_2, \gamma(e_1)\}$, 
$\gamma(T_2)=\{c_4\}$, $c_7:=\gamma(t^3_2u_k)$ and $c_8:=\gamma(t^1_2u_j)$. Then the $8$ colors $c_0, c_1, c_2, c_3, c_4, c_7, c_8, \gamma(e_1)$ 
are pairwise distinct.  Then $\gamma(l(e_1)t_2^2)\circeq \gamma(e_1)$, $\gamma(r(e_1)t_2^2)\circeq \gamma(e_1)$,
$\gamma(l(e_1)u_j)\circeq c_8$ and $\gamma(t^1_2u_k)\circeq c_7$. Then $t^3_2r(e_1)$ provides the $9$th color.

\vskip 0.1cm
{\sc Case 2.2.4}. Suppose that $|\gamma(T_1)|=2=|\gamma(T_2)|$. Let $\gamma(T_1)=\{c_2, \gamma(e_1)\}$ and
 $\gamma(T_2)=\{c_4, \gamma(e_2)\}$.
\vskip 0.1cm

\noindent $\bullet$ Suppose that $\gamma(Q)$ has 2 stars with apices $u,v\in T_1$. It is enough to show that $|\gamma(Q\setminus \{u,v\})|\geq 7$. 
 Let $w$ be such that $T_1=\{u,v,w\}$. By Lemma~\ref{l:TXT} we may assume that no point of $T_2\cup \{w\}$ is an apex of $\gamma(Q)$.
 We note that the 5 colors $c_0, c_1, c_3, c_4, \gamma(e_2)$ are pairwise distinct. 
 
Then $c_6:=\gamma(wu_j)$ must be the $6$th color, and so $\gamma(wu_k)\circeq c_6$. Since $t_2^2$ cannot be an apex of $\gamma$,
then either $e_2=t_2^1t_2^2$ or $e_2=t_2^2t_2^3$. If $e_2=t_2^1t^2_2$, then  $\gamma(t^1_2u_k)\circeq \gamma(e_2)$,
$\gamma(t^3_2u_k)\circeq c_4$, $\gamma(wt^2_2)\circeq c_6$, and so $\gamma(wt^1_2)\in \{\gamma(e_2), c_6\}$. We note that if
$\gamma(wt^1_2)=\gamma(e_2)$ (respectively, $\gamma(wt^1_2)=c_6)$) then $t^1_2$ (respectively, $w$) is an apex of $\gamma$, 
contradicting that $\gamma(T_2\cup \{w\})$ has no apices. Thus $e_2=t_2^2t_2^3$ and so $\gamma(t^3_2u_k)\circeq \gamma(e_2)$. 
Since $\gamma(e_2)$ cannot be a star, then 
$\gamma(t_2^2u_k)\circeq \gamma(e_2)$, $\gamma(t_2^3w)\circeq c_6$ and $\gamma(t_2^2w)\circeq c_6$. But these imply  
that $c_6$ is a star with apex in $w$, a contradiction.  

\vskip 0.1cm
\noindent $\bullet$ Suppose that $\gamma(Q)$ has exactly 1 star with apex in $T_1$. Let $u$ be the only apex of $\gamma$ in $T_1$, let $t_1^l, t_1^m$ be such that 
$T_1=\{u,t_1^l, t_1^m\}$ with $l<m$, and let $e=t_1^lt_1^m$. It is enough to show that $|\gamma(Q\setminus \{u\})|\geq 8$. 

 We note that the 6 colors $c_0, c_1, c_3, c_4, \gamma(e), \gamma(e_2)$ are pairwise distinct. Clearly, either $t_1^lu_j$ or $t_1^mu_k$ provides the $7$th color $c_7$.

\vskip 0.1cm
\noindent $-$ Suppose that $\gamma(t_1^lu_j)=c_7$. Then $\gamma(t_1^mu_k)\circeq \gamma(e)$.
Since $t_1^m$ cannot be an apex, there is $h\in t_1^l*T_2$ such that $\gamma(h)=\gamma(e)$. 

If $t_2^3\in e_2$, then $\gamma(t_2^3u_k)\circeq \gamma(e_2)$, $\gamma(l(e_2)t_1^m)\circeq \gamma(e)$ and $\gamma(l(e_2)v_q)\circeq c_3$.
These and the existence of $h$ imply $\gamma(t_1^mv_p)\circeq c_1$ and $\gamma(t_1^mu_i)\circeq c_1$. Since $t_1^l$ cannot be an apex, then 
 $t_1^lv_p$ cannot be colored $c_7$, and so $t_1^lv_p$ provides the required $8$th color. Similarly, if $t_2^3\notin e_2$ then $e_2=t_2^1t_2^2$, and so 
  $\gamma(t_2^1u_k)\circeq \gamma(e_2)$, $\gamma(t_2^2t_1^m)\circeq \gamma(e)$ and $\gamma(t_2^2v_q)\circeq c_3$. 
These and the existence of $h$ imply $\gamma(t_1^mv_p)\circeq c_1$ and $\gamma(t_1^mu_i)\circeq c_1$. Since $t_1^l$ cannot be an apex, then 
 $t_1^lv_p$ cannot be colored $c_7$, and so $t_1^lv_p$ provides the required $8$th color.

\vskip 0.1cm
\noindent $-$ Suppose that $\gamma(t_1^mu_k)=c_7$. Then $\gamma(t_1^lu_j)\circeq \gamma(e)$.
Since $t_1^l$ cannot be an apex, there is $h\in t_1^m*\{u_i, u_j\}$ such that $\gamma(h)=\gamma(e)$. 
Then $\gamma(t_1^mv_q)\circeq c_3$, $\gamma(t_1^lv_p)\circeq c_1$ and $\gamma(t_1^lu_i)\circeq c_1$. Since $t_1^m$ cannot be an apex, then
$\gamma(t_1^mv_p)\neq c_7$ and so $\gamma(t_1^mv_p)\circeq c_3$. 
Then the triangle formed by $e_2$ and $v_q$ must be colored with $\gamma(e_2)$ and so $\gamma(t_1^ml(e_2))\circeq c_7$. 
If $t^3_2\in e_2$, then $t^3_2u_k$ provides the required $8$th color, and if $t^3_2\notin e_2$ then $e_2=t^1_2t^2_2$ and either $t^2_2u_k$ or $t^3_2v_q$ 
 provides the required $8$th color.

\vskip 0.1cm
\noindent $\bullet$  Suppose that $\gamma$ has no stars in $T_1$. Then either $e_1 = t^1_1t^2_1$ or $e_1 = t^2_1t^3_1$. 
Since $c_2$ cannot be a star, there is $h\in t^2_1*\{v_p, v_q\}$ such that $\gamma(h) = c_2$. We remark that the 7 colors 
$c_0, c_1, c_3, c_2, c_4, \gamma(e_1), \gamma(e_2)$ are pairwise distinct.

\vskip 0.1cm
\noindent $-$ Suppose that $e_1 = t^1_1t^2_1$. Then either $t^1_1u_j$ or $t^2_1u_k$ needs a new color $c_8$. 
Then $\{\gamma(e_1),c_8\}=\{\gamma(t^1_1u_j), \gamma(t^2_1u_k)\}$, as otherwise we are done. 
If $\gamma(t^1_1u_j) = \gamma(e_1)$ and $\gamma(t^2_1u_k) = c_8$, then 
$\gamma(t^3_1l(e_2)) \circeq  \gamma(e_2)$, $\gamma(t^3_1r(e_2)) \circeq \gamma(e_2)$ and  for $w \in T_2 \setminus \{e_2\}$, 
$\gamma(t^3_1w) \circeq c_4$, hence either $l(e_2)u_j$ or $r(e_2)u_k$ provides the 9th color.
Similarly, if $\gamma(t^1_1u_j) = c_8$ and $\gamma(t^2_1u_k) = \gamma(e_1)$, then   
$\gamma(t^3_1l(e_2)) \circeq \gamma(e_2)$ and $\gamma(t^3_1r(e_2)) \circeq \gamma(e_2)$, and so either
$\gamma(u_kl(e_2))$ or $\gamma(u_kr(e_2))$ is the 9th color.	

\vskip 0.1cm
\noindent $-$ Suppose that $e_1 = t^2_1t^3_1$.  As before, either $t^2_1u_j$ or $t^3_1u_k$ needs a new color $c_8$, and  
 $\{\gamma(t^2_1u_j), \gamma(t^3_1u_k)\} = \{\gamma(e_1),c_8\}$, as otherwise we are done. 
Since $\gamma(t^2_1u_j) = \gamma(e_1)$ and $\gamma(t^3_1u_k) = c_8$ together with the existence of $h$
imply that $\gamma(t^1_1u_j)$ is the  9th color, we can assume that
 $\gamma(t^2_1u_j) = c_8$ and $\gamma(t^3_1u_k) = \gamma(e_1)$. Then $\gamma(t^1_1u_j) \circeq c_8$.
  If $t_2^3\in e_2$, then $\gamma(t_2^3u_k) \circeq \gamma(e_2)$, 
  $\gamma(l(e_2)t^3_1)\circeq \gamma(e_1)$, and hence $t^2_1u_k$ provides the 9th color. Similarly, if $t_2^3\notin e_2$, then $e_2=t_2^1t_2^2$
and  $\gamma(t_2^1u_k) \circeq \gamma(e_2)$,  $\gamma(t^3_1t_2^2)\circeq \gamma(e_1)$, and hence $t^2_1u_k$ provides the 9th color.

\vskip 0.1cm
{\sc Case 2.3}. Suppose that the segments in $\ss$ are coloured by $\gamma$ as in Figure~\ref{fig:4cases}~($c$). As before, w.l.o.g. 
we may assume that each segment incident with $v_p$ is coloured with $c_1$.

\vskip 0.1cm
{\sc Case 2.3.1}. Suppose that $|\gamma(T_2)|=1$. Let $\Delta:=\Delta(t^1_2, t^3_2, v_q)$. We note that any segment 
that crosses $\Delta$ is incident with $v_p$ or $v_q$. Then $\Delta$ is a separable subset of $Q$ with respect to $\gamma$, and so we can assume that 
$|\gamma(\Delta)|\in \{1,2\}$, as otherwise we are done by Lemma~\ref{l:1T3}. This assumption and $|\gamma(T_2)|=1$ imply that $|\gamma(\Delta)|=2$. Then 
$\gamma(Q)$ has a star with apex $v_q\in []$ and we are done by applying  Lemma~\ref{l:1T3} and Proposition~\ref{p:closed}~($iii$) 
to $Q\setminus \{v_q\}$.

{\sc Case 2.3.2}. Suppose that $|\gamma(T_1)|=1$ and $|\gamma(T_2)|=2$. We remark that the 5 colors in $C=\{c_0, c_1, \gamma(T_1)\}\cup \gamma(T_2)$ are pairwise
distinct. Let $c_6:=\gamma(t^2_1v_q), c_7:=\gamma(t^1_1u_i)$ and $c_8:=\gamma(t^3_1u_j)$ 
and note that the colors in $C\cup \{c_6, c_7, c_8\}$ are pairwise distinct. We remark that possibly $c_3\in \{c_6, c_7, c_8\}$.  
Then $\gamma(r(e_2)u_k)\circeq \gamma(e_2)$, $\gamma(l(e_2)u_j)\circeq c_8$, 
$\gamma(t^3_1u_i)\circeq c_7$ and $\gamma(t^1_1v_q)\circeq c_6$. Then $\gamma(t^2_1l(e_2))$ must be the 
$9$th color.
   
\vskip 0.1cm
{\sc Case 2.3.3}. Suppose that $|\gamma(T_1)|=2=|\gamma(T_2)|$. 
\vskip 0.1cm
\noindent $\bullet$ Suppose that $\gamma(Q)$ has 2 stars with apices $u,v\in T_1$. It is enough to show that $|\gamma(Q\setminus \{u,v\})|\geq 7$. 
 Let $w$ be such that $T_1=\{u,v,w\}$. By Lemma~\ref{l:TXT} we may assume that no point of $T_2\cup \{w\}$ is an apex of $\gamma$.
 We note that the 4 colors in $C=\{c_0, c_1\}\cup \gamma(T_2)$ are pairwise distinct. 
Then $c_5:=\gamma(wu_i)$ must be the $5$th color. Since $w$ cannot be an apex of $\gamma$, then $c_6:=\gamma(wv_q)$ must be
the $6$th color. We remark that possibly $c_3\in \{c_5, c_6\}$. Then either $l(e_2)u_j$ or $r(e_2)u_k$ provides the $7$th color.

\vskip 0.1cm
\noindent $\bullet$ Suppose that $\gamma(Q)$ has exactly 1 star with apex in $T_1$. Let $u$ be the only apex of $\gamma$ in $T_1$, and suppose that 
$T_1=\{u,t_1^l, t_1^m\}$ with $l<m$. It is enough to show that $|\gamma(Q\setminus \{u\})|\geq 8$. Let $\Delta:=\Delta(t_1^l, t_1^m, v_q)$. 
Since any segment that crosses a segment of $\Delta$ is incident a point of $\Delta \cup \{v_p\}$, then 
$\Delta$ is a separable subset of $Q$ with respect to $\gamma$ and so we can assume that $|\gamma(\Delta)|\in \{1,2\}$, as otherwise we are done by Lemma~\ref{l:1T3}. Let
$c_2:=\gamma(t_1^lt_1^m)$. Note that the 6 colors in $C=\{c_0, c_1, c_2, c_3\}\cup \gamma(T_2)$ are pairwise distinct. 

If $|\gamma(\Delta)|=1$ then $\gamma(\Delta)=\{c_2\}$, and $\gamma(t_1^lu_k)$ must be the $7$th color. Then the triangle $\Delta':=\Delta(e_2, t_1^m)$ 
must be coloured with $\gamma(e_2)$ or we are done. Then either $l(e_2)v_q$ or $r(e_2)v_q$ provides the $8$th color. 

Suppose now that $|\gamma(\Delta)|=2$. Then two sides of $\Delta$ have the same color. Since $t_1^l$ cannot be an apex,
 then $\gamma(t_1^lv_q)\neq c_2$, and so either $\gamma(v_qt_1^l)=\gamma(v_qt_1^m)$ or $\gamma(v_qt_1^m)=c_2$.
 If $\gamma(v_qt_1^l)=\gamma(v_qt_1^m)$, then $c_7:=\gamma(v_qt_1^l)$ must be the $7$th color, as otherwise $c_3=\gamma(v_qt_1^l)$ 
is a star with apex  $v_q\in []$ and we are done by Lemma~\ref{l:1T3}. Then $\gamma(t_1^lu_k)\circeq c_2$ and either 
$l(e_2)t_1^m$ or $r(e_2)u_k$ provides the $8$th color. Thus we may assume that $\gamma(v_qt_1^m)=c_2$ and hence $c_7:=\gamma(t_1^lu_k)$ must be the $7$th color. 
Since $t_1^l$ and $t_1^m$ cannot be apices of $\gamma$,  then $\gamma(v_qt_1^l)\circeq c_3$ and $\gamma(t_1^mu_j)\circeq c_7$, respectively. 
If $t^3_2\in e_2$ then either $l(e_2)u_k$ or $t_2^3v_q$ provides the $8$th color,  and if $t^3_2\notin e_2$ then either $t_2^2u_k$ or $t_2^1v_q$ provides the $8$th color.

\vskip 0.1cm
\noindent $\bullet$ Suppose that no point of $T_1$ is an apex of $\gamma$. Then either $e_1=t_1^1t_1^2$ or $e_1=t_1^2t_1^3$ and 
there is a segment $h\in t^2_1*\Delta_U$ such that $\gamma(h)=\gamma(t_1^1t_1^3)$.
Let $c_2$ and $c_4$ be such that $\{c_2, \gamma(e_1)\}=\gamma(T_1)$ and $\{c_4, \gamma(e_2)\}=\gamma(T_2)$. We remark that the 7 colors in 
 $C=\{c_0, \ldots, c_4, \gamma(e_1), \gamma(e_2)\}$ are pairwise distinct. 

Let $\Delta:=\Delta(t^1_1, t^2_1,v_q)$. Since any segment that crosses $\Delta$ is incident with $v_p$ and all these are colored with $c_1$, then
$\Delta$ is a separable subset of $Q$ with respect to $\gamma$, and so we can assume that $|\gamma(\Delta)|\in \{1,2\}$, as otherwise we are done by Lemma~\ref{l:1T3}. 
 
Suppose that $|\gamma(\Delta)|=1$. Then $e_1=t^1_1t^2_1$ and $\gamma(e_1)=\gamma(\Delta)$. Then $\gamma(t^1_1u_k)$ is the $8$th color. This last and
the existence of $h$ imply that either $l(e_2)t^2_1$ or $r(e_2)t^3_1$ provides the $9$th color.

Suppose now that $|\gamma(\Delta)|=2$. Since $t^1_1$ cannot be an apex, we must have $\gamma(t^1_1t^2_1)\neq \gamma(t^1_1v_q)$. Then
either $\gamma(t^1_1t^2_1)=\gamma(t^2_1v_q)$ or $\gamma(t^1_1v_q)=\gamma(t^2_1v_q)$. 
If $\gamma(t^1_1t^2_1)=\gamma(t^2_1v_q)$, then $e_1=t_1^1t_1^2$ and $c_8:=\gamma(t^1_1u_k)$ is the $8$th color. Since $t_1^1$ 
cannot be an apex of $\gamma$, then $\gamma(t^1_1v_q)\circeq c_3$. From this, the existence of $h$, and our supposition that $v_q\in []$ cannot be an apex of $\gamma$, it follows that $\gamma(t^3_1v_q)$ must be the $9$th color.
Suppose finally that $\gamma(v_qt^1_1)=\gamma(v_qt^2_1)$. Since $v_q$ cannot be an apex of $\gamma$, then $\gamma(v_qt^1_1)\neq c_3$ and so $\gamma(v_qt^1_1)$ 
must be the $8$th color. Then $\gamma(r(e_2)u_k)\circeq \gamma(e_2)$. If $e_1=t_1^1t_1^2$, then either $t_1^1u_k$ or $t_1^2l(e_2)$ provides the $9$th color. Similarly, 
if $e_1=t_1^2t_1^3$, either $t_1^3u_k$ or $t_1^2l(e_2)$ provides the $9$th color.

\vskip 0.1cm
{\sc Case 2.4}. Suppose that the segments in $\ss$ are coloured by $\gamma$ as in Figure~\ref{fig:4cases}~($d$). 

\vskip 0.1cm
{\sc Case 2.4.1}. Suppose that $|\gamma(T_1)|=1=|\gamma(T_2)|$. Let $c_6:=\gamma(t^1_1t^1_2)$, $c_7:=\gamma(t^2_1t^2_2)$ and $c_8:=\gamma(t^3_1t^3_2)$.
Clearly, the colors $c_0, c_1, c_3, c_6, c_7, c_8, \gamma(T_1), \gamma(T_2)$ are pairwise distinct. Then  
$\gamma(t_1^3u_k)\circeq c_8$ and $\gamma(t_2^3u_k)\circeq c_8$, and so $t^3_1u_j$ provides the $9$th color.   

\vskip 0.1cm
{\sc Case 2.4.2}. Suppose that $|\gamma(T_1)|=1$ and $|\gamma(T_2)|=2$. Let $\gamma(T_1)=\{c_2\}$, $\gamma(T_2)=\{c_4, \gamma(e_2)\}$,
$c_7:=\gamma(t^1_1u_j)$ and   $c_8:=\gamma(t^3_1u_k)$. We note that the $8$ colors $c_0, c_1, c_2, c_3, c_4, c_7, c_8, \gamma(e_2)$ 
are pairwise distinct. Then $\gamma(t_1^2l(e_2))\circeq \gamma(e_2)$, $\gamma(t_1^2r(e_2))\circeq \gamma(e_2)$, $\gamma(u_kr(e_2))\circeq c_8$ and 
$\gamma(u_jt_1^3)\circeq c_7$. Then $t^1_1l(e_2)$ provides the $9$th color. 

\vskip 0.1cm
{\sc Case 2.4.3}. Suppose that $|\gamma(T_1)|=2$ and $|\gamma(T_2)|=1$. Let $\gamma(T_1)=\{c_2, \gamma(e_1)\}$, $\gamma(T_2)=\{c_4\}$, 
$c_7:=\gamma(t^1_2u_j)$ and $c_8:=\gamma(t^3_2u_k)$. Then the $8$ colors $c_0, c_1, c_2, c_3, c_4, c_7, c_8, \gamma(e_1)$ 
are pairwise distinct. Then $\gamma(t_2^2u_j)\circeq c_7$, $\gamma(l(e_1)t_2^1)\circeq \gamma(e_1)$, $\gamma(r(e_1)t_2^1)\circeq \gamma(e_1)$ and 
$\gamma(t^2_2u_k)\circeq c_8$. Clearly, there exists $w\in \{t_1^1, t_1^3\}\cap \{l(e_1), r(e_1)\}$. Then $\gamma(wu_i)\circeq c_1$ and either $t_2^1v_p$ or $t_2^3v_q$ provides the $9$th color.
 
\vskip 0.1cm
{\sc Case 2.4.4}. Suppose that $|\gamma(T_1)|=2=|\gamma(T_2)|$.  Let $t_2'$ be such that $T_2=\{l(e_2),r(e_2), t_2'\}$, and let $\gamma(T_2)=\{c_4, \gamma(e_2)\}$. 

\vskip 0.1cm
\noindent $\bullet$ Suppose that $\gamma(Q)$ has 2 stars with apices $u,v\in T_1$. It is enough to show that $|\gamma(Q\setminus \{u,v\})|\geq 7$. 
 Let $w$ be such that $T_1=\{u,v,w\}$. Then $c_0, c_1, c_3, c_4, \gamma(e_2)$ are pairwise distinct. Clearly, 
 $c_6:=\gamma(wu_j)$ must be the $6$th color, and hence
 $\gamma(l(e_2)u_k)\circeq \gamma(e_2)$,  $\gamma(r(e_2)u_k)\circeq \gamma(e_2)$,
  $\gamma(wl(e_2))\circeq c_6$, $\gamma(l(e_2)u_j)\circeq c_6$, $\gamma(wu_i)\circeq c_1$ and $\gamma(t_2' u_k)\circeq c_4$.  
Then either $l(e_2)v_p$ or $r(e_2)v_q$ provides the $7$th color.   
 
\vskip 0.1cm
\noindent $\bullet$  Suppose that $\gamma(Q)$ has exactly 1 star with apex in $T_1$. Let $u$ be the only apex of $\gamma$ in $T_1$, let $t_1^l, t_1^m$ be such that 
$T_1=\{u,t_1^l, t_1^m\}$ with $l<m$. We need to show $|\gamma(Q\setminus \{u\})|\geq 8$. 
We note that $c_0, c_1, c_2:=\gamma(t_1^lt_1^m), c_3, c_4,  \gamma(e_2)$ are pairwise distinct. Then either $l(e_2)u_j$ or $r(e_2)u_k$ provides the $7$th color $c_7$.

\vskip 0.1cm
\noindent $-$ Suppose that $\gamma(r(e_2)u_k)=c_7$. Then $\gamma(l(e_2)u_j)\circeq \gamma(e_2)$,  $\gamma(t_1^lu_j)\circeq c_2$, $\gamma(t_1^mu_j)\circeq c_2$,  
$\gamma(t_1^lu_i)\circeq c_1$ and $\gamma(t_1^mu_k)\circeq c_7$. If $t^3_2\in e_2$, then either $t_1^mv_p$ or $t_2^3v_q$ provides the $8$th color. 
Otherwise $e_2=t^1_2t^2_2$, and $\gamma(t_2^1t_1^m)\circeq \gamma(e_2)$, $\gamma(t_2^2u_j)\circeq c_7$ and $\gamma(t_2^3u_k)\circeq c_4$. 
Then either $t_1^mv_p$ or $t_2^2v_q$ provides the $8$th color.

\vskip 0.1cm
\noindent $-$ Suppose now that $\gamma(l(e_2)u_j)=c_7$. Then $\gamma(r(e_2)u_k)\circeq \gamma(e_2)$.
We may assume that $\gamma(t_1^mu_k)\in \{c_2, c_7\}$, as otherwise $\gamma(t_1^mu_k)$ is the required $8$th color. If $\gamma(t_1^mu_k)=c_7$,
 then $\gamma(t_1^ll(e_2))\circeq c_2$ and $\gamma(t_1^l u_j)\circeq c_2$. Since these imply (a contradiction) that $t_1^l$ is an apex of $\gamma$, 
 we may assume that  $\gamma(t_1^mu_k)=c_2$. Again, since $t_1^m$ cannot be an apex, then $\gamma(t_1^lu_k)\circeq c_2$,
 $\gamma(t_1^l u_j)\circeq c_7$, $\gamma(t_1^lu_i)\circeq c_1$ and $\gamma(l(e_2)u_k)\circeq \gamma(e_2)$. Then either $t_1^mv_p$ or
 $l(e_2)v_q$ provides the $8$th color.

\vskip 0.1cm
\noindent $\bullet$ Suppose that no point of $T_1$ is an apex of $\gamma(Q)$. Then $e_1=t_1^2t_1^m$ for some $m\in \{1,3\}$.  
Let $c_2$ and $c_4$ be such that $\{c_2, \gamma(e_1)\}=\gamma(T_1)$ and $\{c_4, \gamma(e_2)\}=\gamma(T_2)$. 
Since $\gamma$ has no apices in $T_1$, there is a segment $h\in t^2_1*\Delta_U$ 
such that $\gamma(h)=c_2$.  We note that the 7 colors $c_0, \ldots, c_4, \gamma(e_1), \gamma(e_2)$ are pairwise distinct. Clearly, either 
$l(e_2)u_j$ or $r(e_2) u_k$ provides the $8$th color $c_8$.

\vskip 0.1cm
\noindent $-$ Suppose that $\gamma(r(e_2)u_k)=c_8$. Then $\gamma(l(e_2)u_j)\circeq \gamma(e_2)$,
 $\gamma(t_1^mu_j)\circeq \gamma(e_1)$ and $\gamma(t_1^2v_q)\circeq c_3$. Since $t_1^m$ cannot be an apex, 
 there is $h_1\in t^2_1*\Delta_U$ such that $\gamma(h_1)=\gamma(e_1)$. The existence of $h_1$ and $h$ imply that 
 $\gamma(t_1^1v_p)\circeq c_1$ and so $\gamma(u_it_1^1)\circeq c_1$. Then either $t_1^2v_p$ or $t_1^3v_q$ provides the $9$th color.

\vskip 0.1cm
\noindent $-$ Suppose that $\gamma(l(e_2)u_j)=c_8$. Then $\gamma(r(e_2)u_k)\circeq \gamma(e_2)$.
Since if $\gamma(t_1^mu_j)=c_8$, the existence of $h$ implies that $\gamma(t_1^{4-m}l(e_2))$ is the $9$th color. Then 
$\gamma(t_1^mu_j)\circeq \gamma(e_1)$, and as in previous paragraph, we can deduce $\gamma(t^2_1v_q)\circeq c_3$ and that there is $h_1\in t_1^2*\Delta_U$ such that $\gamma(h_1)=\gamma(e_1)$. From the existence of $h$ and $h_1$ it follows that $\gamma(t_1^1v_p)\circeq c_1$ and $\gamma(t_1^1u_i)\circeq c_1$. Then either 
$t_1^2v_p$ or $t_1^3v_q$ provides the $9$th color.
\end{proof} 

%%%%%%%%%%%%%%%%%%%%%%%%%%%%%%%%%%%%%%%%%%%%%%%%%%%%%%%%%%%%%%%%%%%%%%%%
%%%%%%%%%%%%%%%%%%%%%%%%%%%%%%%%%%%%%%%%%%%%%%%%%%%%%%%%%%%%%%%%%%%%%%%%
%%%%%%%%%%%%%%%%%%%%%%%%%%%%%%%%%%%%%%%%%%%%%%%%%%%%%%%%%%%%%%%%%%%%%%%%
%%%%%%%%%%%%%%%%%%%%%%%%%%%%%%%%%%%%%%%%%%%%%%%%%%%%%%%%%%%%%%%%%%%%%%%%
%%%%%%%%%%%%%%%%%%%%%%%%%%%%%%%%%%%%%%%%%%%%%%%%%%%%%%%%%%%%%%%%%%%%%%%%
%%%%%%%%%%%%%%%%%%%%%%%%%%%%%%%%%%%%%%%%%%%%%%%%%%%%%%%%%%%%%%%%%%%%%%%%

\section{The proof of Teorem~\ref{thm:main}}\label{s:proof}

We note that $d(2)=1$ follows trivially.  On the other hand, by the Szekeres and Peters result~\cite{6gons}, we know that any point set $P$ in the plane in general position with 
$|P|\geq 17$ contains a subset $Q$ such that 
$Q\sim C_6$. Starting with a coloring $\beta'$ of $D(Q)$ with 3 colors, we proceed as in the proof of Proposition~\ref{p:n-2} in order to extend $\beta'$ to a  coloring 
$\beta$ of $D(P)$ by adding a new star $S_i$ of color $c_i$ with apex $p_i$ for each $p_i\in P\setminus Q$. Since $\beta$ has exactly $|P|-3$ colors, we have that
$d(n)\leq n-3$ for any integer $n\geq 17$. 

Let $\gamma$ be an optimal coloring of $D(X)$. By Lov\'asz's theorem and Proposition~\ref{p:closed}($i$), in order to show Theorem~\ref{thm:main} it suffices to show that 
$|\gamma(X)|\geq 14$. We analyze separately several cases, depending on the sizes of $\gamma(A)$ and $\gamma(B)$.
By Corollary~\ref{c:3,4} we know that $3\leq |\gamma(A)|, |\gamma(B)|\leq 4$.

\vskip 0.1cm
{\sc Case 1}.  $|\gamma(A)|=3$ and $|\gamma(B)|=3$. It follows from Proposition~\ref{p:gamma(A)=3} that there are $a_i, a_j, a_k\in A$ 
(respectively, $b_{p}, b_{q}, b_{r}\in B$) 
with $i< j< k$ (respectively, $p< q < r$) such that none of them is an apex of $\gamma(A)$ (respectively, $\gamma(B)$). Then,
$c_1=\gamma(a_ib_{p}), c_2=\gamma(a_jb_{q})$ and $c_3=\gamma(a_kb_{r})$ are pairwise distinct. Let $ab$ be the segment in 
$\{a_ib_{p}, a_jb_{q}, a_kb_{r}\}$ that is closest to $t^1_1$. By the choice of  $ab$ we have that $\gamma(A \cup B\setminus \{a,b\})$
is disjoint from $\gamma (T_1\cup \{a, b\}\cup T_2)$. On the other hand, we note that $|\gamma(A \cup B\setminus \{a,b\})|\geq |\gamma(A)|+|\gamma(B)|+2=8$. 
The required inequality follows by applying Corollary~\ref{c:AB>=8} to $T_1\cup \{a,b\}\cup T_2$. 
 
\vskip 0.1cm
{\sc Case 2}. Suppose that $|\gamma(B)|=3$ and $|\gamma(A)|=4$. By Proposition~\ref{p:gamma(A)=3} there are  
$b_{p}, b_{q}, b_{r}\in B$ with $p< q < r$ such that none of them is an apex of $\gamma(B)$. Let $\Delta_B:=\Delta(b_{p}, b_{q}, b_{r})$, 
and let $b_{l_1}, b_{l_2}$ be the two points in $B\setminus \Delta_B$. We now modify slightly $\gamma$ as follows: 
recolor the three segments of $\Delta_B$ with color $c_1$ and use color $c_{l_k}$ to color each segment of $\xx$ incident with $b_{l_k}$ for $k=1,2$. 
It is not hard to see that this modification keep $\gamma$ proper and with same number of colors.

\vskip 0.1cm
\noindent {(2.1)} Suppose that $\gamma^{\star}(A)\leq 4$. Then there is $a\in A$ that is not an apex of $\gamma(A)$.
Let us denote by $Q$ to $T_1\cup \{a, b_{p}, b_{q}, b_{r}\} \cup T_2$. From Proposition~\ref{p:closed}~($ii$) and the
choice of $a, b_{l_1}, b_{l_2}$ we can deduce that $|\gamma(X)|\geq |\gamma(Q)|+6$. By applying Lemma~\ref{l:1T3} to 
$Q$ we know that $|\gamma(Q)|\geq 8$, as required.

\vskip 0.1cm
\noindent {(2.2)} Suppose that $\gamma^{\star}(A)=5$. By Proposition~\ref{p:5apices}, there are $a_i, a_j\in A$ 
with $i<j$ such that $\{a_ia_j\}$ is a $2-$star of $\gamma(A)$ and none of $a_i, a_j$ is an apex of any other star of $\gamma(A)$. 
As before, let use denote by $Q$ to $T_1\cup \{a_i, a_j, b_{p}, b_{q}, b_{r}\} \cup T_2$. Again, from Proposition~\ref{p:closed}~($ii$) and
 the choice of $a_i, a_j, b_{l_1}, b_{l_2}$ we can deduce that $|\gamma(X)|\geq |\gamma(Q)|+5$. By applying Lemma~\ref{l:2T3} to $Q$
  we know that $|\gamma(Q)|\geq 9$, as required.

\vskip 0.1cm
{\sc Case 3}. Suppose that $|\gamma(A)|=3$ and $|\gamma(B)|=4$. This case can be handled in the same manner as Case 2 (just interchange the roles of $A$ and $B$). 

\vskip 0.1cm
{\sc Case 4}. Suppose that $|\gamma(A)|=4$ and $|\gamma(B)|=4$. 
\vskip 0.1cm
\noindent {(4.1)} Suppose that $\gamma^{\star}(I)\leq 4$ for some $I\in \{A,B\}$. We only analyze the case $I=A$ because the case $I=B$ can be handled analogously. Then there is 
$a\in A$ such that $a$ is not an apex of $\gamma(A)$.

\vskip 0.1cm
\noindent {$\bullet$} Suppose that $\gamma^{\star}(B)\leq 4$. Then there is $b\in B$ such that $b$ is not an apex of $\gamma(B)$. 
Let $Q=T_1\cup \{a,b\}\cup T_2$. From Proposition~\ref{p:closed}~($ii$) and the choice of $a$ and $b$ we can deduce that
$|\gamma(X)|\geq |\gamma(Q)|+|\gamma(A)|+|\gamma(B)|=|\gamma(Q)|+8$. By applying Lemma~\ref{l:/TT} to $Q$ we obtain that $|\gamma(Q)|\geq 6$, as required.

%%%%%%%%%%%%%%%%%%%%%%%%%%%%%%%%%%%%%%%%%%%%%%%%%%%%%%%%%%%%%%%%%%%%%%%%
%%%%%%%%%%%%%%%%%%%%%%%%%%%%%%%%%%%%%%%%%%%%%%%%%%%%%%%%%%%%%%%%%%%%%%%%

\vskip 0.1cm
\noindent {$\bullet$} Suppose that $\gamma^{\star}(B)=5$. By Proposition~\ref{p:5apices}, there are $b_{p}, b_{q}\in B$ 
with $p<q$ such that $e=b_{p}b_{q}$ is a $2-$star of $\gamma(B)$ and neither $b_{p}$ nor $b_{q}$ is an apex of any other star of $\gamma(B)$. 
Let $Q=T_1\cup \{b_{p}, b_{q}, a\}\cup T_2, A'=A\setminus \{a\}$ and $B'=B\setminus \{b_{p}, b_{q}\}$. 
From Proposition~\ref{p:closed}~($ii$) and the choice of $b_{p}, b_{q}$ and $a$ we can deduce that
$|\gamma(X)|\geq |\gamma(Q)|+|\gamma(A')|+|\gamma(B')|=|\gamma(Q)|+7$. By applying Lemma~\ref{l:TDT} to $Q$ we obtain that $|\gamma(Q)|\geq 7$, as required.

\vskip 0.1cm
\noindent {(4.2)} Suppose that $\gamma^{\star}(A)=5=\gamma^{\star}(B)$. By Proposition~\ref{p:5apices}, there are $a_i, a_j\in A$ (resp. $b_{p}, b_{q}\in B$)
with $i<j$ (resp. $p<q$) such that $\{a_ia_j\}$  (resp.  $\{b_{p}b_{q}\}$) is a $2-$star of $\gamma(A)$ (resp. $\gamma(B)$) and none of $a_i, a_j$ 
(resp. $b_{p}, b_{q}$) is an apex of any other star of $\gamma(A)$ (resp. $\gamma(B)$). 
Let $Q=T_1\cup \{a_i, a_j, b_{p}, b_{q}\}\cup T_2$. From Proposition~\ref{p:closed}~($ii$) and the choice of 
$a_i, a_j, b_{p}, b_{q}$ we can deduce that
$|\gamma(X)|\geq |\gamma(Q)|+6$. By applying Lemma~\ref{l:2T2} to $Q$ we obtain that $|\gamma(Q)|\geq 8$, as required.

%\section{Concluding remarks}

%We recall that Jonsson \cite{jonsson} and independently Fabila-Monroy et al. \cite{ruy} 
%prove that if $P$ is in convex position, then $\chi(D(P)) = f(n)$. We use it to prove the following:

%\begin{thm}\label{t:n-f(k)}
%Let $k\geq 3$ be an integer and $n\geq k$. If $P$ has a $k$--convex--gon, then 
%$$\chi(D(P)) \leq n-k + f(k).$$
%\end{thm}

%\begin{proof}
%Let $K$ be the set of vertices of a $k$--convex--gon in $P$ and $\gamma$ a coloration of $D(P)$ obtained by  
%coloring $D(K)$ with $f(k)$ colors, and 
%every element in $P\setminus K$ being an apex star. We will have that $|\gamma(D(P))| = f(k) + n-k$. 
%\end{proof}

%G. Szekeres et al. prove that, if $n \geq 17$, in any $n$ point set $P$ we can find 
%a $6$--convex--gon \cite{6gons}. So if we apply Theorem \ref{t:n-f(k)}, we get 
%$\chi(D(P)) \leq n - 6 + f(6) = n - 3$.

%Because of the way we color $D(P)$ on the proof of Theorem \ref{t:n-f(k)}, we strongly believe that the point sets in convex 
%position is the configuration that, its disjointness graph, requieres the minimum number of colors,
%{\em i.e.} we strongly believe that following conjecure is true:

%\begin{conjecture}
%For any $n$ point set $P$, 
%$$\chi(D(P)) \geq f(n).$$
%\end{conjecture}


\begin{thebibliography}{99}

\bibitem{magistral} B.~M.~\'Abrego, M.~Cetina, S.~Fern\'andez-Merchant, 
  J.~Lea\~nos and G.~Salazar,  On $\le k$--edges, crossings, and
  halving lines of geometric drawings of $K_n$. {\em Discrete
    Comput.~Geom.} {\bf 48} (2012), no. 1, 192--215.
        
\bibitem{lara} O.~Aichholzer, G.~Araujo-Pardo, N.~Garc\'ia-Col\'in, T.~Hackl, D.~Lara, C.~Rubio-Montiel and  J.~Urrutia. Geometric Achromatic and Pseudoachromatic Indices. 
{\em Graphs and Combinatorics} 32, 431-451 (2016). 
        
\bibitem{birgit} O.~Aichholzer, J.~Kyn\v{c}l, M.~Scheucher, and B.~Vogtenhuber. On $4$--Crossing-Families in Point Sets
and an Asymptotic Upper Bound. In {\em Proceedings of the 37th European Workshop on Computational Geometry (EuroCG 2021)}
 pp. 1--8 (2021).        
            
\bibitem{gaby} G. Araujo, A. Dumitrescu, F. Hurtado, M. Noy, and J. Urrutia. On
the chromatic number of some geometric type Kneser graphs. {\em Comput.
Geom.}, 32(1):59--69, 2005.

\bibitem{barany} I. B\'ar\'any. A short proof of Kneser's conjecture. J. Combin. Theory Ser. A 1978, 25, 325-326.

\bibitem{dujmovic} V. Dujmovi\'c and D. R. Wood, Thickness and antithickness of graphs. Journal of Computational Geometry.   
volume 9, 365-386, 2018.

\bibitem{ErdosSzekeres} P. Erd\H{o}s, G. Szekeres. A combinatorial problem in geometry. {\it Comp. Math.} 2, 463-470 (1935).

\bibitem{mario} R. Fabila-Monroy, C. Hidalgo-Toscano, J Lea\~nos, and M. Lomel\'{\i}-Haro. 
  {The Chromatic Number of the Disjointness Graph of the Double Chain.}
  {\it Discrete Mathematics \& Theoretical Computer Science}.
  Vol. 22 no. 1, 2020.

\bibitem{Cn} R. Fabila-Monroy, J. Jonsson, P. Valtr and D. R. Wood, The exact chromatic number of the convex segment disjointness graph, 	
arXiv:1804.01057 (2018).

\bibitem{us1} J.~Lea\~nos, C.~Ndjatchi, and L.~R.~R\'ios-Castro, On the connectivity of the disjointness graph of segments of point sets in general position in the plane.
{\it Discrete Mathematics and Theoretical Computer Science (DMTCS)}, vol. $24 :1, 2022, \#15, 1-16$

\bibitem{kneser} M. Kneser. {\it Jahresbericht der Deutschen Mathematiker-Vereinigung}, vol.
58, page 27. 1956. Aufgabe 360.

%%%%%%%%%%%%%%%%%%%%%%%%%%%%%%%%%%%%%%%%%%

%\bibitem{}M. Kneser. Aufgabe 360. Jahresbericht der Deutschen Mathematiker-Vereinigung, 58:27, 1956.

\bibitem{lovasz} Lov\'asz. Kneser's conjecture, chromatic number, and homotopy. J. Combin. Theory Ser. A, 25(3):319-324, 1978.

\bibitem{pach-tardos-toth} J. Pach, G. Tardos,  and G. T\'oth, {\it Disjointness graphs of segments}, in: Boris Aronov and Matthew J.
Katz (eds.), {\it 33th International Symposium on Computational Geometry (SoCG 2017)} , vol. 77 of Leibniz
International Proceedings in Informatics (LIPIcs), 59:1--15, Leibniz-Zentrum f\"{u}r Informatik, Dagstuhl, 2017.

\bibitem{pach-tomon} J. Pach and I. Tomon. On the Chromatic Number of Disjointness Graphs of Curves.
{\it In 35th International Symposium on Computational Geometry (SoCG 2019)}. Schloss Dagstuhl-Leibniz-Zentrum fuer Informatik, 2019.

\bibitem{6gons} G. Szekeres and L. Peters, Computer solution to the 17-points Erd\H{o}-Szekeres problem. {\em The ANZIAM Journal}. 48 (2), 151-164 (2006)

\end{thebibliography}
 \end{document}